\documentclass[10pt]{amsart}

\usepackage[margin=1in,headheight=13.6pt]{geometry}

\usepackage{textcmds} 
\usepackage{amsmath}
\usepackage{amsxtra}
\usepackage{amscd}
\usepackage{amsthm}
\usepackage{amsfonts}
\usepackage{amssymb}
\usepackage{eucal}
\usepackage[all]{xy}
\usepackage{graphicx}
\usepackage{comment}
\usepackage{epsfig}
\usepackage{psfrag}
\usepackage{mathrsfs}
\usepackage{amscd}
\usepackage{rotating}
\usepackage{lscape}
\usepackage{amsbsy}
\usepackage{verbatim}

\usepackage[hidelinks]{hyperref}

\usepackage{moreverb}
\usepackage{url}
\usepackage{extarrows}
\usepackage{setspace}

\usepackage{cancel}
\usepackage{bbm}

\usepackage{tikz-cd}
\usepackage{tikz}
\usetikzlibrary{matrix,arrows,decorations.pathmorphing}

\DeclareMathOperator*{\colim}{colim}
\DeclareMathOperator*{\utimes}{\times}

\DeclareMathAlphabet{\mathpzc}{OT1}{pzc}{m}{it}

\newtheorem{cor}[subsubsection]{Corollary}
\newtheorem{lem}[subsubsection]{Lemma}
\newtheorem{prop}[subsubsection]{Proposition}

\newtheorem{conj}[subsubsection]{Conjecture}

\newtheorem{thm}[subsubsection]{Theorem}

\newtheorem{mainthm}{Theorem}

\newenvironment{mainthmbis}[1]
  {\renewcommand{\themainthm}{\ref{#1}$'$}%
   \addtocounter{mainthm}{-1}%
   \begin{mainthm}}
  {\end{mainthm}}

\newtheorem{defn}[subsubsection]{Definition}

\theoremstyle{remark}
\newtheorem{rem}[subsubsection]{Remark}
\newtheorem{example}[subsubsection]{Example}

\theoremstyle{remark}

\numberwithin{equation}{section}

\newcommand{\nc}{\newcommand}
\nc{\renc}{\renewcommand}
\nc{\ssec}{\subsection}
\nc{\sssec}{\subsubsection}
\nc{\on}{\operatorname}

\nc\ol{\overline}
\nc\wt{\widetilde}
\nc\tboxtimes{\wt{\boxtimes}}
\nc\tstar{\wt{\star}}
\nc{\alp}{\alpha}

\nc{\ZZ}{{\mathbb Z}}
\nc{\NN}{{\mathbb N}}
\nc{\OO}{{\mathbb O}}
\renc{\SS}{{\mathbb S}}
\nc{\DD}{{\mathbb D}}
\nc{\GG}{{\mathbb G}}
\renewcommand{\AA}{{\mathbb A}}
\nc{\Fq}{{\mathbb F}_q}
\nc{\Fqb}{\ol{{\mathbb F}_q}}
\nc{\Ql}{\ol{{\mathbb Q}_\ell}}
\nc{\id}{\mathsf{id}}
\nc\X{\mathcal X}

\nc{\Hom}{\on{Hom}}
\nc{\Lie}{\on{Lie}}
\nc{\Loc}{\on{Loc}}
\nc{\Pic}{\on{Pic}}
\nc{\Bun}{\on{Bun}}
\nc{\IC}{\on{IC}}
\nc{\Aut}{\on{Aut}}
\nc{\rk}{\on{rk}}
\nc{\Sh}{\on{Sh}}
\nc{\Perv}{\on{Perv}}
\nc{\pos}{{\on{pos}}}
\nc{\Conv}{\on{Conv}}
\nc{\Sph}{\on{Sph}}
\nc{\Sym}{\on{Sym}}
\nc{\BunBb}{\overline{\Bun}_B}
\nc{\BunNb}{\overline{\Bun}_N}
\nc{\BunTb}{\overline{\Bun}_T}
\nc{\BunBbm}{\overline{\Bun}_{B^-}}
\nc{\BunBbel}{\overline{\Bun}_{B,el}}
\nc{\BunBbmel}{\overline{\Bun}_{B^-,el}}
\nc{\Buno}{\overset{o}{\Bun}}
\nc{\BunPb}{{\overline{\Bun}_P}}
\nc{\BunBM}{\Bun_{B(M)}}
\nc{\BunBMb}{\overline{\Bun}_{B(M)}}
\nc{\BunPbw}{{\widetilde{\Bun}_P}}
\nc{\BunBP}{\widetilde{\Bun}_{B,P}}
\nc{\GUb}{\overline{G/U}}
\nc{\GUPb}{\overline{G/U(P)}}

\nc\syminfty{\on{Sym}^{\infty}}
\nc\lal{\ol{\lambda}}
\nc\xl{\ol{x}}
\nc\thl{\ol{\theta}}
\nc\nul{\ol{\nu}}
\nc\mul{\ol{\mu}}
\nc\Sum\Sigma
\nc{\oX}{\overset{\circ}{X}{}}
\nc{\hl}{\overset{\leftarrow}h{}}
\nc{\hr}{\overset{\rightarrow}h{}}
\nc{\M}{{\mathcal M}}
\nc{\N}{{\mathcal N}}
\nc{\F}{{\mathcal F}}
\nc{\D}{{\mathcal D}}
\nc{\Y}{{\mathcal Y}}
\nc{\G}{{\mathcal G}}
\nc{\E}{{\mathcal E}}
\nc{\CalC}{{\mathcal C}}
\nc\Dh{\widehat{\D}}
\renewcommand{\O}{{\mathcal O}}
\nc{\K}{{\mathcal K}}
\renewcommand{\S}{{\mathcal S}}
\nc{\T}{{\mathcal T}}
\nc{\V}{{\mathcal V}}
\renc{\P}{{\mathcal P}}
\nc{\A}{{\mathcal A}}
\nc{\U}{{\mathcal U}}

\nc{\frn}{{\check{\mathfrak u}(P)}}

\nc{\fC}{\mathfrak C}
\nc\f{{\mathfrak f}}

\nc{\qo}{{\mathfrak q}}
\nc{\po}{{\mathfrak p}}
\nc{\s}{{\mathfrak s}}
\nc\w{\text{w}}
\renewcommand{\r}{{\mathfrak r}}

\renewcommand{\mod}{{\on{-}\mathsf{mod}}}

\nc\mathi\iota
\nc\Spec{\on{Spec}}
\nc\Mod{\on{Mod}}
\nc{\tw}{\widetilde{\mathfrak t}}
\nc{\pw}{\widetilde{\mathfrak p}}
\nc{\qw}{\widetilde{\mathfrak q}}
\nc{\jw}{\widetilde j}

\nc{\grb}{\overline{\Gr_{X^{\fset}}}}
\nc{\I}{\mathcal I}
\renewcommand{\i}{\mathfrak i}
\renewcommand{\j}{\mathfrak j}

\nc{\lambdach}{{\check\lambda}}
\nc{\Lambdach}{{\check\Lambda}{}}
\nc{\much}{{\check\mu}}
\nc{\omegach}{{\check\omega}}
\nc{\nuch}{{\check\nu}}
\nc{\etach}{{\check\eta}}
\nc{\alphach}{{\check\alpha}}

\nc{\rhoch}{{\check\rho}}

\nc{\Hb}{\overline{\H}}

\nc{\BA}{{\mathbb{A}}}
\nc{\BB}{\mathbb{B}}
\nc{\BC}{{\mathbb{C}}}
\nc{\BD}{{\mathbb{D}}}
\nc{\BE}{{\mathbb{E}}}
\nc{\BF}{{\mathbb{F}}}
\nc{\BG}{{\mathbb{G}}}
\nc{\BH}{{\mathbb{H}}}
\nc{\BI}{{\mathbb{I}}}
\nc{\BM}{{\mathbb{M}}}
\nc{\BN}{{\mathbb{N}}}
\nc{\BO}{{\mathbb{O}}}
\nc{\BP}{{\mathbb{P}}}
\nc{\BQ}{{\mathbb{Q}}}
\nc{\BR}{{\mathbb{R}}}
\nc{\BS}{{\mathbb{S}}}
\nc{\BT}{{\mathbb{T}}}
\nc{\BV}{{\mathbb{V}}}
\nc{\BZ}{{\mathbb{Z}}}

\nc{\bbone}{\mathbbm{1}}
\nc{\bbA}{{\mathbb{A}}}
\nc{\bbB}{\mathbb{B}}
\nc{\bbC}{{\mathbb{C}}}
\nc{\bbD}{{\mathbb{D}}}
\nc{\bbE}{{\mathbb{E}}}
\nc{\bbF}{{\mathbb{F}}}
\nc{\bbG}{{\mathbb{G}}}
\nc{\bbH}{{\mathbb{H}}}
\nc{\bbI}{{\mathbb{I}}}
\nc{\bbK}{{\mathbb{K}}}
\nc{\bbL}{{\mathbb{L}}}
\nc{\bbM}{{\mathbb{M}}}
\nc{\bbN}{{\mathbb{N}}}
\nc{\bbO}{{\mathbb{O}}}
\nc{\bbP}{{\mathbb{P}}}
\nc{\bbQ}{{\mathbb{Q}}}
\nc{\bbR}{{\mathbb{R}}}
\nc{\bbS}{{\mathbb{S}}}
\nc{\bbT}{{\mathbb{T}}}
\nc{\bbU}{{\mathbb{U}}}
\nc{\bbV}{{\mathbb{V}}}
\nc{\bbW}{{\mathbb{W}}}
\nc{\bbX}{{\mathbb{X}}}
\nc{\bbY}{{\mathbb{Y}}}
\nc{\bbZ}{{\mathbb{Z}}}

\nc{\CA}{{\mathcal{A}}}
\nc{\CB}{{\mathcal{B}}}

\nc{\CE}{{\mathcal{E}}}
\nc{\CF}{{\mathcal{F}}}
\nc{\CH}{{\mathcal{H}}}

\nc{\CL}{{\mathcal{L}}}
\nc{\CC}{{\mathcal{C}}}
\nc{\CG}{{\mathcal{G}}}
\nc{\CM}{{\mathcal{M}}}
\nc{\CN}{{\mathcal{N}}}
\nc{\CK}{{\mathcal{K}}}
\nc{\CO}{{\mathcal{O}}}
\nc{\CP}{{\mathcal{P}}}
\nc{\CQ}{{\mathcal{Q}}}
\nc{\CR}{{\mathcal{R}}}
\nc{\CS}{{\mathcal{S}}}
\nc{\CU}{{\mathcal{U}}}
\nc{\CV}{{\mathcal{V}}}
\nc{\CW}{{\mathcal{W}}}
\nc{\CX}{{\mathcal{X}}}
\nc{\CY}{{\mathcal{Y}}}
\nc{\CZ}{{\mathcal{Z}}}
\nc{\CI}{{\mathcal{I}}}

\nc{\csM}{{\check{\mathcal A}}{}}
\nc{\oM}{{\overset{\circ}{\mathcal M}}{}}
\nc{\obM}{{\overset{\circ}{\mathbf M}}{}}
\nc{\oCA}{{\overset{\circ}{\mathcal A}}{}}
\nc{\obA}{{\overset{\circ}{\mathbf A}}{}}
\nc{\ooM}{{\overset{\circ}{M}}{}}
\nc{\osM}{{\overset{\circ}{\mathsf M}}{}}
\nc{\vM}{{\overset{\bullet}{\mathcal M}}{}}
\nc{\nM}{{\underset{\bullet}{\mathcal M}}{}}
\nc{\oD}{{\overset{\circ}{\mathcal D}}{}}
\nc{\obC}{{\overset{\circ}{\mathbf C}}{}}
\nc{\obD}{{\overset{\circ}{\mathbf D}}{}}
\nc{\oA}{{\overset{\circ}{\mathbb A}}{}}
\nc{\op}{{\overset{\bullet}{\mathbf p}}{}}
\nc{\oU}{{\overset{\bullet}{\mathcal U}}{}}
\nc{\oZ}{{\overset{\circ}{\mathcal Z}}{}}
\nc{\ofZ}{{\overset{\circ}{\mathfrak Z}}{}}
\nc{\oF}{{\overset{\circ}{\fF}}}

\nc{\fa}{{\mathfrak{a}}}
\nc{\fb}{{\mathfrak{b}}}
\nc{\fc}{{\mathfrak{c}}}
\nc{\fd}{{\mathfrak{d}}}
\nc{\ff}{{\mathfrak{f}}}
\nc{\fg}{{\mathfrak{g}}}
\nc{\fgl}{{\mathfrak{gl}}}
\nc{\fh}{{\mathfrak{h}}}
\nc{\fj}{{\mathfrak{j}}}
\nc{\fl}{{\mathfrak{l}}}
\nc{\fm}{{\mathfrak{m}}}
\nc{\fn}{{\mathfrak{n}}}
\nc{\fu}{{\mathfrak{u}}}
\nc{\fp}{{\mathfrak{p}}}
\nc{\fr}{{\mathfrak{r}}}
\nc{\fs}{{\mathfrak{s}}}
\nc{\ft}{{\mathfrak{t}}}
\nc{\fz}{{\mathfrak{z}}}
\nc{\fsl}{{\mathfrak{sl}}}
\nc{\hsl}{{\widehat{\mathfrak{sl}}}}
\nc{\hgl}{{\widehat{\mathfrak{gl}}}}
\nc{\hg}{{\widehat{\mathfrak{g}}}}
\nc{\chg}{{\widehat{\mathfrak{g}}}{}^\vee}
\nc{\hn}{{\widehat{\mathfrak{n}}}}
\nc{\chn}{{\widehat{\mathfrak{n}}}{}^\vee}

\nc{\fA}{{\mathfrak{A}}}
\nc{\fB}{{\mathfrak{B}}}
\nc{\fD}{{\mathfrak{D}}}
\nc{\fE}{{\mathfrak{E}}}
\nc{\fF}{{\mathfrak{F}}}
\nc{\fG}{{\mathfrak{G}}}
\nc{\fK}{{\mathfrak{K}}}
\nc{\fL}{{\mathfrak{L}}}
\nc{\fM}{{\mathfrak{M}}}
\nc{\fN}{{\mathfrak{N}}}
\nc{\fP}{{\mathfrak{P}}}
\nc{\fU}{{\mathfrak{U}}}
\nc{\fV}{{\mathfrak{V}}}
\nc{\fX}{{\mathfrak{X}}}
\nc{\fY}{{\mathfrak{Y}}}
\nc{\fZ}{{\mathfrak{Z}}}

\nc{\bb}{{\mathbf{b}}}
\nc{\bc}{{\mathbf{c}}}
\nc{\bd}{{\mathbf{d}}}
\nc{\bbf}{{\mathbf{f}}}
\nc{\be}{{\mathbf{e}}}
\nc{\bg}{{\mathbf{g}}}
\nc{\bi}{{\mathbf{i}}}
\nc{\bj}{{\mathbf{j}}}
\nc{\bn}{{\mathbf{n}}}
\nc{\bo}{{\mathbf{o}}}
\nc{\bp}{{\mathbf{p}}}
\nc{\bq}{{\mathbf{q}}}
\nc{\bt}{{\mathbf{t}}}
\nc{\bu}{{\mathbf{u}}}
\nc{\bv}{{\mathbf{v}}}
\nc{\bx}{{\mathbf{x}}}
\nc{\bs}{{\mathbf{s}}}
\nc{\by}{{\mathbf{y}}}
\nc{\bw}{{\mathbf{w}}}
\nc{\bA}{{\mathbf{A}}}
\nc{\bK}{{\mathbf{K}}}
\nc{\bB}{{\mathbf{B}}}
\nc{\bC}{{\mathbf{C}}}
\nc{\bG}{{\mathbf{G}}}
\nc{\bD}{{\mathbf{D}}}
\nc{\bH}{{\mathbf{H}}}
\nc{\bM}{{\mathbf{M}}}
\nc{\bN}{{\mathbf{N}}}
\nc{\bO}{{\mathbf{O}}}
\nc{\bT}{{\mathbf{T}}}
\nc{\bV}{{\mathbf{V}}}
\nc{\bW}{{\mathbf{W}}}
\nc{\bX}{{\mathbf{X}}}
\nc{\bZ}{{\mathbf{Z}}}
\nc{\bS}{{\mathbf{S}}}

\nc{\sA}{{\mathsf{A}}}
\nc{\sB}{{\mathsf{B}}}
\nc{\sC}{{\mathsf{C}}}
\nc{\sD}{{\mathsf{D}}}
\nc{\sF}{{\mathsf{F}}}
\nc{\sG}{{\mathsf{G}}}
\nc{\sK}{{\mathsf{K}}}
\nc{\sM}{{\mathsf{M}}}
\nc{\sO}{{\mathsf{O}}}
\nc{\sW}{{\mathsf{W}}}
\nc{\sQ}{{\mathsf{Q}}}
\nc{\sP}{{\mathsf{P}}}
\nc{\sV}{{\mathsf{V}}}
\nc{\sS}{{\mathsf{S}}}
\nc{\sT}{{\mathsf{T}}}
\nc{\sZ}{{\mathsf{Z}}}
\nc{\sfp}{{\mathsf{p}}}
\nc{\sll}{{\mathsf{l}}}
\nc{\sr}{{\mathsf{r}}}
\nc{\bk}{{\mathsf{k}}}
\nc{\sg}{{\mathsf{g}}}

\nc{\sff}{{\mathsf{f}}}
\nc{\sfb}{{\mathsf{b}}}
\nc{\sfc}{{\mathsf{c}}}
\nc{\sd}{{\mathsf{d}}}
\nc{\se}{{\mathsf{e}}}

\nc{\tA}{{\widetilde{\mathbf{A}}}}
\nc{\tB}{{\widetilde{\mathcal{B}}}}
\nc{\tg}{{\widetilde{\mathfrak{g}}}}
\nc{\tG}{{\widetilde{G}}}
\nc{\TM}{{\widetilde{\mathbb{M}}}{}}
\nc{\tO}{{\widetilde{\mathsf{O}}}{}}
\nc{\tU}{{\widetilde{\mathfrak{U}}}{}}
\nc{\TZ}{{\tilde{Z}}}
\nc{\tx}{{\tilde{x}}}
\nc{\tbv}{{\tilde{\bv}}}
\nc{\tfP}{{\widetilde{\mathfrak{P}}}{}}
\nc{\tz}{{\tilde{\zeta}}}
\nc{\tmu}{{\tilde{\mu}}}

\nc{\urho}{\underline{\rho}}
\nc{\uB}{\underline{B}}
\nc{\uC}{{\underline{\mathbb{C}}}}
\nc{\ui}{\underline{i}}
\nc{\uj}{\underline{j}}
\nc{\ofP}{{\overline{\mathfrak{P}}}}
\nc{\oB}{{\overline{\mathcal{B}}}}
\nc{\og}{{\overline{\mathfrak{g}}}}
\nc{\oI}{{\overline{I}}}

\nc{\eps}{\varepsilon}
\nc{\hrho}{{\hat{\rho}}}

\nc{\one}{{\mathbf{1}}}
\nc{\two}{{\mathbf{t}}}

\nc{\Rep}{{\mathop{\operatorname{\rm Rep}}}}
\nc{\Tot}{{\mathop{\operatorname{\rm Tot}}}}
\nc{\Ker}{{\mathop{\operatorname{\rm Ker}}}}
\nc{\Hilb}{{\mathop{\operatorname{\rm Hilb}}}}
\nc{\Ext}{{\mathop{\operatorname{\rm Ext}}}}
\nc{\CHom}{{\mathop{\operatorname{{\mathcal{H}}\it om}}}}
\nc{\GL}{{\mathop{\operatorname{\rm GL}}}}
\nc{\gr}{{\mathop{\operatorname{\rm gr}}}}
\nc{\Id}{{\mathop{\operatorname{\rm Id}}}}
\nc{\de}{{\mathop{\operatorname{\rm def}}}}
\nc{\length}{{\mathop{\operatorname{\rm length}}}}
\nc{\supp}{{\mathop{\operatorname{\rm supp}}}}

\nc{\Cliff}{{\mathsf{Cliff}}}
\nc{\Fl}{\on{Fl}}
\nc{\Fib}{{\mathsf{Fib}}}
\nc{\Coh}{{\on{Coh}}}
\nc{\coh}{{\on{coh}}}

\nc{\QCoh}{{\on{QCoh}}}
\nc{\IndCoh}{{\on{IndCoh}}}
\nc{\FCoh}{{\mathsf{FCoh}}}

\nc{\reg}{{\text{\rm reg}}}

\nc{\cplus}{{\mathbf{C}_+}}
\nc{\cminus}{{\mathbf{C}_-}}
\nc{\cthree}{{\mathbf{C}_*}}
\nc{\Qbar}{{\bar{Q}}}
\nc\Eis{\on{Eis}}
\nc\CT{\on{CT}}
\nc\Eisb{\ol\Eis{}}
\nc\Eisr{\on{Eis}^{rat}{}}
\nc\wh{\widehat}
\nc{\Def}{\on{Def_{\check{\fb}}(E)}}
\nc{\barZ}{\overline{Z}{}}
\nc{\barbarZ}{\overline{\barZ}{}}
\nc{\barpi}{\overline\pi}
\nc{\barbarpi}{\overline\barpi}
\nc{\barpip}{\overline\pi{}^+}
\nc{\barpim}{\overline\pi{}^-}

\nc{\fq}{\mathfrak q}

\nc{\fqb}{\ol{\fq}{}}
\nc{\fpb}{\ol{\fp}{}}
\nc{\fpr}{{\fp^{rat}}{}}
\nc{\fqr}{{\fq^{rat}}{}}

\nc{\hattimes}{\wh\otimes}

\nc{\bh}{{\bar{h}}}
\nc{\bOmega}{{\overline{\Omega(\check \fn)}}}

\nc{\seq}[1]{\stackrel{#1}{\sim}}

\nc{\cT}{{\check{T}}}
\nc{\cG}{{\check{G}}}
\nc{\cM}{{\check{M}}}
\nc{\cB}{{\check{B}}}
\nc{\cP}{{\check{P}}}

\nc{\ct}{{\check{\mathfrak t}}}
\nc{\cg}{{\check{\fg}}}
\nc{\cb}{{\check{\fb}}}
\nc{\cn}{{\check{\fn}}}
\nc{\cp}{{\check{\fp}}}
\nc{\cm}{{\check{\fm}}}

\nc{\cLambda}{{\check\Lambda}}

\nc{\cla}{{\check\lambda}}
\nc{\cmu}{{\check\mu}}
\nc{\cnu}{{\check\nu}}
\nc{\ceta}{{\check\eta}}

\nc{\DefbE}{{\on{Def}_{\cB}(E_\cT)}}

\nc{\imathb}{{\ol{\imath}}}
\nc{\rlr}{\overset{\longrightarrow}{\underset{\longrightarrow}\longleftarrow}}

\nc{\oBun}{\overset{\circ}\Bun}
\nc{\BunBbb}{\ol{\ol{Bun}}_B}
\nc{\BunBr}{\Bun_B^{rat}}
\nc{\BunBrsg}{\Bun_B^{rat,\on{s.g.}}}
\nc{\BunBrp}{\Bun_B^{rat,polar}}
\nc{\BunBrpbg}{\Bun_B^{rat,polar,\on{b.g.}}}
\nc{\BunBrpsg}{\Bun_B^{rat,polar,\on{s.g.}}}
\nc{\BunTrp}{\Bun_T^{rat,polar}}
\nc{\BunTrpbg}{\Bun_T^{rat,polar,\on{b.g.}}}
\nc{\BunTrpsg}{\Bun_T^{rat,polar,\on{s.g.}}}
\nc{\BunNr}{\Bun_N^{rat}}
\nc{\BunNre}{\Bun_N^{enh,rat}}
\nc{\BunTr}{\Bun_T^{rat}}
\nc{\Vect}{\on{Vect}}
\nc{\Whit}{\on{Whit}}
\nc{\bTb}{\ol{\on{CT}}}

\nc{\bTr}{\on{CT}^{rat}{}}
\nc\jmathr{\jmath^{rat}{}}
\nc{\ux}{\underline{x}}
\nc{\clambda}{{\check\lambda}}
\nc{\calpha}{{\check\alpha}}

\nc{\inftyGrpd}{{\mathsf{Grpd}_\infty}}

\nc{\fset}{\mathsf{fSet}}
\nc{\fSet}{\fset}
\nc{\LocSysG}{\LocSys_{\cG}}
\nc{\Sing}{{\on{Sing}}}
\nc{\dr}{{\on{dR}}}
\nc{\Ind}{\on{Ind}}
\nc{\Sat}{\on{Sat}}
\nc{\Ho}{\on{Ho}}
\nc{\Res}{\on{Res}}
\nc{\sotimes}{\overset{!}\otimes}
\nc{\mmod}{{\on{-}}{\mathbf{mod}}}
\nc{\Maps}{\on{Maps}}
\nc{\CMaps}{{\mathcal Maps}}
\nc{\bMaps}{{\mathbf{Maps}}}

\nc{\dgSch}{\on{DGSch}}
\nc{\dgindSch}{\on{DGindSch}}
\nc{\indSch}{\on{indSch}}
\nc{\Sch}{\mathsf{Sch}}
\nc{\affdgSch}{\on{DGSch}^{\on{aff}}}
\nc{\affSch}{\on{Sch}^{\on{aff}}}

\nc{\Groupoids}{\on{Grpd}}
\nc{\inftypic}{\infty\on{-PicGrpd}}
\nc{\inftyCat}{{\mathsf{Cat}_{\infty}}}
\nc{\MoninftyCat}{\infty\on{-Cat}^{Mon}}
\nc{\SymMoninftyCat}{\infty\on{-Cat}^{\on{SymMon}}}
\nc{\SymMonStinftyCat}{\on{DGCat}^{\on{SymMon}}}
\nc{\MonStinftyCat}{\on{DGCat}^{Mon}}
\nc{\inftystack}{\on{Stk}}
\nc{\inftystackalg}{Stk^{1\text{-}alg}}
\nc{\inftyprestack}{\on{PreStk}}
\nc{\inftydgnearstack}{\on{NearStk}}
\nc{\inftydgstack}{\on{Stk}}
\nc{\inftydgstackalg}{DGStk^{1\text{-}alg}}
\nc{\inftydgprestack}{\on{PreStk}}

\nc{\HC}{\CH\bC}

\nc{\csupp}{\supp}
\nc{\Arth}{\on{Arth}}
\nc{\ArthG}{{\on{Arth}_\cG}}
\nc{\ul}{\underline}

\renc{\neg}{\on{neg}}

\nc{\Congr}{\bK^1}
\nc{\CongrN}{\bN^1}
\nc{\CongrT}{\bT^1}
\nc{\CongrNneg}{(\bN^-)^1}

\nc{\Z}{\mathcal{Z}}

\nc{\calN}{\N}
\nc{\calW}{\mathcal{W}}
\nc{\calF}{\mathcal{F}}
\nc{\calH}{\mathcal{H}}
\nc{\calO}{\mathcal{O}}
\nc{\calK}{\mathcal{K}}

\nc{\Ran}{\mathsf{Ran}}
\nc{\Jets}{\on{Jets}}
\nc{\act}{\mathsf{act}}
\nc{\Av}{\mathsf{Av}}
\nc{\Ad}{\on{Ad}}
\nc{\BGRan}{BG_{\Ran}}
\nc{\codim}{\on{codim}}
\nc{\cpt}{{\on{cpt}}}
\nc{\dR}{{\on{dR}}}
\nc{\DGCat}{\mathsf{DGCat}}
\nc{\DGCatcont}{\on{DGCat}_{cont}}
\nc{\glob}{{\on{glob}}}
\nc{\loc}{{\on{loc}}}

\renewcommand{\op}{{\on{op}}}
\nc{\pt}{{\on{pt}}}
\nc{\PreStk}{{\mathsf{PreStk}}}
\nc{\Cat}{{\mathsf{Cat}}}

\nc{\ShvCat}{{\mathsf{ShvCat}}}

\nc{\restr}[2]{\left. #1 \right |_{#2}}
\nc{\uprestr}[2]{\left. #1 \right |^{#2}}

\nc{\bLoc}{{\mathbf{Loc}}}
\nc{\bGamma}{{\mathbf{\Gamma}}}
\nc{\bLocA}{\mathbf{Loc}^\A}
\nc{\bGammaA}{\mathbf{\Gamma}^\A}
\nc{\bLocB}{\mathbf{Loc}^\B}
\nc{\bGammaB}{\mathbf{\Gamma}^\B}
\nc{\bLocH}{\mathbf{Loc}^\H}
\nc{\bGammaH}{\mathbf{\Gamma}^\H}

\nc{\gen}{\mathit{gen}}
\nc{\ggen}{\!\on{-}\!\mathit{gen}}

\nc{\hto}{\hookrightarrow}

\nc{\ext}{\mathsf{ext}}

\nc{\ev}{\mathsf{ev}}
\nc{\ins}{\mathsf{ins}}
\nc{\rat}{\mathsf{rat}}

\nc{\usotimes}[1]{\underset{#1}{\otimes}}

\nc{\ch}{{\mathfrak{ch}}}

\renc{\fD}{{\Dmod}}
\nc{\fH}{{\mathfrak{H}}}

\nc{\p}{{\mathfrak{p}}}
\renc{\r}{{\mathfrak{r}}}

\nc{\xto}{\xrightarrow}

\renc{\sec}{\section}
\nc{\enh}{{\on{enh}}}

\renc{\gen}{\mathsf{gen}}
\nc{\BunGBgen}{\Bun_G^{B-\gen}}
\nc{\BunGHgen}{\Bun_G^{H-\gen}}
\nc{\BunGNgen}{\Bun_G^{N-\gen}}

\nc{\Fun}{\mathsf{Fun}}
\nc{\End}{\mathsf{End}}

\nc{\lr}{\xymatrix{ \ar@<-0.4ex>[r] \ar@<.5ex>[l]  & } }
\nc{\rr}{\xymatrix{ \ar@<-0.2ex>[r] \ar@<.7ex>[r]  & } }
\nc{\rrr}{\xymatrix{ \ar@<.0ex>[r] \ar@<.7ex>[r] \ar@<-0.7ex>[r] & } }

\nc{\Stab}{\mathsf{Stab}}
\nc{\Orb}{\mathsf{Orb}}

\renc{\exp}{\mathit{exp}}

\renc{\q}{\mathfrak{q}}

\nc{\virg}[1]{``#1"}

\nc{\QA}[2]
{\textbf{Question:} {#1} 
\\
\textbf{Answer:} {#2}}

\renc{\bold}[1]{\boldsymbol{#1}}

\nc{\bigt}[1]{\big( #1 \big) }
\nc{\Bigt}[1]{\Big( #1 \Big) }

\nc{\extwhit}{{\CW h}(G,\mathsf{ext})}

\nc{\footcite}{\footnote}

\nc{\GA}{{G(\AA)}}
\nc{\GO}{{G(\OO)}}
\nc{\GK}{G(\KK)}
\nc{\NK}{{N(\bbK)}}
\nc{\Shv}{\mathsf{Shv}}
\nc{\inc}{\mathsf{inc}}

\nc{\Par}{\mathsf{Par}}

\renc{\i}{\mathfrak{i}}

\nc{\NA}{N(\AA)}
\nc{\VA}{V(\AA)}

\nc{\Glue}{\mathsf{Glue}}
\nc{\laxlim}{\text{laxlim}}

\nc{\FT}{\mathsf{FT}}

\nc{\out}{\mathsf{out}}

\nc{\hol}{\mathsf{hol}}
\nc{\Hol}{\on{Hol}}
\nc{\add}{\mathsf{add}}

\nc{\sto}{\rightsquigarrow}
\nc{\squigto}{\rightsquigarrow}

\nc{\fW}{\mathfrak{W}}

\nc{\vrho}{\varrho}

\nc{\counit}{\mathsf{counit}}
\nc{\unit}{\mathsf{unit}}
\nc{\corr}{\mathsf{corr}}
\nc{\Corr}{\mathsf{Corr}}

\nc{\IndSch}{\mathsf{IndSch}}
\nc{\Tate}{{\mathsf{Tate}}}

\nc{\surjto}{\twoheadrightarrow}

\renc{\j}{\mathfrak{j}}

\nc{\J}{\mathcal{J}}

\nc{\pro}{\mathsf{pro}}
\nc{\fty}{\mathsf{ft}}
\nc{\Pro}{\mathsf{Pro}}

\nc{\coact}{\mathsf{coact}}
\nc{\aff}{\mathsf{aff}}

\nc{\Nilp}{\on{Nilp}}

\nc{\Gch}{{\check{G}}}
\nc{\Pch}{{\check{P}}}
\nc{\Mch}{{\check{M}}}
\nc{\Qch}{{\check{Q}}}

\nc{\LL}{\mathbb{L}}

\nc{\x}{\varkappa} 

\nc{\Otimes}{\boldsymbol{\otimes}}
\nc{\Times}{\boldsymbol{\times}}
\nc{\flip}{\text{<}}

\nc{\coeffRan}{\mathsf{coeff}^{\Ran}}

\nc{\Ha}{H(\sA)}
\nc{\Groups}{\mathsf{Groups}}

\nc{\Groth}{\mathsf{Groth}}

\nc{\rlto}{\rightleftarrows}

\nc{\DGCatRan}{\ShvCatCrys(\Ran)}

\nc{\longto}{\longrightarrow}

\nc{\LS}{{\on{LS}}}

\renc{\Jets}{\mathsf{Jets}}
\nc{\mer}{\mathsf{mer}}

\nc{\W}{\mathcal{W}}

\nc{\Sect}{\mathsf{Sect}}
\renc{\Maps}{\mathsf{Maps}}

\renc{\bf}{\mathbf{f}}

\nc{\y}{\mathtt{y}}
\renc{\x}{\mathtt{x}}

\nc{\un}{{\it un}}
\nc{\indep}{\mathsf{indep}}
\nc{\CoAlg}{\mathsf{CoAlg}}

\nc{\coeff}{\mathsf{coeff}}

\nc{\R}{\mathcal{R}}

\renc{\hat}{\widehat}

\nc{\TK}{T(\mathsf{K})} 
\nc{\TtKK}{\Tt(\mathpzc{K})} 
\nc{\TtK}{\Tt(\mathsf{K})} 
\nc{\KK}{\mathbb K}

\nc{\Dmod}{\mathfrak{D}}

\nc{\curs}[1]{\mathpzc{#1}}

\nc{\Bshv}{\bold{\B}}
\nc{\Bind}{\H_{\indep}}
\nc{\BRan}{\H_{\Ran}}
\nc{\ARan}{\A_{\Ran}}
\nc{\Aind}{\A_{\indep}}

\nc{\GrRan}{\Gr}
\nc{\Gr}{\mathsf{Gr}}
\nc{\GrGRan}{\Gr_{G}}
\nc{\GrGind}{\Gr_{G}^{\indep}}
\nc{\Grind}[1]{\Gr_{#1}^{\indep} }
\nc{\GrGdom}{\curs{Gr}_G}

\nc{\GMapsRan}[1]{\mathsf{GMaps}(X,{#1})}
\nc{\GSectRan}[1]{\mathsf{GSect}({#1}/X)}
\nc{\GMapsind}[1]{\mathsf{GMaps}(X,{#1})^\indep}
\nc{\GSectind}[1]{\mathsf{GSect}({#1}/X)^\indep}
\nc{\GMapsdom}[1]{\curs{GMaps}(X,{#1})}
\nc{\GSectdom}[1]{\curs{GSect}({#1}/X)}

\nc{\chind}{\ch^{\indep}}
\nc{\chdom}{\curs{ch}}

\nc{\QSect}[1]{\curs{QSect}(#1/X)} 
\nc{\QMaps}[1]{\curs{QMaps}(X,#1)} 

\nc{\Zar}{\mathit{Zar}}

\nc{\loccit}{\textit{loc.$\,$cit.}}
\nc{\Crys}{\on{Crys}}
\nc{\ShvCatCrys}{\ShvCat^{\Crys}}

\nc{\BPE}{{\BP E}}
\nc{\BVE}{{\BV E}}
\nc{\BBE}{{\BB E}}

\newcommand{\mapsfrom}{\mathrel{\reflectbox{\ensuremath{\mapsto}}}}

\nc{\Wh}{{{\CW}h}}

\nc{\ChiralCat}{\mathsf{ChiralCat}}
\nc{\RRep}{\mathfrak{R}ep}
\nc{\SSph}{\mathfrak{S}ph}

\nc{\tto}{\twoheadrightarrow}
\nc{\disj}{{\mathsf{disj}}}

\nc{\C}{\CC}
\nc{\Tch}{{\check{T}}}

\nc{\triv}{\mathsf{triv}}

\nc{\Alg}{\mathsf{Alg}}
\nc{\CAlg}{\mathsf{CAlg}}
\nc{\Spread}{\mathsf{Spread}}
\nc{\Dom}{\mathsf{Dom}}

\nc{\Jac}{\on{Jac}}
\renc{\CD}[1]{{#1}^{\on{CD}}}

\nc{\String}{\on{String}}

\renc{\min}{{\mathit{min}}}

\nc{\rrep}{\on-\!\mathbf{rep}}

\nc{\WWh}{\mathfrak{W}h}
\nc{\Grpd}{\mathsf{Grpd}}

\nc{\timesdisj}{\overset{\circ}\times}

\renc{\NA}{N(\sA)}
\nc{\chiral}{\mathsf{chiral}}

\nc{\Hopf}{\mathsf{Hopf}}

\nc{\heart}{\heartsuit}
\nc{\kk}{\mathbbm{k}} 

\nc{\HHom}{\CH{om}} 
\nc{\Cone}{\on{Cone}}

\nc{\EE}{\mathbb{E}}
\renc{\HC}{{\on{HC}}}
\nc{\HH}{{\on{HH}}}
\nc{\even}{{\on{even}}}
\nc{\SingSupp}{\on{SingSupp}}
\nc{\Supp}{\on{Supp}}

\nc{\temp}{{\mathit{temp}}}
\nc{\antitemp}{{\mathit{anti\!\on{-}\!temp}}}
\nc{\lantitemp}{{}^{\mathit{anti\!\on{-}\!temp}}}
\nc{\aut}{{\mathit{aut}}}
\nc{\cusp}{{\mathit{cusp}}}
\nc{\geom}{{\on{geom}}}
\nc{\ren}{{\on{ren}}}
\nc{\naive}{{\on{naive}}}
\nc{\stargen}{{*\ggen}}
\nc{\Whitgen}{{\Whit\ggen}}

\nc{\spec}{{\on{spec}}}

\nc{\gch}{\mathfrak{\check{g}}}

\nc{\Hecke}{\on{Hecke}}

\nc{\LSGch}{{\LS_\Gch}}
\nc{\LSPch}{{\LS_\Pch}}
\nc{\LSMch}{{\LS_\Mch}}

\nc{\Hsx}[2]{\H_{{#1} \leftarrow {#2}}}
\nc{\Hdx}[2]{\H_{{#1} \to {#2}}}
\nc{\Hcorr}[3]{ \H_{{#1} \leftarrow {#2} \to {#3}} }
\nc{\Hopcorr}[3]{ \H_{{#1} \to {#2} \leftto {#3}} }

\nc{\ICohsx}[2]{\ICohW_{{#1} \leftarrow {#2}}}
\nc{\ICohdx}[2]{\ICohW_{{#1} \to {#2}}}
\nc{\ICohcorr}[3]{ \ICohW_{{#1} \leftarrow {#2} \to {#3}} }
\nc{\ICohopcorr}[3]{ \ICohW_{{#1} \to {#2} \leftto {#3}} }

\nc{\QCohsx}[2]{\QCohW_{{#1} \leftarrow {#2}}}
\nc{\QCohdx}[2]{\QCohW_{{#1} \to {#2}}}
\nc{\QCohcorr}[3]{ \QCohW_{{#1} \leftarrow {#2} \to {#3}} }
\nc{\QCohopcorr}[3]{ \QCohW_{{#1} \to {#2} \leftto {#3}} }

\renc{\AA}{\bbA}
\nc{\Asx}[2]{\AA_{{#1} \leftarrow {#2}}}
\nc{\Adx}[2]{\AA_{{#1} \to {#2}}}
\nc{\Acorr}[3]{ \AA_{{#1} \leftarrow {#2} \to {#3}} }
\nc{\Aopcorr}[3]{ \AA_{{#1} \to {#2} \leftto {#3}} }

\nc{\Bsx}[2]{\B_{{#1} \leftarrow {#2}}}
\nc{\Bdx}[2]{\B_{{#1} \to {#2}}}
\nc{\Bcorr}[3]{ \B_{{#1} \leftarrow {#2} \to {#3}} }
\nc{\Bopcorr}[3]{ \B_{{#1} \to {#2} \leftto {#3}} }

\nc{\ICohzero}[3]{\ICoh_0 \bigt{#1 \times_{{#2}_\dR} #3}}
\nc{\IndCohzero}{\ICohzero}

\nc{\form}[3]{#1 \times_{{#2}_\dR} #3 }

\nc{\ind}{{\mathsf{ind}}}
\nc{\oblv}{{\mathsf{oblv}}}

\nc{\Aff}{\mathsf{Aff}}
\nc{\dgAff}{\Aff}

\nc{\deloop}{\mathsf{deloop}}
\renc{\loop}{\mathsf{loop}}

\nc{\coev}{\mathsf{coev}}

\nc{\bE}{\mathbf{E}}

\nc{\ShvCatH}{{\ShvCat^{\bbH}}}
\nc{\ShvCatQW}{\ShvCat^{\QCohW}}

\nc{\bbimod}{\on{-}\mathbf{bimod}}

\nc{\Tw}{\mathsf{Tw}}
\nc{\Arr}{\mathsf{Arr}}

\nc{\bDelta}{\bold\Delta}
\nc{\BiCat}{\mathsf{BiCat}}
\nc{\Seg}{\mathsf{Seg}}
\nc{\Cart}{\mathsf{Cart}}
\nc{\Bimod}{\mathsf{Bimod}}
\nc{\lax}{\mathit{lax}}

\nc{\pr}{\mathsf{pr}}

\nc{\zero}{ \{ 0 \}   }
\nc{\Perf}{\on{Perf}}

\nc{\leftto}{\leftarrow}
\nc{\lto}{\leftto}
\nc{\xlto}[1]{\xleftarrow{#1}}

\nc{\ltemp}{{}^\temp}
\nc{\lcusp}{{}^\cusp}
\nc{\TwCorr}{\mathsf{TwCorr}}

\nc{\Affover}[1]{{\Aff_{/#1}}}
\nc{\Affoverop}[1]{{( \Affover{#1})^\op}}

\nc{\AffOver}[2]{{(\Aff_{#1})_{/#2}}}

\nc{\AffOverop}[2]{{( \AffOver{#1}{#2})^\op}}

\nc{\aft}{{\mathit{aft}}}

\renc{\vert}{{\mathit{vert}}}
\nc{\horiz}{{\mathit{horiz}}}
\nc{\type}{{\mathit{type}}}
\nc{\adm}{{\mathit{adm}}}

\nc{\g}{\mathfrak{g}}

\nc{\free}{\mathsf{free}}

\nc{\Sform}{{S \times_{S_\dR} S}}
\nc{\Yform}{{\Y \times_{\Y_\dR} \Y}}
\nc{\SdR}{ {S_{\dR}}}

\nc{\laft}{{\mathit{laft}}}

\nc{\Affevcocaft}{\Aff_{\aft}^{< \infty}}
\nc{\Affaftevcoc}{\Aff_{\aft}^{< \infty}}

\nc{\Affevcoclfp}{\Aff_{\lfp}^{< \infty}}

\nc{\Schevcoclfp }{\Sch_{\lfp}^{< \infty}}
\nc{\Schevcocaft}{\Sch_{\aft}^{< \infty}}
\nc{\Schaftevcoc}{\Sch_{\aft}^{< \infty}}

\nc{\Stkevcoc}{\Stk^{< \infty}}
\nc{\Stkevcoclfp}{\Stk_{\lfp}^{< \infty}}
\nc{\Stkperfevcoclfp}{\Stk_{\mathit{perf},\lfp}^{< \infty}}
\nc{\Stkperflfp}{\Stk_{\mathit{perf},\lfp}}
\nc{\Stklfp}{\Stk_{\lfp}}

\nc{\evcoc}{\mathit{e.c.}}

\nc{\ICoh}{\IndCoh}

\nc{\citep}{\cite}

\renc{\H}{\bbH}

\nc{\uno}{\mathbbm{1}}

\nc{\CohBig}{{\Coh^{-\infty}}}

\nc{\Tang}{\mathbb{T}}

\nc{\LieAlg}{\mathsf{LieAlg}}

\nc{\Serre}{{\on{Serre}}}

\nc{\MPreStk}{\mathsf{MPreStk}}

\nc{\all}{{\on{all}}}

\nc{\QCohwedge}{\bbQ^\wedge}
\nc{\ICohwedge}{\bbI^\wedge}
\nc{\ICohW}{\ICohwedge}
\nc{\QCohW}{\QCohwedge}

\nc{\ShvCatA}{\ShvCat^{\AA}}
\nc{\ShvCatB}{{\ShvCat^\B}}

\nc{\naiveto}{{\xto{\naive}}}
\nc{\conaiveto}{{\xto{\conaive}}}

\nc{\strong}{\mathit{strong}}
\nc{\costrong}{\mathit{costrong}}

\nc{\conv}{\mathit{conv}}

\nc{\Q}{\bbQ}

\nc{\bY}{\mathbf{Y}}
\nc{\Loop}{\mathsf{LOOP}}

\nc{\DG}{{\on{DG}}}

\nc{\coind}{\mathsf{coind}}

\nc{\co}{{\on{co}}}

\nc{\laftdef}{{\mathit{laft-def}}}

\nc{\qsmooth}{{\mathit{qs.smooth}}}
\nc{\smooth}{{\mathit{smooth}}}

\nc{\LKE}{\on{LKE}}
\nc{\RKE}{\on{RKE}}

\nc{\ShvCatAco}{\ShvCatA_{\co}}
\nc{\ShvCatHco}{\ShvCatH_{\co}}

\nc{\Stk}{\mathsf{Stk}}

\renc{\2}{(\infty,2)}
\renc{\1}{(\infty,1)}

\nc{\doubleCat}{\mathsf{doubleCat}}
\nc{\Spaces}{\mathcal{S}\!\mathit{paces}}

\nc{\ALG}{\mathsf{ALG}}
\nc{\MAPS}{\mathsf{MAPS}}

\nc{\CAT}{\mathsf{CAT}}

\nc{\oneCat}{{\Cat_{\1}}}
\nc{\oneCAT}{{\CAT_{\1}}}

\nc{\twoCat}{{\Cat_{\2}}}
\nc{\twoCAT}{{\CAT_{\2}}}

\nc{\DGCAT}{\mathsf{DGCAT}}
\nc{\twoCatDG}{{\CAT_{\2}^\DG}}
\nc{\twoCATDG}{{\CAT_{\2}^\DG}}

\nc{\twoCATDGw}{{\CAT_{\2, w*}^\DG}}
\nc{\twoCATDGww}{{\CAT_{\2, ww*}^\DG}}

\nc{\AlgBimod}{\Alg^{\mathit{bimod}}}
\nc{\AlgBimodDGCat}{\AlgBimod(\DGCat)}
\nc{\ALGBimod}{\ALG^{\mathit{bimod}}}
\nc{\twoAlgBimod}{\ALGBimod}

\nc{\rev}{{\on{rev}}}

\nc{\lfp}{{\mathit{lfp}}}

\nc{\RBeck}{{\on{R-BC}}}
\nc{\LBeck}{{\on{L-BC}}}

\nc{\schem}{\mathit{schem}}
\nc{\proper}{\mathit{proper}}

\nc{\res}{{\mathit{res}}}

\nc{\UQCoh}{\U^{\QCoh}}
\nc{\UQ}{\UQCoh}

\nc{\LieAlgbd}{{\on{Lie-algbd}}}

\nc{\LY}{{L\Y}}

\nc{\TangQ}{\Tang^{\QCoh}}

\nc{\Fil}{{\on{Fil}}}
\nc{\AssGr}{\on{assoc-gr}}

\nc{\Cech}{\on{Cech}}

\nc{\FormMod}{\mathsf{FormMod}}
\nc{\FormModunderStkevcoc} {\FormMod_{\Stk^{< \infty}/}^\lfp }

\nc{\vDmod}{\virg{\Dmod}}
\nc{\LSG}{\LS_G}
\nc{\LSM}{\LS_M}
\nc{\LSP}{\LS_P}
\nc{\LST}{\LS_T}
\nc{\LSB}{\LS_B}
\nc{\Gm}{{\GG_m}}
\nc{\GIT}{/\!/}
\renc{\t}{\ft}

\nc{\PP}{\bbP}
\nc{\Nglob}{\check{\N}_\glob}

\nc{\Psid}{\on{Ps-Id}}
\nc{\PsId}{\Psid}
\nc{\unshift}{\Rightarrow}
\nc{\coker}{\on{cone}}

\nc{\irred}{{\on{irred}}}
\nc{\red}{{\on{red}}}
\nc{\Spr}{{\on{Spr}}}
\nc{\DL}{\mathsf{DL}}
\nc{\St}{{\on{St}}}
\nc{\Glued}{{\mathsf{Glued}}}

\nc{\LSGchLeviirred}{\LS_\Gch^{\on{Levi-irred}}}
\nc{\LSLeviirred}{\LS_G^{\on{Levi-irred}}}

\nc{\olBun}{\ol{\Bun}}
\nc{\cl}{\on{cl}}

\nc{\RHom}{\on{RHom}}

\nc{\Poinc}{\on{Poinc}}
\nc{\Betti}{\on{Betti}}
\nc{\LSGcoarse}{\LS_{G,\on{coarse}}}

\nc{\mon}{{\Gm\on{-mon}}}

\nc{\TBunG}{\sT_{\Bun_G}}
\nc{\lperp}{{}^\perp}

\nc{\Levi}{\on{Levi}}
\nc{\Verdier}{{\on{Verdier}}}

\nc{\pscpt}{{\on{ps-cpt}}}

\nc{\ICohNSt}{\ICoh_{\N + \St}(\LSG)}
\nc{\NSt}{\ICohNSt}

\nc{\DmodBig}{\Dmod_{!+*}}

\nc{\WSph}{\on{WS}}

\nc{\backsl}{\backslash}

\nc{\AJ}{\on{AJ}}

\nc{\GR}{G(R)}

\nc{\B}{{\bbone_{\Sph_G}^\temp}}
\nc{\unittempG}{{\bbone_{\Sph_G}^\temp}}
\nc{\unittempM}{{\bbone_{\Sph_M}^\temp}}

\nc{\BunPinfty}
{\Bun_P^{x,\infty}}

\nc{\adj}{\on{adj}}
\nc{\Iso}{\on{Iso}}

\nc{\nonext}{\on{non-ext}}

\nc{\dom}{{\on{dom}}}

\nc{\Mat}{\mathit{Mat}}

\nc{\Hren}{H_{\BM}}
\renc{\BM}{\on{BM}}

\nc{\univ}{\on{univ}}
\nc{\sing}{\on{sing}}

\nc{\new}{\on{new}}

\nc{\Gcirc}{G^\circ}

\nc{\ppart}{(\!(t)\!)}

\nc{\equivariant}{\on{-equiv}}

\nc{\Cotrnk}{\on{Cotrnk}}
\nc{\Ctrnk}{\Cotrnk}
\nc{\Cotrunk}{\Cotrnk}

\nc{\DeltaGen}{\wt\Delta}

\nc{\GS}{G[\Sigma]}
\nc{\GSgamma}{\Ad_\gamma(\GS)}
\nc{\GSgamman}{\GS_\gamma^{\leq n}}
\nc{\GSn}{\GS^{\leq n}}

\nc{\lopen}{{\lambda{\mathit{-open}}}}
\nc{\lclosed}{{\lambda{\mathit{-closed}}}}

\nc{\BK}{{B(\KK)}}

\nc{\omegaren}{\omega^{\ren}}

\nc{\Ybg}{\Y_{\beta, \gamma}}

\nc{\Toep}[1]{\on{Toep}_{#1}}
\nc{\ToepDeg}[1]{\Toep{#1}^{\on{deg}}}

\nc{\jXS}{j^{\circ, \times \to \all}}
\nc{\jcirc}{j^{\circ \to \all}}
\nc{\jtimes}{j^{\times \to \all}}
\nc{\jXStoCirc}{j^{\circ, \times \to \circ}}
\nc{\jXStoTimes}{j^{\circ, \times \to \times}}
\nc{\GrXS}{\Gr_G^{\circ, \times}}

\nc{\iBullettoTimes}{i^{\bullet \to \times}}

\nc{\BunGnontriv}{\Bun_G^{\mathsf{non}\!\on-\!\mathsf{triv}}}

\nc{\Kos}{{\on{Kos}}}

\nc{\Nch}{\check{\N}}
\nc{\WS}{\W\S}

\nc{\NO}{N(\OO)}

\nc{\sR}{\mathsf{R}}
\renc{\Nilp}{\Nch}

\renc{\ss}{\on{ss}}

\nc{\Op}{\on{Op}}

\nc{\OpUnrGlob}{\LS_{\Gch}^{\Op\ggen}}
\nc{\sigmatriv}{\sigma_{\on{triv}}}

\begin{document}

\title{On the geometric Ramanujan conjecture}

\begin{abstract}

We prove two main results relevant for the geometric Langlands program.
The first result is an analogue of the Ramanujan conjecture: any cuspidal D-module on $\Bun_G$ is tempered. We actually prove a more general statement: any D-module that is $*$-extended from a quasi-compact open substack of $\Bun_G$ is tempered. Then the assertion about cuspidal objects is an immediate consequence of a theorem of Drinfeld-Gaitsgory.
Building up on this, we prove our second main result, the \emph{automorphic gluing} theorem for the group $SL_2$: it states that any D-module on $\Bun_{SL_2}$ is determined by its tempered part and its constant term. 
This theorem (vaguely speaking, an analogue of Langlands' classification for the group $SL_2(\bbR)$) corresponds under geometric Langlands to the spectral gluing theorem of Arinkin-Gaitsgory and the author.
In the final part of the paper, we argue that the above results pave the way for a potential proof of the geometric Langlands conjecture for $G=SL_2$.
  
\end{abstract}

\author{Dario Beraldo}

\maketitle

\sec{Introduction}

\ssec{Cuspidal and tempered D-modules}

In the (unramified and global) geometric version of the Langlands program, one fixes an algebraically closed ground field $\kk$ of characteristic zero and considers $\Bun_G := \Bun_G(X)$, the stack of $G$-bundles on a smooth projective $\kk$-curve $X$.  Here and everywhere else in this paper, $G$ denotes a connected reductive $\kk$-group.

\sssec{}

The \emph{automorphic side} of the geometric Langlands correspondence is the DG category of D-modules on $\Bun_G$, denoted in this paper by $\Dmod(\Bun_G)$. 
We let $\Dmod(\Bun_G)^\cusp$ and $\Dmod(\Bun_G)^\temp$ be the full subcategories consisting of \emph{cuspidal} and \emph{tempered} objects, respectively.

\sssec{}

The definition of cuspidality parallels exactly the classical case: an object $\F \in \Dmod(\Bun_G)$ is said to be cuspidal if its constant terms $\CT_P(\F)$ vanish for all parabolic subgroups $P \subsetneq G$.
Let us recall that
$$
\CT_P := (\fq_P)_{*} \circ (\fp_P)^{!}:
\Dmod(\Bun_G)
\longto
\Dmod(\Bun_P)
\longto
\Dmod(\Bun_M)
$$
is the functor of D-module pull-push along the natural diagram
$$
\Bun_G \xleftarrow{\;\;\fp_P\;\;}
\Bun_P \xto{\;\;\fq_P\;\;}
\Bun_M,
$$
where $M$ is the Levi quotient of $P$.

\sssec{}

The temperedness condition for objects of $\Dmod(\Bun_G)$ is reviewed below: see Section \ref{intro:temp} for a leisurely discussion and Sections \ref{ssec:derived-satake}-\ref{ssec:tempered-objects} for the proper treatment. Besides the original definition taken from \cite[Section 12]{AG1}, we will present three different-looking characterizations: a review of the characterization in terms of the big cell of the affine Grassmannian following \cite{omega}, a characterization in terms of Whittaker invariants, and finally a t-structure characterization.
In this paper, the most important characterization is the latter one; however, it seems to us that the three points of view are all useful in different parts of the theory.

\sssec{}

If we assume the geometric Langlands conjecture, the tempered condition is very natural. Namely, after \cite{AG1}, the Langlands conjecture calls for an equivalence
\begin{equation} \label{eqn:GL-conj}
\LL_G:
\ICoh_{\Nch}(\LSGch)
\xto{\;\; \simeq \; \;}
\Dmod(\Bun_G),
\end{equation}
where $\LSGch$ is the (derived) stack of $\Gch$-local systems on $X$ and $\ICoh_{\Nch}(\LSGch)$ is the DG category of ind-coherent sheaves with nilpotent singular support. The precise definition is not important for now: what is important is that $\ICoh_{\Nch}(\LSGch)$ contains $\QCoh(\LSGch)$ as a full subcategory, embedded as the subcategory of $\ICoh_{\Nch}(\LSGch)$ consisting of obejcts with zero singular support. Then $\Dmod(\Bun_G)^\temp$ is designed to correspond to $\QCoh(\LSGch)$ under \eqref{eqn:GL-conj}, yielding the conjecture
\begin{equation} \label{eqn:GL-conj-temp}
\LL_G^\temp:
\QCoh(\LSGch)
\xto{\simeq}
\Dmod(\Bun_G)^\temp.
\end{equation}

\sssec{}

In general and vague terms, the \emph{Ramanujan conjecture} states that cuspidal objects are tempered.
Our first main result confirms this expectation in the context of the geometric Langlands program:

\begin{mainthm} \label{mainthm:Ramanujan}
The following inclusion holds: $\Dmod(\Bun_G)^\cusp \subseteq \Dmod(\Bun_G)^\temp$.
\end{mainthm}

\ssec{Motivation for the validity of Ramanujan conjecture}

Let us explain how to deduce the above conjecture from the geometric Langlands conjecture and its compatibility with constant terms. 

\sssec{}

According to \cite[Section 13]{AG1}, the \emph{spectral side}\footnote{that is, the side of $\ICoh_{\Nch}(\LSGch)$} of \eqref{eqn:GL-conj} has its own constant term functors
$$
\CT_{\Pch}^{\spec}:
\ICoh_{\Nch}(\LSGch)
\longto
\ICoh_{\Nch}(\LSMch),
$$
defined by pull-push for ind-coherent sheaves along the natural diagram
$$
\LSGch \leftto \LSPch \to \LSMch.
$$

\sssec{}

It is expected that $\CT_P$ and $\CT_\Pch^{\spec}$ correspond to each other\footnote{up to twisting with an explicit line bundle on $\LSMch$, which is irrelevant for us and omitted in what follows} under geometric Langlands: i.e., the following diagram ought to be commutative for any $P \subseteq G$:
\begin{equation} 
\nonumber
\begin{tikzpicture}[scale=1.5]
\node (00) at (0,0) {$ \ICoh_{\Nilp}(\LSGch)$};
\node (10) at (3,0) {$ \ICoh_{\Nilp}(\LSMch)$.};
\node (01) at (0,1) {$ \Dmod(\Bun_G)$};
\node (11) at (3,1) {$\Dmod(\Bun_M)$};
\path[->,font=\scriptsize,>=angle 90]
(00.east) edge node[above] {$\CT_{\Pch}^{\spec}$}  (10.west); 
\path[->,font=\scriptsize,>=angle 90]
(01.east) edge node[above] {$\CT_P$} (11.west); 
\path[->,font=\scriptsize,>=angle 90]
(01.south) edge node[right] {$\stackrel ? \simeq$} (00.north);
\path[->,font=\scriptsize,>=angle 90]
(11.south) edge node[right] {$\stackrel ? \simeq$} (10.north);
\end{tikzpicture}
\end{equation}
This is the compatibilty of geometric Langlands with constant terms (equivalently, by adjunction, with Eisenstein series), see again \cite[Section 13]{AG1} or \cite[Chapter 6]{Outline}.

\sssec{}

In particular, we might define cuspidal objects in $ \ICoh_{\Nilp}(\LSGch)$ as those objects annihilated by all $\CT_\Pch^{\spec}$ for all $\Pch \subsetneq \Gch$, and then geometric Langlands would force
\begin{equation} \label{eqn:GL-cusp}
\LL_G^\cusp:
\ICoh_{\Nilp}(\LSGch)^\cusp
\xto{\;\;\simeq \;\;}
\Dmod(\Bun_G)^\cusp.
\end{equation}
However, $\ICoh_{\Nilp}(\LSGch)^\cusp$ was computed\footnote{Ultimately, the computation boils down to the Jacobson-Morozov theorem.} in \cite[Section 13.3]{AG1} to be equivalent to $\QCoh(\LS_{\Gch}^{\irred})$,
embedded into $\ICoh_{\Nilp}(\LSGch)$ via the composition
\begin{equation}\label{eqn:short-chain}
\begin{tikzpicture}[scale=1.5]
\node (a) at (0,1) {$\QCoh(\LS_{\Gch}^{\irred})$};
\node (b) at (3.5,1) {$\QCoh(\LS_{\Gch})
\subseteq
\ICoh_{\Nilp}(\LSGch).$};
\path[right hook ->,font=\scriptsize,>=angle 90]
(a.east) edge node[above] {$(j^{\irred})_{*}$} (b.west);
\end{tikzpicture}
\end{equation}
Here $j^{\irred}: \LS_{\Gch}^{\irred} \hto \LS_{\Gch}$ is the open embedding of the stack of \emph{irreducible} $\Gch$-local systems.

\sssec{}

Thus, the fully faithful embedding 
$$
(j^{\irred})_{*}: 
\QCoh(\LS_{\Gch}^{\irred})
\hto
\QCoh(\LS_{\Gch})
$$
is the Langlands-dual version of the Ramanujan conjecture.

\ssec{Star-extensions}

We do not tackle the statement of Theorem \ref{mainthm:Ramanujan} directly. Rather, we divide the theorem in two parts by considering yet another full subcategory of $\Dmod(\Bun_G)$ that will turn out to sit between $\Dmod(\Bun_G)^\cusp$ and $\Dmod(\Bun_G)^\temp$. 

\sssec{}

The full subcategory in question is the DG category $\Dmod(\Bun_G)^\stargen$, defined as follows. 
We begin by recalling that $\Bun_G$ is never quasi-compact (unless $G$ is trivial): indeed, $\Bun_G$ is an increasing union of countably-many quasi-compact open substacks obtained by bounding the level of instability. 
Denote by $\Dmod(\Bun_G)^\stargen$ the full subcategory of $\Dmod(\Bun_G)$ generated under colimits by $*$-extensions from quasi-compact open substacks. 

\sssec{}

In \cite[Appendix B]{CT}, it is proven that any cuspidal D-module is a $*$-extension from an explicit open substack $\U_\cusp \subset \Bun_G$. This substack is not quasi-compact in general, but its intersection with each connected component of $\Bun_G$ is. This implies that any cuspidal D-module belongs to $\Dmod(\Bun_G)^\stargen$.
Hence, the Ramanujan conjecture is an immediate corollary of the following result:

\begin{mainthm}\label{mainthm:star}
The following inclusion holds: $\Dmod(\Bun_G)^\stargen 
\subseteq 
\Dmod(\Bun_G)^\temp$.
\end{mainthm}

\sssec{}

The statement of Theorem \ref{mainthm:star} was proposed by us (see \cite{DL}) in the form of a conjecture. Our motivation for this conjecture was the behaviour of Deligne-Lusztig duality on the spectral side of the geometric Langlands correspondence. 
Let us briefly recall some key points of that situation. The chain \eqref{eqn:short-chain} can be refined to a longer chain
of fully faithful embeddings
\begin{equation} \label{chain-spec}
\QCoh(\LS_{\Gch}^\irred)
\subseteq
\QCoh(\LSGch)^{\ss}
\subseteq
\QCoh(\LSGch)
\subseteq
\ICoh_{\Nch}(\LSGch),
\end{equation}
where $\QCoh(\LSGch)^{\ss}$ is \virg{morally} the full-subcategory of $\QCoh(\LSGch)$ spanned by objects set-theoretically supported on the locus of semi-simple $\Gch$-local systems.\footnote{The word \virg{morally} is included because the locus of semi-simple $\Gch$-local systems is not a well-defined locally closed substack of $\LSGch$. See the discussion in \cite{DL} for details.}
In \cite{DL}, we considered the spectral Deligne-Lusztig functor $\DL_\Gch^{\spec}$ and showed that: 
\begin{itemize}
\item
$\DL_\Gch^{\spec}$ acts as the identity on $\QCoh(\LS_{\Gch}^\irred)$;
\item
its essential image equals $\QCoh(\LSGch)^{\ss}$, hence in particular it is contained in $\QCoh(\LSGch)$.
\end{itemize}
We then argued that $\QCoh(\LSGch)^{\ss}$ ought to correspond to $\Dmod(\Bun_G)^\stargen$ under the geometric Langlands conjecture. Combining this fact with the other three more standard version of the conjecture, we see that \eqref{chain-spec} should correspond to the chain
$$
\Dmod(\Bun_G)^\cusp
\subseteq
\Dmod(\Bun_G)^\stargen
\subseteq
\Dmod(\Bun_G)^\temp
\subseteq
\Dmod(\Bun_G).
$$

\begin{example} \label{ex:abelian}
For $G =T$ abelian and $X$ of arbitrary genus, the inclusions in the above two chains are all equivalences.
Hence, from now on we will focus on non-abelian groups.
\end{example}

\begin{example}
Now consider $X = \PP^1$ and let $G$ be arbitrary (connected reductive, as always). 
One can show directly, following \cite[Section 3.1]{omega} for instance, that 
$$
\Dmod(\Bun_G(\bbP^1))^\stargen 
\simeq
\Dmod(\Bun_G(\PP^1))^\temp.
$$ 
This establishes Theorem \ref{mainthm:star} in the genus $0$ case.\footnote{More details will be provided in Section \ref{ssec:case of P1}.}
On the other hand, Theorem \ref{mainthm:Ramanujan} is trivially true since $\Dmod(\Bun_G(\PP^1))^\cusp \simeq 0$. 
\end{example}

\begin{rem}

Now let us return to the general situation of $G$ and $X$ both arbitrary.
As a consequence of Theorem \ref{mainthm:star}, the notion of temperedness depends only on the geometry of $\Bun_G$ \emph{at infinity}. More precisely, an object $\F \in \Dmod(\Bun_G)$ is tempered if and only if so is 
$$
(i_Z)_!(i_Z)^!(\F)
\simeq
\ker 
\Bigt{
\F \to (j_U)_* (j_U)^* (\F)
}
$$
for any closed substack $Z \subseteq \Bun_G$ whose complement $U$ is quasi-compact in every connected component.

\end{rem}

\ssec{The diagonal}

In the main body of the paper, we explain that the proof of Theorem \ref{mainthm:star} reduces to checking the temperedness of a single D-module on the product $\Bun_G \times \Bun_G \simeq \Bun_{G \times G}$. 
Namely, we will show that Theorem \ref{mainthm:star} is equivalent to:

\begin{mainthmbis}{mainthm:star}\label{mainthm:diagonal}
Let $\Delta: \Bun_G \to \Bun_G \times \Bun_G$ be the diagonal map, $\Delta_*:\Dmod(\Bun_G) \to \Dmod(\Bun_G \times \Bun_G)$ the de Rham pushforward, and $\omega_{\Bun_G} \in \Dmod(\Bun_G)$ the dualizing sheaf. The object $\Delta_{*}(\omega_{\Bun_G})$ is tempered in $\Dmod(\Bun_G \times \Bun_G) $.
\end{mainthmbis}

\begin{rem}

This stands in sharp contrast with the fact, proven in \cite[Theorem A]{omega}, that $\omega_{\Bun_G}$ is \emph{anti-tempered} (that is, right orthogonal to all tempered objects).
Of course, there is no contradiction since $\Delta_{*}$ does not respect the relevant Hecke actions and thus it is not supposed to preserve (anti)tempered objects (see below).
\end{rem}

\ssec{Tempered objects} \label{intro:temp}

Now we come to the definition of $\Dmod(\Bun_G)^\temp$, see \cite[Section 12]{AG1} for the original source and Sections \ref{ssec:derived-satake}-\ref{ssec:tempered-objects} for details. 
We also refer to our treatment in \cite{omega}.

\sssec{}

Even though $\Dmod(\Bun_G)^\temp$ is defined intrinsically (i.e., independently of geometric Langlands), the reason for its introduction comes from the geometric Langlands conjecture: as mentioned earlier, the full subcategory $\Dmod(\Bun_G)^\temp \subseteq \Dmod(\Bun_G)$ is designed to correspond to the full subcategory $\QCoh(\LSGch) \subseteq \ICoh_{\Nch}(\LSGch)$ under geometric Langlands.

\sssec{}

In this section we list several equivalent characterizations of tempered objects. They look very different from one another and seem to be all useful in different applications.
The general context is that of a DG category $\C$ equipped with an action\footnote{We use the symbol $\star$ to indicate the action of $\Sph_G$ on $\C$.} of the monoidal DG category $$
\Sph_G := \Dmod(\GO \backsl \GK /\GO).
$$
The latter is the \emph{spherical DG category} (see \cite[Section 12]{AG1}), equipped with the monoidal structure given by convolution. 
Whenever $\C$ is endowed with an action of $\Sph_G$, we will be able to define its full subcategory $\C^\temp$ of \emph{tempered objects}.

\sssec{}

We will apply the general definition in the case of the Hecke action of $\Sph_G$ on $\Dmod(\Bun_G)$ at a point $x \in X$.
By \cite{indep}, this definition is independent of the choice of the point $x \in X$ (see also \cite{shvcatHH} for a related discussion). In any case, the particular point $x \in X$ is not important: whenever we can prove that an object is $x$-tempered, the same proof will work verbatim for any other point of the curve.

\begin{rem}
Consider the general setup of $\C$ acted on by $\Sph_G$. As claimed, the definitions below will produce a full subcategory $\C^\temp \subseteq \C$  formed by \emph{tempered objects}. We also define $\C^{\antitemp} \subseteq \C$ to be the right orthogonal to $\C^{\temp}$. Objects of $\C^{\antitemp}$ are called \emph{anti-tempered}.
Of course, $\C^{temp}$ and $\C^{\antitemp}$ determine each other and in certain situations it is more natural to define $\C^{\antitemp}$ first.
\end{rem}

\sssec{}

The mechanics of the original definition, due to Arinkin-Gaitsgory, consists of the following steps, see \citep{AG1}, as well as \cite{Outline} and \citep{omega} for details.

One first defines temperedness in the universal case, that is, for the left action of $\Sph_G$ on itself. In this case, the derived Satake theorem (\cite{BF} and Section \ref{ssec:derived-satake} below) states that $\Sph_G$ is equivalent to a certain DG category of ind-coherent sheaves; one takes $\Sph_G^\temp$ to be the corresponding full subcategory of quasi-coherent sheaves. 

One then sets $\C^\temp \subseteq \C$ to be the full-subcategory
$$
\C^\temp
:=
\{
\S \star c \; \big| \; \mbox{$\S \in \Sph_G^\temp$ and $c \in \C$}
\}.
$$

\sssec{} \label{sssec:misterB}

Thus, to provide more explicit characterizations of $\C^\temp$, we need to gain a good understanding of $\Sph_G^{\temp}$. We are aware of three different ways to accomplish that: by using the negative part of the loop group, by using Whittaker objects, by using the t-structure. (The latter method is one of the novelties of the present paper.) 
To state these characterizations in order, we need to recall two natural objects of $\Sph_G$ and make an observation.

\begin{enumerate}

\item
The first object is $\B := (f^!)^R(\omega_{ G \backsl G(R) /G})$, where $\GR$ is the negative part of the loop group,
$$
f: G \backsl G(R) /G \to \GO \backsl \GK / \GO
$$
the natural map, and $(f^!)^R$ is the right adjoint to the natural pullback $f^!$. This is called \emph{tempered unit of $\Sph_G$}, see \cite{omega} for the reason.

\item

The second object, called the \emph{basic Whittaker-spherical object}, is given by $\WS_0 := \Av_*^\GO (\chi_{\Gr_N})$. Definitions and appropriate references will appear in the main body of the paper (alternatively, the reader might consult \cite{thesis}).

\item

Finally, anticipating the contents of the next section, we claim that $\Sph_G$ admits a natural t-structure containing several \emph{non-trivial infinitely connective} objects.

\end{enumerate}
Granting the items in the above list, we can provide three explicit definitions of the notions of tempered/antitempered objects.

\begin{thm}
Let $\C$ be a DG category equipped with an action of $\Sph_G$. Denoting the action by $\star$, we have:
\begin{enumerate}
\item
$c \in \C$ is tempered iff there exists an isomorphism $ \B \star c \simeq c$;
\item
$c \in \C$ is anti-tempered iff $\B \star c \simeq 0$;
\item
$c \in \C$ is anti-tempered iff $\WS_0 \star c \simeq 0$;
\item
$c \in \C$ is tempered iff $\A \star c \simeq 0$ for any infinitely connective object $\A \in \Sph_G$.
\end{enumerate}
\end{thm}

\sssec{}

The first two statements were proved in \cite{omega} and will only be used marginally in this paper. The third and fourth statements are proven in Section \ref{sec:temp objects}.
The third statement is well-known to experts and it is not used in this paper; we decided to include it nevertheless because it belongs to the same circle of ideas. On the other hand, the fourth statement is the crucial input for our proofs.

\ssec{Infinitely connective objects}

Let $\C$ be a DG category equipped with a t-structure. 
Essential to this paper is the notion of \emph{infinitely connective object}. Let us review this notion, provide some examples and explain the role it plays here.

\sssec{}

First, let us recall the main definition, following our discussion in \cite[Sections 1.5 and 6.1]{omega}. When dealing with t-structures, we use the \emph{cohomological notation}. Given a DG category $\C$ equipped with a t-structure, let 
$$
\C^{\leq - \infty} := \bigcap_{n \gg 0 } \C^{\leq -n}.
$$
This is the full subcategory of $\C$ consisting of \emph{infinitely connective objects}.

\sssec{}

At first glance, it might seem that $\C^{\leq - \infty}$ is always zero. For instance, this is the case for $\C = \Vect$ and much more generally for $\C = \QCoh(\Y)$ where $\Y$ is a derived algebraic stack.
Even more generally, $\C^{\leq - \infty} \simeq 0$ whenever the t-structure on $\C$ is left-complete.

However, nonzero infinitely connective objects play an important role in the algebraic geometry of (even mildly) singular varieties: for instance, they are responsible for the difference between quasi-coherent and ind-coherent sheaves on a singular quasi-smooth scheme.

\sssec{}

Namely, in the particular case of $\LSGch$ mentioned above (and much more generally), the DG category $\ICoh_{\Nch}(\LSGch)$ possesses a natural t-structure and $\ICoh_{\Nch}(\LSGch)^{\leq - \infty}$ is equivalent to the kernel of the standard projection functor
$$
\Psi_{0 \hto \Nch}: 
\ICoh_{\Nch}(\LSGch)
\tto
\QCoh(\LSGch).
$$

\begin{rem}
The above is a precise manifestation of the fact that the difference between $\ICoh$ and $\QCoh$ lies cohomologically at $-\infty$, see \cite{ICoh}.
\end{rem}

\sssec{}

On the other side of the geometric Langlands conjecture, we will see that the DG category $\Dmod(\Bun_G)$ possesses a natural t-structure, too. However, in this case we find that $ \Dmod(\Bun_G)^{\leq - \infty} \simeq 0$. Indeed, as proven in \cite{finiteness}, the t-structure on the DG category of D-modules on an algebraic stack (quasi-compact or not) is left-complete.

\sssec{}

The above observations force the conjectural Langlands functor $\bbL_G: \ICoh_{\Nch}(\LSGch) \to \Dmod(\Bun_G)$ to have infinite cohomological amplitude with respect to the natural t-structures on both sides. Put another way, infinitely connective objects on the two sides of the geometric Langlands correspondence do not match. This is where anti-tempered objects come handy. Indeed, by design, \emph{infinitely connective objects on the spectral side should correspond to anti-tempered objects on the automorphic side}.

From this discussion, one might draw the conclusion that infinitely connective objects only appear and play a role on the spectral side of the geometric Langlands conjecture. This is actually not true: as we explain next, infinitely connective objects appear and are crucial in the proof of Theorem \ref{mainthm:diagonal} as well.

\sssec{}

Indeed, while the t-structure on D-modules on stacks is left-complete, this is usually not the case for D-modules on ind-schemes. The simplest example is arguably the ind-scheme $\AA^\infty := \colim_{d} \AA^d$: it is a nice exercise to check that the dualizing D-module $\omega_{\AA^\infty}$ is infinitely connective.

Now, $\Bun_G$ is closely related to two ind-schemes: the affine Grassmannian $\Gr_{G,x}$ on the one hand, and the ind-scheme $G[\Sigma]$ of maps from a smooth affine curve $\Sigma$ to $G$ on the other hand. 
For instance, when $G$ is semisimple, $\Bun_G$ can be realized as a quotient $\Gr_{G.x}/G[X-x]$.
In view of \cite[Theorem D and Corollary 1.5.3]{omega}, we have:

\begin{thm} \label{thm:omega}
The natural t-structures on $\Dmod(\Gr_{G,x})$ and $\Dmod(G[\Sigma])$ are not left-complete (unless $G$ is abelian, which is excluded from our considerations after Example \ref{ex:abelian}). In fact, both $\omega_{\Gr_{G,x}}$ and $\omega_{G[\Sigma]}$ are infinitely connective.%
\footnote{In \citep{omega}, we provided several other examples of indschemes $Y$ whose dualizing sheaf is infinitely connective. }
\end{thm}

For the proof of Theorem \ref{mainthm:diagonal}, we will need to consider two prestacks built out of $\Gr_{G,x}$ and $G[X-x]$.

\sssec{}

Related to $\Gr_{G,x}$ is the \emph{local Hecke stack}\footnote{This is an abuse of terminology in that $\Hecke_{G,x}^\loc$ is not an algebraic stack and not even an ind-stack. We adhere to this terminology by reasons of tradition.}
$$
\Hecke_{G,x}^\loc :=
\GO \backsl \GK / \GO.
$$
It parametrizes pairs of $G$-bundles on the infinitesimal disk at $x$ together with an isomorphism of their restrictions to the punctured disk. Its DG category of D-modules is $\Sph_G$ by construction. In the main body of the paper, we equip $\Sph_G$ with a natural t-structure and observe that Theorem \ref{thm:omega} implies that $\Sph_G^{\leq - \infty}$ is nonzero.

\sssec{}

Related to $G[X-x]$ is the \emph{global Hecke ind-stack}
$$
\Hecke^\glob_{G,x}
:=
\Bun_G \utimes_{\Bun_G(X-x)} \Bun_G.
$$
This ind-stack parametrizes pairs of $G$-bundles on $X$ together with an isomorphism of their restrictions to the punctured curve $X-x$.
The part of Theorem \ref{thm:omega} concerning $G[\Sigma]$ implies that $\Dmod(\Hecke^\glob_{G,x})$ contains nontrivial infinitely connective objects.

\sssec{} \label{sssec:r maps of Hecke}

We now ask how the infinitely connective objects of $\Sph_{G,x}$ and $\Dmod(\Hecke^\glob_{G,x})$ are related.
First of all, there is an evident restriction map
$$
r:
\Hecke^\glob_{G,x}
\longto
\Hecke_{G,x}^\loc
$$
whose pullback yields a functor
$$
r^!:
\Sph_{G,x}
\longto
\Dmod(
\Hecke^\glob_{G,x}).
$$
In the main body of the paper, we will prove the following statement and show it is the key t-structure estimate behind the proof of Theorem \ref{mainthm:diagonal}.

\begin{mainthm} \label{mainthm:Hecke}
The functor $r^!:
\Sph_{G}
\longto
\Dmod(
\Hecke^\glob_{G,x})$ preserves infinitely connective objects.
\end{mainthm}

\ssec{Automorphic gluing for $SL_2$}

In the second part of the paper, we restrict to $G=SL_2$ and prove the \emph{automorphic gluing theorem} in that case.

\sssec{} \label{sssec:rough}

Let us recall the statement, which was briefly discussed in \cite{shvcatHH} and \cite{epiga}. Roughly, the theorem says that any $\F \in \Dmod(\Bun_{SL_2})$ can be reconstructed from its tempered part and its constant term.
To make this idea precise, we first need to refine the notion of constant term to make it Hecke-equivariant, as follows.

\sssec{}

The issue with the usual constant term functor (here $G$ is any group, $P \subseteq G$ a parabolic with Levi quotient $M$)
$$
\CT_P:
\Dmod(\Bun_G) 
\xto{(\p_B)^!} 
\Dmod(\Bun_P)
\xto{(\q_B)_*}
\Dmod(\Bun_M)
$$
is that the target DG category does not carry an action of $\Sph_G$, hence it does not even make sense to ask whether this functor is $\Sph_G$-equivariant. 

\sssec{}

Following \cite{Outline} and \cite{Barlev} (see also \cite{shvcatHH} for another point of view), there is a canonical way to modify $\Dmod(\Bun_P)$ and $\Dmod(\Bun_M)$ to make each of them $\Sph_G$-equivariant: one replaces $\Dmod(\Bun_P)$ with $\Dmod(\Bun_G^{P\ggen})$ and $\Dmod(\Bun_M)$ with $I(G,P)$. 
We postpone the definitions of $\Bun_G^{P\ggen}$ and $I(G,P)$ to Section \ref{ssec:aut-gluing-statement}; here are however the adelic analogies:
\begin{eqnarray}
\nonumber
\Dmod(\Bun_G) & \approx & \Fun( \GO \backsl \GA / G(\bbF_X))  \\
\nonumber
\Dmod(\Bun_P)& \approx & \Fun( P(\OO) \backsl P(\AA) / P(\bbF_X)) \\
\nonumber
\Dmod(\Bun_G^{P \ggen}) 
& \approx & 
\Fun(G(\OO) \backsl G(\AA) / P(\bbF_X)) \\
\nonumber
\Dmod(\Bun_M) 
& \approx &
\Fun( M(\OO) \backsl M(\AA) / M(\bbF_X)) \\
\nonumber
I(G,P) 
& \approx &
\Fun( G(\OO) \backsl G(\AA) / U(\AA) M(\bbF_X)).
\end{eqnarray}
These analogies hopefully make it clear that $\Sph_G$ acts by convolution on $\Dmod(\Bun_G^{P \ggen})$ and $I(G,P)$, but not on $\Dmod(\Bun_P)$ and $\Dmod(\Bun_M)$. 

\sssec{}

Following \cite{Outline} again, one defines a $\Sph_G$-linear functor
$$
\CT_P^{\enh}: 
\Dmod(\Bun_G)
\longto
I(G,P)
$$
whose composition with the natural forgetful (conservative) functor
$$
\oblv_{G,P}: I(G,P) \to \Dmod(\Bun_M)
$$
is the usual constant term $\CT_P$. For this reason, the first functor is called \emph{enhanced constant term}.

\sssec{}

Now, since $I(G,P)$ is endowed with an action of $\Sph_G$, it makes sense to consider its tempered subcategory $I(G,P)^\temp$ and the associated temperization functor $I(G,P) \tto I(G,P)^\temp$. 
The commutativity of the following diagram is tautological:
\begin{equation} 
\nonumber
\begin{tikzpicture}[scale=1.5]
\node (00) at (0,0) {$ \Dmod(\Bun_G) ^\temp$};
\node (10) at (2.5,0) {$ I(G,P) ^\temp$.};
\node (01) at (0,1) {$ \Dmod(\Bun_G)$};
\node (11) at (2.5,1) {$I(G,P)$};
\path[->,font=\scriptsize,>=angle 90]
(00.east) edge node[above] {$\CT_P^\enh$}  (10.west); 
\path[->,font=\scriptsize,>=angle 90]
(01.east) edge node[above] {$\CT_P^\enh$} (11.west); 
\path[->>,font=\scriptsize,>=angle 90]
(01.south) edge node[right] {$\temp$} (00.north);
\path[->>,font=\scriptsize,>=angle 90]
(11.south) edge node[right] {$\temp$} (10.north);
\end{tikzpicture}
\end{equation}

\begin{mainthm}[Automorphic gluing for $G=SL_2$] \label{mainthm:aut-gluing}
When $G = SL_2$ and $P=B$, the above diagram is a fiber square.
\end{mainthm}

\sssec{}

In other words, the natural functor
$$
\gamma:
\Dmod(\Bun_{SL_2})
\longto
\Dmod(\Bun_{SL_2})^\temp
\utimes_{ I(SL_2, B) ^\temp}
I(SL_2, B),
$$
induced by temperization and enhanced constant term, is an equivalence. This implements the original rough idea of Section \ref{sssec:rough}: in short, to make the statement precise, we need to use the enhanced version of $\CT_B$ and take gluing into account.

\begin{rem}
The above equivalence ought to correspond, under geometric Langlands, to the equivalence
$$
\gamma^{\spec}:
\ICoh_{\Nch}(\LS_{\Gch})
\xto{ \;\; \simeq \;\;}
\QCoh(\LS_{\Gch})
\utimes_{
\QCoh \bigt{ (\LS_{\Gch})^\wedge_ {\LS_{\check{B}}} }
}
\ICoh_0 \bigt{ (\LS_{\Gch})^\wedge_{\LS_{\check{B}}} }
$$
with $\Gch = PGL_2$.
This equivalence was constructed and proven in \cite{epiga}, relying heavily on \citep{AG2}.
\end{rem}

\sssec{}

The proof of Theorem \ref{mainthm:aut-gluing} splits in two parts: fully faithfulness and essential surjectivity of $\gamma$. We will see that the fully faithfulness is a corollary of Theorem \ref{mainthm:star}. On the other hand, the essential surjectivity requires the following result, which might be of interest in its own right:

\begin{mainthm} \label{mainthm:rank1}
For $G=SL_2$, let $\Y$ be one of $\Gr_G$, $\Bun_G$, $\Bun_G^{B \ggen}$. Denote by $p_\Y: \Y \to \pt$ the structure map.
Then an object $\F \in \Dmod(\Y)$ is tempered if and only if $(p_\Y)_!(\F) \simeq 0$.
\end{mainthm}

\begin{rem}
The above result was stated as a conjecture in \cite[Remark 1.2.2]{omega}.
\end{rem}

\ssec{Structure of the paper}

In Section \ref{sec:temp objects}, we discuss some background material and then prove our various characterizations of temperedness.
In Section \ref{sec:proofs}, we prove Theorem \ref{mainthm:Hecke} and deduce Theorem \ref{mainthm:star} from it.
In Section \ref{sec:SL2}, we study $\Sph_{SL_2}^\antitemp$ and prove Theorem \ref{mainthm:rank1}. In Section \ref{sec:aut-gluing}, we use Theorem \ref{mainthm:star} and  Theorem \ref{mainthm:rank1} together to deduce the automorphic gluing theorem. Finally, in Section \ref{sec:preview}, we briefly explain how to modify Gaitsgory's outline (\cite{Outline}) of the proof of Geometric Langlands for $G=GL_2$ to cover the case of $G=SL_2$: the new ingredients at play are the automorphic gluing theorem and the Deligne-Lusztig duality on the spectral side.

\ssec{Notations}

Our notations are in line with those used in \cite{omega}. In particular, we invite the reader to consult \cite[Section 2]{omega} for a comprehensive review.

We remind that we use the cohomological conventions for t-structures: for $\C$ a DG category equipped with a t-structure, $\C^{\leq 0}$ (the full subcategory of connective objects) is left orthogonal to $\C^{\geq 1}$ (at the level of the triangulated category underlying $\C$). 
A functor $f: \C \to \D$ between DG categories with t-structures is said to be left t-exact if $f(\C^{\geq 0})\subseteq \D^{\geq 0}$ and right t-exact if $f(\C^{\leq 0})\subseteq \D^{\leq 0}$.

\ssec{Acknowledgements}

I would like to thank Tony Scholl for asking about the Ramanujan conjecture in geometric Langlands (Cambridge, Feb 2020), Sam Raskin for a suggestion that helped simplify the proof of Theorem \ref{mainthm:rank1}, Dennis Gaitsgory for countless generous explanations throughout the years.

\sec{Tempered and anti-tempered objects} \label{sec:temp objects}

In this section, we discuss the notions of tempered and anti-tempered objects from the point of view of the t-structures. This point of view seems to be new: it did not appear in the previous treatments (\cite[Section 12]{AG1} and \cite{omega}) of $\Dmod(\Bun_G)^\temp$.

\ssec{Some standard t-structures}

Here we review certain definitions given in \cite{ICoh}, \cite{finiteness} and \cite{AG1}.
We start with the definition of the natural t-structure on the DG category of ind-coherent sheaves in the quasi-smooth case. Then we pass to the discussion of t-structures on DG categories of D-modules on stacks and, afterwards, on ind-stacks.

\sssec{}

Our conventions on stacks follow those of \cite{DG-cptgen} and \cite{finiteness}, namely we will consider algebraic stacks that are locally QCA. The \virg{local QCA}, as opposed to just \virg{QCA}, is allowed to accommodate for $\Bun_G$.

\sssec{}

Let $\Y$ be a quasi-smooth QCA derived stack: in practice, we will be interested in only one example: $\LSGch:= \LSGch(X)$. Attached to $\Y$ is its stack of singularities $\Sing(\Y) := H^{-1}(T^* \Y)$.
Given any conical closed subset $M \subseteq \Sing(\Y)$, consider the DG category $\ICoh_M(\Y)$ of \emph{ind-coherent sheaves with singular support contained in $M$}. This is a full subcategory of $\ICoh(\Y)$. If $M$ contains the zero section of $\Sing(\Y)$, then there is a colocalization
\begin{equation}
\nonumber
\begin{tikzpicture}[scale=1.5]
\node (a) at (0,1) {$\Xi_{0 \to M}:\QCoh(\Y)$};
\node (b) at (3,1) {$\ICoh_M(\Y):\Psi_{0 \to M}.$};
\path[right hook ->,font=\scriptsize,>=angle 90]
([yshift= 1.5pt]a.east) edge node[above] {$ $} ([yshift= 1.5pt]b.west);
\path[->>,font=\scriptsize,>=angle 90]
([yshift= -1.5pt]b.west) edge node[below] {$ $} ([yshift= -1.5pt]a.east);
\end{tikzpicture}
\end{equation}

\sssec{} 

Consider the standard t-structure on $\QCoh(\Y)$: an object $\F$ belongs to $\QCoh(\Y)^{\leq 0}$ iff its pullback to some (equivalently, any) atlas is connective.
By reducing to affine DG schemes and then to $\Vect$, one checks the well-known fact that $\QCoh(\Y)^{\leq - \infty} \simeq 0$.

\sssec{} \label{sssec:icoh t-str}

Using the above t-structure on $\QCoh(\Y)$, one defines a t-structure on $\ICoh_M(\Y)$ by setting
$$
\ICoh_M(\Y)^{\leq 0}
:=
\{
\F \in \ICoh_M(\Y) \; 
|
\; 
\Psi_{0 \to M}(\F) \in \QCoh(\Y)^{\leq 0}
\}.
$$
In other words, $\Psi_{0 \to M}$ is right t-exact by design. One quickly deduces that $\Xi_{0 \to M}$ is right t-exact too. It then follows that $\Psi_{0 \to M}$ is t-exact: indeed, being right adjoint to a right t-exact functor, it is left t-exact.
The following simple observation is crucial for us.

\begin{lem} \label{lem:ker(Psi)}
In the situation above, we have:
$$
\ICoh_M(\Y)^{\leq -\infty} \simeq \ker(\Psi_{0 \to M}).
$$ 
In particular, $\ICoh_M(\Y)^{\leq -\infty}$ is the right orthogonal of the inclusion $\Xi_{0 \to M}: \QCoh(\Y) \hto \ICoh_M(\Y)$.
\end{lem}

\begin{proof}
The inclusion  $\ker(\Psi_{0 \to M}) \subseteq \ICoh_M(\Y)^{\leq -\infty}$ is obvious. The opposite inclusion follows from the fact that $\QCoh(\Y)^{\leq -\infty} \simeq 0$.
\end{proof}

\sssec{}

Now let $\Y$ be an algebraic stack, not necessarily quasi-compact and not necessarily quasi-smooth. Recall that $\ICoh(\Y)$ is equipped with a t-structure defined in the same as above, that is, by requiring that $\Psi_\Y:\ICoh(\Y) \to \QCoh(\Y)$ be right t-exact. 
This t-structure is used in the following definition.

\sssec{}

Our next task is to discuss t-structures on categories of D-modules.
Let $\Y$ be an algebraic stack as above. We define a t-structure on $\Dmod(\Y)$ by declaring that the forgetful functor $\oblv_{\Y,R}: \Dmod(\Y) \to \ICoh(\Y)$ be left t-exact.

\begin{lem} \label{lem:Dmod t-properties}
Let $f: \X \to \Y$ be a map of locally QCA algebraic stacks. We have:
\begin{enumerate}
\item
If $f = j$ is an open embedding, then $j^*: \Dmod(\Y) \tto \Dmod(\X)$ is t-exact.
\item
If $f = i$ is a closed embedding, then $i_*: \Dmod(\X) \hto \Dmod(\Y)$ is t-exact.
\item
The t-structure on $\Dmod(\Y)$ is Zariski-local.
\item
If $f$ is smooth of dimension $d$, then $f^![-d]:  \Dmod(\Y) \to \Dmod(\X)$ is t-exact.
\item
If $f$ is affine, then $f_*:  \Dmod(\X) \to \Dmod(\Y)$ is right t-exact.
\item
$\Dmod(\Y)^{\leq -\infty} \simeq 0$.
\end{enumerate}
 
\end{lem}

\begin{proof}
Assertion (1) is obvious: $j^*$ and $j_*$ are both left t-exact by definition; the latter facts implies by adjunction that $j^*$ is also right t-exact. Assertion (3) now follows from (1) via the argument of \cite[Lemma 7.8.7]{ker-adj}.
If $f$ is a smooth atlas, assertion (4) follows from \cite[Lemma 6.3.7]{finiteness}.
Then the general case of (4), as well as (2) and (5), all reduce to the case of schemes where they are well-known. 
Finally, for (6), we can proceed as follows: by (4), it suffices to prove that $\Dmod(U)^{\leq - \infty} \simeq 0$ for any scheme $U$ locally of finite type. By (3), this can be checked Zariski-locally, and so we are reduced to the affine case. Then, by (2), we can embed $\Dmod(U)$ in a t-exact manner into $\Dmod(\AA^n)$ for some $n$. Hence it remains to show that $\Dmod(\AA^n)^{\leq - \infty} \simeq 0$. For this, note that the forgetful functor $\Dmod(\AA^n) \to \QCoh(\AA^n) \xto{H^*(\AA^n,-)} \Vect$ is conservative and t-exact up to a finite shift.
\end{proof}

\sssec{}

Now we extend the above definitions to DG categories of D-modules on ind-stacks.

\begin{defn}
An ind-stack is a filtered colimit of locally QCA stacks under closed embeddings.
\end{defn}

\sssec{}

For an ind-stack $\Y \simeq \colim_j \Y_j$ as defined above, we have $\Dmod(\Y) \simeq \colim_j \Dmod(\Y_j)$, as well as $\Dmod(\Y) \simeq \lim_j \Dmod(\Y_j)$.
The colimit is taken along the structure pushforward functors, the limit under their right adjoints (the $!$-pullback functors).

We define a t-structure on $\Dmod(\Y)$ by declaring $\Dmod(\Y)^{\leq 0}$ to be the full subcategory generated by the essential images of the inclusions $\Dmod(\Y_j)^{\leq 0} \subseteq \Dmod(\Y)$ for all $j$.
These t-structures enjoy properties similar to those of Lemma \ref{lem:Dmod t-properties}.

\begin{rem}
As an example, this construction yields a t-structure on the DG category $\Dmod(\Hecke_{G,x}^{\glob})$ of D-modules on the ind-stack $\Hecke_{G,x}^{\glob}$.
\end{rem}

\sssec{}

It is clear that the above t-structure on $\Dmod(\Y)$ agrees with the usual one when $\Y$ happens to be an ind-scheme, see \cite[Section 2.4.2]{omega} and references therein.
Accordingly, we have t-structures on $\Dmod(\Gr_{G,x})$ and on $\Dmod(G[\Sigma])$. We would like to have a t-structure on $\Sph_G \simeq \Dmod(\Hecke_{G,x}^\loc)$ as well. Since $\Hecke_{G,x}^\loc$ is not an ind-stack, this is not covered by the above paradigm. Therefore, we give a definition \virg{by hands}:

\begin{defn} \label{def:Sph t-str}

We put a t-structure on $\Sph_G := \Dmod(\Gr_G)^\GO$ by declaring the forgetful functor 
$$
\oblv^\GO: \Sph_G \to \Dmod(\Gr_G)
$$
to be right t-exact: an object $\F \in \Sph_G$ is connective iff so is $\oblv^\GO(\F)$.

\end{defn}

\begin{lem}
The functor $\oblv^\GO: \Sph_G \to \Dmod(\Gr_G)$ is t-exact with respect to the above t-structures.
\end{lem}

\begin{proof}
We need to prove that $\oblv^\GO$ is left t-exact. We can exhaust $\Gr_G$ with $\GO$-invariant subschemes $Y_n$ with the following property: the $\GO$-action on $Y_n$ factors through a finite dimensional quotient group $H_n$ with pro-unipotent kernel. Define a t-structure on $\Dmod(Y_n)^\GO$ by declaring the functor \begin{equation} \label{eqn:partial-oblvGO-right}
\oblv^{\GO}: \Dmod(Y_n)^\GO \to \Dmod(Y_n)
\end{equation} 
to be right t-exact. 
Note that the fully faithful embedding $i_!: \Dmod(Y_n)^\GO \hto \Dmod(\Gr_G)^\GO$ is right t-exact, and so its right adjoint is left t-exact.
Thus, by the definition of the t-structure on $\Dmod(\Gr_G)$, it remains to prove that \eqref{eqn:partial-oblvGO-right} is left t-exact.
By pro-unipotence, we have a natural equivalence
$$
\Dmod(Y_n)^\GO 
\simeq 
\Dmod(Y_n)^{H_n}
\simeq
\Dmod(Y_n/H_n).
$$
Assertion (4) of Lemma \ref{lem:Dmod t-properties} shows that the t-structure on $\Dmod(Y_n)^\GO$ corresponds, under the above equivalence, to a shift by $\dim(H_n)$ of the usual t-structure on $\Dmod(Y_n/H_n)$. In particular, the pullback functor along $Y_n \tto Y_n/H_n$ is t-exact.
\end{proof}

\begin{lem} \label{lem:filt}
The above t-structure on $\Sph_G$ is compatible with filtered colimits.
\end{lem}

\begin{proof}
By definition, a t-structure on $\C$ is compactible with filtered colimits if $\C^{\geq 0}$ is closed under filtered colimits. We know this is the case for $\C = \Dmod(\Gr_G)$: a way to see this is to observe that $\Dmod(\Gr_G)^{\leq 0}$ is generated by objects that are compact in $\Dmod(\Gr_G)$. Then the lemma follows from the fact that $\oblv^\GO: \Sph_G \to \Dmod(\Gr_G)$ is left t-exact, continuous and conservative.
\end{proof}

\begin{rem}[Not used in what follows]
It is easy to see that $\Sph_G^{\leq - \infty}$ is generated by $\Av_!^{\GO}(\F)$ for all $\F \in \Dmod(\Gr_G)$ infinitely connective and ind-holonomic. Indeed, any such object $\Av_!^{\GO}(\F)$ is infinite connective because $\Av_!^{\GO}$ is right t-exact. On the other hand, given $\A \in \Sph_G^{\leq -\infty}$, we obtain that
$$
\Av_!^{\GO} \oblv^\GO \A
\simeq
(p_G)_!(\omega_G)\otimes \A
\simeq
H_*(G) \otimes \A
$$
with $\oblv^\GO \A$ ind-holonomic and infinitely connective. (For the ind-holonomic part: note that any object of $\Sph_G$ is ind-holonomic.)
\end{rem}

\ssec{t-good and t-excellent maps}

The following definitions will play an important role in the proof of Theorem \ref{mainthm:star}.

\begin{defn}
Let $\X$ and $\X'$ be ind-stacks, so that their DG categories of D-modules are equipped with t-structures. We say that a map $f: \X \to \X'$ is
\begin{itemize}
\item
\emph{t-good} if $f^!$ sends $\Dmod(\X')^{\leq - \infty}$ to $\Dmod(\X)^{\leq - \infty}$;
\item
\emph{t-excellent} if the D-module pullback $f^!: \Dmod(\X') \to \Dmod(\X)$ is right t-exact up to a finite shift.
\end{itemize}
\end{defn}

\begin{rem}
We can extend the context of the above definitions to the case where $\X'= \Hecke_{G,x}^\loc$. In this case $\Dmod(\X') \simeq \Sph_{G,x}$ is equipped with the t-structure of Definition \ref{def:Sph t-str}.
\end{rem}

\sssec{}
A t-excellent map is in particular t-good. 
The opposite implication is false: the map $r$ of Section \ref{sssec:r maps of Hecke} is a counterexample.
A composition of t-good (respectively, t-excellent) maps is t-good (respectively, t-excellent).

\begin{example}
Let $\Y$ be an ind-scheme whose dualizing D-module is infinitely connective and let $y \in \Y(\kk)$. 
As mentioned, relevant examples of such $\Y$ are $\AA^\infty$, $\Gr_N$, $\Gr_G$, $G[\Sigma]$ for $G$ non-abelian and $\Sigma$ a smooth affine curve. 
Then the closed embedding $i_y: \pt \hto \Y$ determined by  $y \in \Y(\kk)$ is not t-good.
On the other hand, the structure map $\Y \to \pt$ is t-excellent (as well as t-good for a trivial reason).
Actually, in this simple case the map $p_\Y^!$ sends the entire $\Vect$ into $\Dmod(\Y)^{\leq - \infty}$: this is so because the t-structure on $\Vect$ is right-complete.
\end{example}

\begin{example}
Let $\Y \to Y \leftto Z$ be a diagram where $\Y$ is an ind-scheme of ind-finite type, while $Y,Z$ are quasi-compact schemes. Then the induced map $Z \times_Y \Y \to \Y$ is t-excellent.
Indeed, this reduces to the fact that any map between schemes of finite type is t-excellent.
\end{example}

\begin{example}

We now give a non-trivial non-example. Let $G$ be a non-abelian reductive group as usual, $B$ a Borel subgroup and $N$ its unipotent radical.
We claim that the map $j: \Gr_N \to \Gr_G$ is not t-good.

Recall that $j$ is a locally closed embedding, arising as the composition of an affine open embedding $\Gr_N \hto \ol\Gr_N$ with a closed embedding $\ol\Gr_N \hto \Gr_G$. In fact, the first map is the inclusion of a divisor complement, \cite{MV, Xinwen}.
It follows that $j_*$ is t-exact.
Now consider $j^-: \Gr_{N^-} \to \Gr_G$ and the object $j^-_*(\omega_{\Gr_{N^-}}) \in \Dmod(\Gr_G)$. In view of the above facts, this object is infinitely connective. If $j: \Gr_N \to \Gr_G$ were t-good, then $j^! (j^-)_*(\omega_{\Gr_{N^-}})$ would be infinitely connective too. However, the intersection of opposite semi-infinite orbits is a scheme $Z$ of finite type, see e.g. \cite{MV}. In particular, by the example above, $\omega_Z$ is bounded in the t-structure of $\Dmod(Z)$.

\end{example}

\begin{lem} \label{lem:smooth goodness}
The property of being t-good can be checked smooth-locally on the source.
\end{lem}

\begin{proof}
Let $g: \X \to \Y$ be a t-good map and $\rho: \U \to \X$ a smooth map. Since $\rho$ is in particular t-excellent by Lemma \ref{lem:Dmod t-properties}, we obtain that $\rho \circ g: \U \to \Y$ is t-good.
For the other direction, suppose that $\pi: \U \to \X$ is a smooth cover and that $\U \xto{\pi} \X \xto{f} \Y$ is t-good: we wish to show that $f$ is t-good.
To this end, take $\F \in \Dmod(\Y)^{\leq - \infty}$ and $\G \in \Dmod(\X)^{\geq m}$. Since $m \in \bbZ$ is arbitrary, it suffices to show that
$$
\HHom_{\Dmod(\X)}(f^!(\F), \G) \in \Vect^{>0}.
$$
By smooth descent for D-modules, we have
$$
\HHom_{\Dmod(\X)}(f^!(\F), \G)
\simeq
\lim_{[n] \in \bDelta}
\HHom_{\Dmod(\U_n)}
\Bigt{
( \pi_n) ^! f^!(\F),(\pi_n)^!\G
},
$$
where the $\pi_n$ are the structure maps of the Cech resolution. Now we observe that $f \circ \pi_n$ is t-good by assumption and so $(\pi_n) ^! f^!(\F)$ is infinitely connective. On the other hand, $(\pi_n)^!\G$ is eventually coconnective (this is because $\pi_n$ is t-excellent and $\G$ is eventually coconnective).
Hence $\HHom_{\Dmod(\X)}(f^!(\F), \G)$ is a totalization of vector spaces belonging to $\Vect^{>0}$, so it is itself in $\Vect^{>0}$.
\end{proof}

\begin{lem}
Consider a fiber square
\begin{equation} 
\nonumber
\begin{tikzpicture}[scale=1.5]
\node (00) at (0,0) {$ \Y$};
\node (10) at (1,0) {$ \Z$};
\node (01) at (0,.8) {$ \W$};
\node (11) at (1,.8) {$\X$};
\path[->,font=\scriptsize,>=angle 90]
(00.east) edge node[above] {$f$}  (10.west); 
\path[->,font=\scriptsize,>=angle 90]
(01.east) edge node[above] {$f'$} (11.west); 
\path[right hook ->,font=\scriptsize,>=angle 90]
(01.south) edge node[right] {$i'$} (00.north);
\path[right hook ->,font=\scriptsize,>=angle 90]
(11.south) edge node[right] {$i$} (10.north);
\end{tikzpicture}
\end{equation}
with $i$ a closed embedding or an affine open embedding.
If $f$ a t-good, then so is $f'$. 
\end{lem}

\begin{proof}
Since $i_*$ is t-exact and conservative (in fact, fully faithful), an object $\F \in \Dmod(\X)$ is infinitely connective if and only if so is $i_*(\F)$. Of course, the same holds true for $(i')_*$. Then the assertion follows immediately from base-change.
\end{proof}

\ssec{Derived Satake} \label{ssec:derived-satake}

Now we recall the derived Satake equivalence. We will observe that this is one instance (perhaps the only instance?) of Langlands duality where infinitely connective objects and anti-tempered objects coincide on the automorphic side.

\begin{thm} [Derived Satake theorem \cite{BF, AG1}] \label{thm:der Satake}
There is a monoidal equivalence
\begin{equation} \label{eqn:derived Sat}
\Sat_G:
\ICoh_{\Nch}((\pt \times_\gch \pt )/\Gch)
\xto{\;\; \simeq \;\;}
\Sph_G.
\end{equation}
\end{thm}

\begin{rem}
To save space, we set $\Omega \gch := \pt \times_\gch \pt$. For more information on this theorem, we refer the reader to  \cite{BF, AG1}, as well as to our discussion in \cite{omega}. 
\end{rem}

\sssec{}
 
Both sides of \eqref{eqn:derived Sat} are equipped with t-structures. On $\ICoh_{\Nch}((\pt \times_\gch \pt )/\Gch)$, we have the t-structure discussed in Section \ref{sssec:icoh t-str}: this applies since $\Omega \gch/\Gch \simeq \LSGch(\PP^1)$ is a quasi-smooth derived stack.
On $\Sph_G$, we use the t-structure of Definition \ref{def:Sph t-str}.

\begin{prop}
The equivalence $\Sat_G: \ICoh_{\Nch}(\Omega \gch/\Gch) \to
\Sph_G$ is t-exact.
\end{prop}

\begin{proof}
The crucial piece of structure we exploit is the compatibility of derived Satake with the usual (i.e., underived) geometric Satake equivalence. First of all, the underived Satake theorem (\citep{MV}) provides an equivalence
$$
\Sat'_G:
\Rep(\Gch)^\heartsuit
\xto{\;\;\simeq\;\;}
\Sph_G^{\heartsuit}
$$
of abelian categories. Next, observe that $\Rep(\Gch)^\heartsuit
\simeq \ICoh_{\Nch}(\Omega \gch/\Gch)^\heartsuit$ canonically.
Then the mentioned compatibility between $\Sat'_G$ and $\Sat_G$ amounts to the datum of a commutative diagram
\begin{equation} 
\nonumber
\begin{tikzpicture}[scale=1.5]
\node (00) at (0,0) {$ \ICoh_{\Nch}(\Omega \gch/\Gch)$};
\node (10) at (2,0) {$ \Sph_G$,};
\node (01) at (0,1) {$ \Rep(\Gch)^\heartsuit $};
\node (11) at (2,1) {$\Sph_G^\heartsuit$};
\path[->,font=\scriptsize,>=angle 90]
(00.east) edge node[above] {$\Sat_G$}  (10.west); 
\path[->,font=\scriptsize,>=angle 90]
(01.east) edge node[above] {$\Sat'_G$} (11.west); 
\path[right hook ->,font=\scriptsize,>=angle 90]
(01.south) edge node[right] {$ $} (00.north);
\path[right hook ->,font=\scriptsize,>=angle 90]
(11.south) edge node[right] {$ $} (10.north);
io\end{tikzpicture}
\end{equation}
with the right vertical arrows being the tautological embeddings. Alternatively, the left vertical arrow can be realized as the composition of $\Rep(\Gch)^\heartsuit \subseteq \Rep(\Gch)$ with the functor
$$
\xi:
\Rep(\Gch) 
\xto{\Delta_*^\ICoh}
\ICoh( \Omega \gch /\Gch)
\xto{\Psi_{\Nch \to \all}}
\ICoh_{\Nch}( \Omega \gch /\Gch).
$$
Here $\Delta: \pt/\Gch \hto \Omega \gch /\Gch$ is the natural closed embedding.
Now the claimed t-exactness can be proven by devissage, that is, by combining the following facts: $\xi$ is a t-exact functor whose essential image generates the target under colimits, the t-structure on $\ICoh_{\Nch}( \Omega \gch /\Gch)$ is right-complete, the t-structure on $\Sph_G$ is compatible with filtered colimits (as shown in Lemma \ref{lem:filt}). 
\end{proof}

\sssec{}

Using Lemma \ref{lem:ker(Psi)}, we obtain the following exact sequence of DG categories:
\begin{equation} \label{eqn:exact-seq Sat-spec}
\begin{tikzpicture}[scale=1.5]
\node (a) at (0,1) {$\QCoh(\Omega \gch/\Gch)$};
\node (b) at (2.7,1) {$\ICoh_{\Nch}(\Omega \gch/\Gch)$};
\node (c) at (5.4,1) {$\ICoh_{\Nch}(\Omega \gch/\Gch)^{\leq - \infty}$,};
\path[right hook ->,font=\scriptsize,>=angle 90]
([yshift= 1.5pt]a.east) edge node[above] {$\Xi_{0 \to \Nch}$} ([yshift= 1.5pt]b.west);
\path[->>,font=\scriptsize,>=angle 90]
([yshift= -1.5pt]b.west) edge node[below] {$\Psi_{0 \to \Nch}$} ([yshift= -1.5pt]a.east);
\path[->>,font=\scriptsize,>=angle 90]
([yshift= 1.5pt]b.east) edge node[above] {$\pi$} ([yshift= 1.5pt]c.west);
\path[right hook ->,font=\scriptsize,>=angle 90]
([yshift= -1.5pt]c.west) edge node[below] {$\iota$} ([yshift= -1.5pt]b.east);
\end{tikzpicture}
\end{equation}
where $\iota$ is the natural inclusion of the full subcategory of infinitely connective objects and $\pi$ is the natural projection given by the formula $\F \mapsto \coker (\Xi \Psi \to \id)$.

\sssec{}

Now, we \emph{define} the exact sequence
\begin{equation} \label{eqn:exact-seq Sat-aut}
\begin{tikzpicture}[scale=1.5]
\node (a) at (0,1) {$\Sph_G^\temp$};
\node (b) at (1.8,1) {$\Sph_G$};
\node (c) at (4,1) {$ \Sph_G^{\antitemp}$};
\path[right hook ->,font=\scriptsize,>=angle 90]
([yshift= 1.5pt]a.east) edge node[above] {$ $} ([yshift= 1.5pt]b.west);
\path[->>,font=\scriptsize,>=angle 90]
([yshift= -1.5pt]b.west) edge node[below] {$\temp$} ([yshift= -1.5pt]a.east);
\path[->>,font=\scriptsize,>=angle 90]
([yshift= 1.5pt]b.east) edge node[above] {$\antitemp$} ([yshift= 1.5pt]c.west);
\path[right hook ->,font=\scriptsize,>=angle 90]
([yshift= -1.5pt]c.west) edge node[below] {$ $} ([yshift= -1.5pt]b.east);
\end{tikzpicture}
\end{equation}
to be the one corresponding to \eqref{eqn:exact-seq Sat-spec} under derived Satake.
Since the equivalence $\Sat_G$ is t-exact, we obtain:

\begin{cor} \label{cor:antitemp=t-nil}
The two full subcategories $\Sph_G^\antitemp \subseteq \Sph_G$ and $\Sph_G^{\leq - \infty} \subseteq \Sph_G$ agree.
\end{cor}

\ssec{Tempered and anti-tempered objects} \label{ssec:tempered-objects}

\sssec{}

Now let $\C$ be a DG category equipped with an action of $\Sph_G$. 
We set:
$$
\C^\temp
:= \Sph_G^{\temp} \usotimes{\Sph_G} \C
$$
$$
\C^\antitemp
:= \Sph_G^{\antitemp} \usotimes{\Sph_G} \C.
$$
Since the four functors appearing in \eqref{eqn:exact-seq Sat-aut} are all $\Sph_G$-linear, we can tensor \eqref{eqn:exact-seq Sat-aut} with $\C$ over $\Sph_G$ and obtain a new exact sequence 
\begin{equation} \label{eqn:exact-seq Sat-C}
\begin{tikzpicture}[scale=1.5]
\node (a) at (0,1) {$\C^\temp$};
\node (b) at (1.4,1) {$\C$};
\node (c) at (3.2,1) {$\C^\antitemp$.};
\path[right hook ->,font=\scriptsize,>=angle 90]
([yshift= 1.5pt]a.east) edge node[above] {$ $} ([yshift= 1.5pt]b.west);
\path[->>,font=\scriptsize,>=angle 90]
([yshift= -1.5pt]b.west) edge node[below] {$\temp$} ([yshift= -1.5pt]a.east);
\path[->>,font=\scriptsize,>=angle 90]
([yshift= 1.5pt]b.east) edge node[above] {$\antitemp$} ([yshift= 1.5pt]c.west);
\path[right hook ->,font=\scriptsize,>=angle 90]
([yshift= -1.5pt]c.west) edge node[below] {$ $} ([yshift= -1.5pt]b.east);
\end{tikzpicture}
\end{equation}
So, $\C^\temp$ and $\C^{\antitemp}$ are full subcategories of $\C$, with the former left orthogonal to the latter. Accordingly, an object $c \in \C$ is tempered if and only if the full subcategory $\Sph_G^{\antitemp} \subseteq \Sph_G$ acts on $c$ by zero. Thanks to Corollary \ref{cor:antitemp=t-nil}, we can replace $\Sph_G^{\antitemp}$ with $\Sph_G^{\leq - \infty}$:

\begin{cor} \label{cor:key}
An object $c \in \C$ is \emph{tempered} if
$\A \star c \simeq 0$ for any $\A \in \Sph_G^{\leq - \infty}$.
\end{cor}

\sssec{}

Now recall the Hecke action of $\Sph_G$ on $\Dmod(\Bun_G)$: this is defined after fixing a point $x \in X$. We sometimes denote the Hecke action at $x$ by $\star_x$, and for emphasis we sometimes write $\Sph_{G,x}$ in place of $\Sph_G$.
The following statement, which we might take as a definition, is the link between tempered D-modules on $\Bun_G$ and infinitely connective objects that we were alluding to in the introduction. 

\begin{cor} \label{cor:key BunG}
An object $\F \in \Dmod(\Bun_G)$ is \emph{tempered} if
$\A \star_x \F \simeq 0$ for any $x \in X$ and any $\A \in \Sph_G^{\leq - \infty}$.
\end{cor}

\ssec{Spherical-Whittaker objects}

Here we express tempered objects in terms of Whittaker invariants. This characterization will not be used anywhere in the paper.
We will exhibit a collection of \emph{compact} generators of $\Sph_G^\temp$: these are the \emph{spherical-Whittaker objects}, defined below.

\sssec{} \label{sssec:notation Whit}

To fix the notations and our conventions on left and right actions, let $\Gr_G := \GO \backsl \GK$. This way, we see that $\Sph_G$ acts on $\Dmod(\Gr_G)$, while $\GK$ acts from the right.
In particular, $\NK$ acts on the right and we consider the Whittaker DG category $\Dmod(\Gr_G)^{\Whit}$.  Since the left $\Sph_G$-action on $\Dmod(\Gr_G)$ is compatible with the right $\GK$-action, it is clear that $\Dmod(\Gr_G)^{\Whit}$ retains the left action of $\Sph_G$.

\begin{rem}

To work with Whittaker invariants, we need some familiarity wiht the formalism of (loop) group actions on DG categories, see \cite{thesis} for an introduction to the theory. Many further advances and applications can be found in \cite{Sam, Winter}.

\end{rem}

\sssec{}

Any point $g \in \GK$ yields a point $[g]:= \GO g \in \Gr_G$. In particular, any coweight $\lambda \in \Lambda$ yields a point $[t^\lambda] \in \Gr_G$. 
The most important object $\W_0 \in \Dmod(\Gr_G)^{\Whit}$ is obtained by Whittaker-averaging the delta D-module at the unit point of $\Gr_G$. In formulas,
$$
\W_0 := \Av_!^{\Whit}(\delta_{[t^0]}) \in \Dmod(\Gr_G)^{\Whit}.
$$

\sssec{}

The \emph{geometric Casselman-Shalika formula} of \cite{FGV} shows that there is an equivalence
\begin{equation} \label{eqn:Cass-Shalika}
\Rep(\Gch) \simeq \Dmod(\Gr_G)^{\Whit},
\end{equation}
uniquely characterized by the following properties:
\begin{itemize}
\item
the trivial $\Gch$-representation goes to $\W_0$;
\item
the equivalence is compatible with derived Satake and the actions of $\ICoh_{\Nch}(\Omega \gch/\Gch)$ and $\Sph_G$ on the two sides.
\end{itemize}

\sssec{}

Following the same idea, consider the $\Sph_G$-linear functor
$$
\act_{\W_0}:
\Sph_G 
\longto
\Dmod(\Gr_G)^{\Whit}
$$
defined by the action of $\Sph_G$ on $\W_0 \in \Dmod(\Gr_G)^{\Whit}$.

\sssec{}

Let us determine the spectral counterpart of $\act_{\W_0}$ under derived Satake and \eqref{eqn:Cass-Shalika}. This is the unique functor $\ICoh_\N(\Omega \gch/\Gch)$-linear functor
$$
\ICoh_\N(\Omega \gch/\Gch)
\longto
\ICoh(\pt/\Gch) \simeq \Rep(\Gch)
$$
that sends the unit of $\ICoh_\N(\Omega \gch/\Gch)$ to the trivial $\Gch$-representation $\kk_0 \in \Rep(\Gch)$.
Denote by $\pi: \Omega \gch /\Gch \to \pt/\Gch$ is the obvious (proper) projection. Tautologically, the functor in question can be written as
\begin{equation} \label{eqn:a spectral functor}
 \pi_*^\ICoh \circ \Xi_{\Nch \to \all}
:
\ICoh_\N(\Omega \gch/\Gch)
\longto
\ICoh(\pt/\Gch) \simeq \Rep(\Gch)
\end{equation}
and alternatively as $\pi_* \circ \Psi_{0 \to \Nch}$.
This latter description implies:
\begin{cor}
Any object of $\Dmod(\Gr_G)^{\Whit}$ is tempered.
\end{cor}

\sssec{}

Observe now that, by construction, the functor $\act_{\W_0}$ can be rewritten as the composition
$$
\Av_!^{N(\OO) \to \Whit} 
\circ 
\oblv^{N(\OO) \to \GO}:
\Sph_G
\longto
\Dmod(\Gr_G)^{\Whit}.
$$
From this, we see that $\act_{\W_0}$ admits a continuous right adjoint explicitly given by the formula
\begin{equation} \label{eqn:from Whit to Sph}
\Av_*^{N(\OO) \to \GO}
\circ
\oblv^{N(\OO) \to \Whit} 
:
\Dmod(\Gr_G)^{\Whit}
\longto
\Sph_G.
\end{equation}
The latter functor lands in $\Sph_G^\temp$. Indeed, as observed, $\Dmod(\Gr_G)^{\Whit}$ consists only of tempered objects and the functor in question preserves temperedness, being evidently $\Sph_G$-linear.

\begin{prop} \label{prop:Sph-temp}
The essential image of the resulting functor $\xi: 
\Dmod(\Gr_G)^{\Whit}
\longto
\Sph_G^\temp$
generates the target under colimits. Morevoer, $\xi$ preserves compactness.
\end{prop}

\begin{proof}
Let us write down the corresponding functor on the spectral side and prove the claims there. We begin by computing the right adjoint of \eqref{eqn:a spectral functor}. Since the map $\pi: \Omega \gch /\Gch \to \pt/\Gch$ is proper, the functor $\pi^!$ is right adjoint to $\pi_*^\ICoh$. It follows that \eqref{eqn:from Whit to Sph} is, spectrally, the functor
$$
\Psi_{\Nch \to \all} \circ \pi^!: 
\ICoh(\pt/\Gch)
\longto
\ICoh_\N(\Omega \gch/\Gch).
$$
Obviously, this functor factors through $\QCoh(\Omega \gch/\Gch)$ as expected. Thus, the spectral counterpart of $\xi$ is simply the functor $\pi^*: \QCoh(\pt/\Gch) \to \QCoh(\Omega \gch/\Gch)$ up to tensoring with a shifted line bundle.
This immediately shows that $\xi$ preserves compactness. To conclude, it suffices to prove that the essential image of $\pi^*$ generates the target under colimits: this is clear from the affineness of the derived scheme $\Omega \gch$.
\end{proof}

\begin{cor}
The \emph{spherical-Whittaker sheaves} parametrized by $\lambda \in \Lambda^{\dom}$ and defined by the formula
$$
\WS_\lambda := \Av_*^\GO \Av_!^{\Whit}(\delta_{[t^\lambda]})
$$
are compact and generate $\Sph_G^\temp$.
\end{cor}

\begin{proof}
This holds true because the objects $\Av_!^{\Whit}(\delta_{[t^\lambda]})$ for all dominant $\lambda$'s are compact (obviously) and generate $\Dmod(\Gr_G)^{\Whit}$ by the Casselman-Shalika formula.
\end{proof}

\sssec{}

It follows  easily from Proposition \ref{prop:Sph-temp} that $\Sph_G^\temp$ is generated under colimits by the essential image of
$$
\Sph_G 
\xto{\F \mapsto \F \star \WS_0}
\Sph_G.
$$
In particular, we have:
\begin{cor}
Let $\C$ be a DG category equipped with an action of $\Sph_G$. An object $c \in \C$ is anti-tempered if and only if $\WS_0 \star c \simeq 0$.
\end{cor}

\begin{rem}
This point of view yields another proof of the anti-temperedness of $\omega_{\Bun_G}$, see \cite{omega}. Indeed, it suffices to prove that $\WS_0 \star \omega_{\Bun_G} \simeq 0$. This computation is left to the reader. 
\end{rem}

\sec{Proofs} \label{sec:proofs}

In this section, we show the implications
$$
\mbox{Theorem \ref{mainthm:Hecke}}
\implies
\mbox{Theorem \ref{mainthm:diagonal}}
\implies
\mbox{Theorem \ref{mainthm:star}}
$$
and then we prove Theorem \ref{mainthm:Hecke}.

\ssec{First reduction step}

Let us explain the implication $\mbox{Theorem \ref{mainthm:diagonal}}
\implies
\mbox{Theorem \ref{mainthm:star}}$.

\sssec{}

Fix $x \in X$ once and for all, and consider the Hecke action of $\Sph_G$ on $\Dmod(\Bun_G)$ at $x$. We usually omit the symbol $x$ from the notation and simply refer to this action as the Hecke action. 

\sssec{}  \label{sssec:3notions}

Observe that $\Bun_{G \times G} \simeq \Bun_G \times \Bun_G$ and that
\begin{equation} \label{eqn:Dmod on product}
\Dmod(\Bun_G \times \Bun_G) 
\simeq
\Dmod(\Bun_G) \otimes \Dmod(\Bun_G).
\end{equation}
This equivalence of DG categories follows from the dualizability of $\Dmod(\Bun_G)$, see \cite{DG-cptgen} for details.
Hence, we see that $\Dmod(\Bun_G \times \Bun_G)$ is equipped with three Hecke actions: the Hecke action of $\Sph_{G \times G}$, as well as the $\Sph_G$-actions on the left and right factors. Accordingly, objects of $\Dmod(\Bun_G \times \Bun_G)$ have three different notions of temperedness. These notions are really different in general (it is easy to give examples); however, by symmetry, they all coincide for the object $\Delta_*(\omega_{\Bun_G})$.

\sssec{}

For concreteness, in our proofs of Theorem \ref{mainthm:diagonal} and Theorem \ref{mainthm:star}, we will use the left $\Sph_G$-action on $\Dmod(\Bun_G \times \Bun_G)$. So, the precise restatement of Theorem \ref{mainthm:diagonal} that we will prove reads as follows.

\begin{thm} \label{thm:diagonal for left Sph-action}
The object $\Delta_{*}(\omega_{\Bun_G})\in \Dmod(\Bun_G \times \Bun_G) $ is tempered with respect to the left $\Sph_G$-action on the first factor.
\end{thm}

\sssec{} \label{sssec:Psid*}

Now we can show that Theorem \ref{mainthm:star} is a corollary of the above statement. Using \eqref{eqn:Dmod on product}, we regard $\Delta_{*}(\omega_{\Bun_G})$ as the kernel of a functor 
$$
\Dmod(\Bun_G)^\vee \to \Dmod(\Bun_G).
$$ 
This functor was introduced in \cite{DG-cptgen} under the notation $\Psid^{\on{naive}}$; we will denote it by $\Psid_*$ in accordance with our discussion in \cite{DL}. From this point of view, Theorem \ref{thm:diagonal for left Sph-action} asserts that there exists an isomorphism of functors:
$$
\Psid_* \simeq \temp \circ \Psid_*.
$$
In other words, this means that $\Psid_*$ lands in $\ltemp\Dmod(\Bun_G)$. Now, it remains to quote \cite[Lemma 2.1.9]{DL}, which identifies the essential image of $\Psid_*$ with $\Dmod(\Bun_G)^\stargen$.

\ssec{Second reduction step}

Now we prove the implication $\mbox{Theorem \ref{mainthm:Hecke}}
\implies
\mbox{Theorem \ref{thm:diagonal for left Sph-action}}$.

\sssec{} \label{sssec:reformulation of temp of Delta omega}

By Corollary \ref{cor:key} and the discussion of Section \ref{sssec:3notions}, the temperedness of $\Delta_*(\omega_{\Bun_G})$ is equivalent to proving that
$$
\A \star \Delta_*(\omega_{\Bun_G}) \simeq 0
$$ 
for any $\A \in \Sph_G^{\leq -\infty}$, where $\star$ denotes the left action of $\Sph_{G,x}$ on $\Dmod(\Bun_G \times \Bun_G)$ at the point $x \in X$.
We would like to describe the functor
$$
- \star \Delta_*(\omega_{\Bun_G})
:
\Sph_G \longto \Dmod(\Bun_G \times \Bun_G)
$$
explicitly. To do this, we need to embark on a small technical digression.

\sssec{}

When a map $f:\X \to \Y$ is ind-schematic, we can consider its \emph{renormalized push-forward} at the level of D-modules, see \cite[Section 5.1.6]{omega}. The definition is concrete terms goes as follows.  Write $\X = \colim_{j \in \J} \X_j$, with closed embeddings $\X_j \hto \X$ and schematic maps $f_j: \X_j \hto \X \xto{f} \Y$. Recall that
$$
\Dmod(\X) \simeq \colim_{j \in \J} \Dmod(\X_j),
$$
with the colimit taken along the structure pushforward functors $(i_{j \to j'})_*: \Dmod(\X_j) \hto \Dmod(X_{j'})$.
With these notations, the renormalized pushforward $f_{*,\ren}: \Dmod(\X) \to \Dmod(\Y)$ is the functor defined by 
$$
f_{*,\ren}
\simeq
\colim_{j \in \J} \; (f_j)_{*}.
$$
Thanks to Lemma \ref{lem:Dmod t-properties}, this functor is right t-exact as soon as $f$ is ind-affine. We will apply this general discussion to the following geometric situation:

\begin{lem} \label{lem:functor h}
The natural forgetful map $h: \Hecke_{G,x}^\glob \to \Bun_G \times \Bun_G$ is ind-schematic and ind-affine. 
\end{lem} 

\begin{proof}
The proof follows easily from \cite[Sections 2.7.2-2.7.5]{contract}.
\end{proof}

\begin{cor}
The functor $h_{*,\ren}: \Dmod(\Hecke_{G,x}^\glob) \to \Dmod(\Bun_G \times \Bun_G)$ is well-defined and right t-exact.
\end{cor}

\sssec{}

We can now unravel the functor
$$
- \star \Delta_*(\omega_{\Bun_G})
:
\Sph_G \longto \Dmod(\Bun_G \times \Bun_G).
$$
The result of the computation is the following.
\begin{lem}
We have:
\begin{equation} \label{eqn:funny functor}
- \star \, \Delta_*(\omega_{\Bun_G})
\simeq
h_{*,\ren} \circ r^!:
\Sph_G \longto \Dmod(\Bun_G \times \Bun_G)
\end{equation}
where the maps $h$ and $r$ (introduced above and in Section \ref{sssec:r maps of Hecke} respectively) form the correspondence
$$
\Hecke_{G,x}^{\loc} :=
\GO \backsl \GK / \GO
\xleftarrow{\;\; r \; \;}
\Hecke_{G,x}^\glob
\xto{\;\; h \; \;}
\Bun_G \times \Bun_G.
$$
\end{lem}

\begin{proof}
Consider the following commutative diagram and observe that the central square is cartesian.
\begin{equation} 
\nonumber
\begin{tikzpicture}[scale=1.5]
\node (00) at (0,0) {$ \Hecke_G^{\glob} \times \Bun_G  $};
\node (10) at (3,0) {$ \Bun_G \times \Bun_G$};
\node (01) at (0,1) {$\Hecke_G^{\glob}$};
\node (11) at (3,1) {$ \Bun_G$};
\node (21) at (4.5,1) {$\pt$};
\node (-10) at (-3,0) {$\Bun_G \times \Bun_G$};
\node (0-1) at (0,-1) {$\Hecke_G^{\loc}$};
\path[->,font=\scriptsize,>=angle 90]
(11.east) edge node[above] {$p_{\Bun_G}$} (21.west); 
\path[->,font=\scriptsize,>=angle 90]
(00.east) edge node[above] {$h^\to \times \id$} (10.west); 
\path[->,font=\scriptsize,>=angle 90]
(01.east) edge node[above] {$h^\to$} (11.west); 
\path[->,font=\scriptsize,>=angle 90]
(00.west) edge node[above] {$h^\leftto \times \id$} (-10.east); 
\path[->,font=\scriptsize,>=angle 90]
(01.south) edge node[right] {$(\id, h^\to)$} (00.north);
\path[->,font=\scriptsize,>=angle 90]
(11.south) edge node[right] {$\Delta$} (10.north);
\path[->,font=\scriptsize,>=angle 90]
(00.south) edge node[right] {$r \circ p_1$} (0-1.north);
\path[->,font=\scriptsize,>=angle 90]
(01.west) edge node[above] {$h$} (-10.north);
\end{tikzpicture}
\end{equation}
After unwinding the definition of the Hecke action, a base-change along the central fiber square shows that
\begin{equation} \label{eqn:funny Hecke}
- \star \Delta_*(\omega_{\Bun_G})
\simeq
(h^\leftto \times \id)_! \circ (\id, h^\to)_{*} \circ r^!.
\end{equation}
Since $h^\leftto$ is ind-proper, the functor $(h^\leftto \times \id)_!$ agrees with the renormalized de Rham push-forward. On the other hand, $(\id, h^\to)$ is schematic and so its de Rham and renormalized push-forwards agree. Hence the two push-forward functors appearing in \eqref{eqn:funny Hecke} compose to give the renormalized push-forward along $h$.
\end{proof}

\sssec{}

The above computation shows that Theorem \ref{mainthm:diagonal} (in the incarnation of Theorem \ref{thm:diagonal for left Sph-action}) is equivalent to proving that $h_{*,\ren} \circ r^!: \Sph_G \to \Dmod(\Bun_G \times \Bun_G)$ annihilates $\Sph_G^{\leq - \infty}$.
Since $h_{*,\ren}$ is right t-exact and since $\Dmod(\Bun_G \times \Bun_G)^{\leq - \infty} \simeq 0$, it suffices to prove that the map
$$
r: \Hecke_{G,x}^\glob
\longto
\Hecke_{G,x}^\loc
$$
is t-good. This is precisely the statement of Theorem \ref{mainthm:Hecke}.

\ssec{The semisimple case} \label{ssec:ss case}

We now prove Theorem \ref{mainthm:Hecke} in the case $G$ is semisimple. The general reductive case, which is only slightly more complicated, is treated afterwards in Section \ref{ssec:red case}.

\sssec{}

Let  $\A \in \Sph_G^{\leq - \infty}$.
We need to verify that $r^!(\A) \in \Dmod(\Hecke_{G,x}^\glob)$ is infinitely connective, too. In view of Lemma \ref{lem:smooth goodness}, this can be checked smooth-locally on $\Hecke_{G,x}^\glob$. 
Then, by taking open subschemes of an atlas of $\Bun_{G \times G}$, we see that it suffices to check the following. 

\begin{lem}
For any map $S \to \Bun_{G \times G}$ from an affine scheme $S$, the composition
\begin{equation} \label{eqn:composite map}
\Hecke_{G,x}^\glob \utimes_{\Bun_{G \times G}} S
\longto
\Hecke_{G,x}^\glob
\longto
\Hecke_{G,x}^\loc
\end{equation}
is t-good. 
\end{lem}

\sssec{}

The map $S \to \Bun_{G \times G}$ classifies two $G$-bundles $E_1$ and $E_2$ on $X_S$.
By \cite[Theorem 3]{DS}, up to base-changing along a suitable \'etale map $S' \to S$ (and then renaming $S'$ with $S$), we may assume that $E_1$ and $E_2$ are both trivial on $X^*_S := X^* \times S$.
In this case, 
$$
\Hecke_{G,x}^\glob \utimes_{\Bun_{G \times G}} S
\simeq
G[X^*] \times S.
$$

\sssec{}

Up to replacing $S$ with a suitable \'etale cover, we can also assume that $E_1$ and $E_2$ have been trivialized on $\wh D_{x, S}$. Now, let us invoke the Beauville-Laszlo theorem \cite{BL}: letting $R:= H^0(S,\O_S)$, we obtain that $E_1$ and $E_2$ are determined by two elements of $G(R\ppart)$, which we name $\alpha$ and $\beta^{-1}$ respectively.
Unraveling the definition, \eqref{eqn:composite map} is the map
\begin{equation} \label{map:simplified}
G[X^*] \times S
\longto
\Hecke_{G,x}^\loc
\end{equation}
$$
(\phi, f) \mapsto \GO \alpha \phi \beta \GO.
$$

\sssec{}
The forgetful functor $\oblv^\GO: \Sph_{G,x} \to \Dmod(\Gr_{G,x})$ is right t-exact by construction, hence it suffices to prove that the map
\begin{equation} \label{map:to Gr}
G[X^*] \times S
\longto
\Gr_{G,x}
\end{equation}
$$
(\phi, f) \mapsto \alpha \phi \beta \GO
$$
is t-good. In turn, the latter map factors as
$$
G[X^*] \times S
\longto
\Gr_{G,x} \times S
\longto
\Gr_{G,x}.
$$

\sssec{} \label{sssec:simplification}

Since the second map above is obviously t-excellent, we just focus on the left map. This is the same as working over the ground ring $R$ and on the curve $X_S$. 
Instead of burdening the notation, we can pretend for the remainder of Section \ref{ssec:ss case} that our ground field $\kk$ is actually the ground ring $R$; so that $\alpha$ and  $\beta$ are $\kk$-points of $G(\bbK)$. This way, we can get on with our usual notation. In other words, it remains to prove that the map
$$
G[X^*]
\longto
\Gr_{G,x},
\hspace{.4cm}
\phi \mapsto \alpha \phi \beta \GO
$$
of ind-schemes is t-good for any $\alpha, \beta \in G(\kk\ppart)$.

\sssec{}

For $g \in G(\kk\ppart)$, denote by $m_g$ the action map on $\Gr_{G,x}$. We claim that $(m_g)^!$ is right t-exact. Indeed, $(m_g)^! \simeq (m_{g^{-1}})_*$ and the assertion follows from the fact that $m_{g^{-1}}$ is a closed embedding.
Then it suffices to prove that the map
$$
r_\beta:
G[X^*]
\longto
\Gr_{G,x}
\hspace{.4cm}
\phi 
\mapsto 
\phi \beta \GO
$$
is t-good.

\begin{lem} 
For any $\beta \in G(\kk\ppart)$, the map $r_\beta$ is t-excellent.
\end{lem}

\begin{proof}
Let $E$ be the $G$-bundle on $X$ determined by $\beta$ and consider the induced map $i_E: \pt \to \Bun_G(X)$. Since $E$ is by construction trivialized on $X^*$, it is clear that $r_\beta$ is the base-change of $i_{E}$ along the struture projection $\Gr_{G,x} \to \Bun_G$. 

Hence, it suffices to prove that $i_{E}: \pt \to \Bun_G$ is t-excellent.
More generally, it is easy to see that any map from $T \to \Bun_G$ from an affine scheme $T$ (of finite type) is t-excellent. In fact, by working smooth-locally, we reduce to the case of a map of affine schemes of finite type. 
\end{proof}

\ssec{The reductive case} \label{ssec:red case}

We now prove Theorem \ref{mainthm:Hecke} in the general case when $G$ is reductive but not necessarily semisimple. The argument is similar to the one above, the only difference is that we will need to puncture more points on $X$.

\sssec{}

As before, it suffices to check that, for any map $S \to \Bun_{G \times G}$ from an affine scheme $S$, the composition
\begin{equation}  \label{eqn:composite-map-2}
\Hecke_{G,x}^\glob \utimes_{\Bun_{G \times G}} S
\longto
\Hecke_{G,x}^\glob
\longto
\Hecke_{G,x}^\loc
\end{equation}
is t-good. 

\sssec{}

Again, denote by $E_1$ and $E_2$ the two $G$-bundles on $X_S$ classified by $S \to \Bun_{G \times G}$.
This time we use \cite[Theorem 2]{DS} and the fact that the topology of $X_S$ is generated by divisor complements. Then, by working \'etale-locally on $S$, we may assume that we are in the following situation.

There are a nonempty finite set $I$ and a map $\ul x: S \to X^I$ such that $E_1$ and $E_2$ are trivial on $X - D_{\ul x}$. Without loss of generality, we can also assume that one of the maps comprising $\ul x$ is the constant map with value $x$.

\sssec{}

Consider now the (finite) stratification of $X^I$ determined by the various diagonals:
$$
X^I \simeq \bigsqcup_{ q: I \tto J} X^{J, \disj}.
$$
By pulling back along $\ul x$, this induces a stratification of $S = \sqcup_q S_q$ and then a stratification
$$
\Hecke_{G,x}^\glob \utimes_{\Bun_{G \times G}} S
\simeq
\bigsqcup_{ q: I \tto J}
\Hecke_{G,x}^\glob \utimes_{\Bun_{G \times G}} S_q.
$$
Unraveling the construction, the restriction of $\ul x$ to $S_q$ consists of $|J|$ maps to $X$ that have disjoint graphs.

\sssec{}

Clearly, we can check that a map is t-good strata-wise on the source.
Hence after renaming $S_q$ with $S$, we are back to the map \eqref{eqn:composite-map-2}, but with the additional assumption that the two $G$-bundles are trivial on the complement of a disjoint union of graphs $ \bigcup_{j \in J}D_{x_j} \subset X_S$. Set $d := |J|-1$ and fix a bijection $J \simeq \{0, \ldots, d\}$ such that $D_{x_0}$ is the constant divisor $D_x$ associated to our point $x \in X$.
Note that $d = 0$ is allowed: in that case $D_{x}$ is the only divisor and the two $G$-bundles are trivial on $X^*_S$.

\sssec{}

Up to replacing $S = \Spec(R)$ with a suitable \'etale cover, we may assume that $E_1$ and $E_2$ have been trivialized on the formal tubular neighbourhoods of the $D_{x_j}$, for all $j \geq 0$. 
Then $E_1$ and $E_2$ are determined by $J$-tuples
$$
(\alpha_j) \in 
\prod_{j = 0 }^{d}
G(\bbK_{x_j})(R),
\hspace{.4cm}
(\beta_j) \in 
\prod_{j = 0 }^d
G(\bbK_{x_j})(R),
$$
respectively.

\sssec{}

Unraveling the constructions, we see that $\Hecke_{G,x}^\glob \utimes_{\Bun_{G \times G}} S$ is the $S$-indscheme that sends $f:\Spec(R') \to S$ to the set
$$
\Big\{
\phi: (X_S - \bigcup_{j \geq 0}D_{x_j}) \times_S S' \to G \; \Big | \; \mbox{  $\alpha_j \cdot \phi \cdot \beta_j \in G(\OO_{x_j})(R')$ for all $j \geq 1$}
\Big\}.
$$
In plain words: $\phi$ determines an isomorphism between the restrictions of $E_1$ and $E_2$ to $(X_S - \bigcup_{j \geq 0}D_{x_j}) \times_S S'$ and we require this isomorphism to extend across the $D_{x_1,S'}, \ldots, D_{x_d,S'}$ so as to yield an isomorphism on $(X^*)_{S'}$.

\sssec{}

From this point of view, the map 
$$
\Hecke_{G,x}^\glob \utimes_{\Bun_{G \times G}} S \to \Hecke_{G,x}^\loc
$$
is easily described: at the level of $S'$-points, it sends $\phi$ as above to $G(\OO_{x_0}) \alpha_0 \phi \beta_0 G(\OO_{x_0})$.

\sssec{}

Before proceeding, let us simplify the notation. 
Reasoning as in Section \ref{sssec:simplification}, we see that we can rename $R$ by $\kk$ and pretend that we are working over $\kk$ as usual. Thus, have disjoint $\kk$-points $x_1, \ldots, x_d$ in $X^* = X- \{x_0\}$. Setting $\Sigma = X - \{x_0, x_1, \ldots, x_d\}$, the map under investigation is the composition
$$
r':
\Y
\hto
G[\Sigma]
\xto{\phi \mapsto \alpha_0 \cdot \phi \cdot \beta_0}
\Hecke_{G,x}^{\loc},
$$
where $\Y$ is the closed sub-ind-scheme of $G[\Sigma]$ defined by
$$
\Y
:=
\Big\{
\phi \in G[\Sigma] \; | \; \mbox{  $\alpha_j \cdot \phi \cdot \beta_j \in G(\bbO_{x_j})$ for all $j \geq 1$}
\Big\}.
$$

\sssec{}

We wish to show that $r'$ is t-excellent. To this end, observe that $r': \Y \to  \Hecke_{G,x_0}^\loc$ is the base-change of 
$$
r'':
G[\Sigma] \longto \prod_{j \geq 0} \Hecke_{G,x_j}^\loc,
$$
$$
\phi 
\mapsto
(\alpha_j \cdot \phi \cdot \beta_j)_{j \geq 0}
$$
along the closed embedding $\Hecke_{G,x_0}^\loc \hto \prod_{j \geq 0} \Hecke_{G,x_j}^\loc$ that inserts the unit at $x_1, x_2, \ldots, x_d$.

\sssec{}

Since push-forwards along closed embeddings are right t-exact and conservative, it suffices to prove that $r''$ is t-excellent. This map factors as
$$
G[\Sigma]
 \xto{\; \phi \mapsto [\alpha_j \phi] \;}
  \prod_{j \geq 0} \Gr_{G,x_j}
  \xto{\; [g] \mapsto [g \beta_j] \;}
  \prod_{j \geq 0} \Gr_{G,x_j}
   \xto{\; \mathit{quotient} \; }
 \prod_{j \geq 0} \Hecke_{G,x_j}^\loc.
$$

\begin{rem}
To clarify our notation, following Section \ref{sssec:notation Whit}: we regard $\Gr_G = \GO \backsl \GK$; given $g \in \GK$, we write $[g] := \GO g \in \Gr_G$. Similarly for $\Gr_{G,x_j}$ and for a product of those. 
\end{rem}

\sssec{}

Now note that the pullback along the rightmost map is a product of $\oblv^{\GO}$: this is t-exact by the definition of the tensor product t-structure on $\otimes_j \Sph_{G,x_j}$. 
Furthermore, the map in the middle is an isomorphism, hence it is in particular t-excellent. Thus, it remains to show:

\begin{lem} 
For any nonempty $m$-tuple of $\kk$-points $(x_1, \ldots, x_m)$ of $X$ and any collection of $(\alpha_i) \in  \prod_{i=1}^m G(\bbK_{x_i})$, the map 
$$
G[\Sigma]
 \xto{\; \phi \mapsto [\alpha_i \phi] \;}
  \prod_{i \geq 0} \Gr_{G,x_i}
$$ 
is t-excellent.
\end{lem}

\begin{proof}
The ind-scheme $\prod_{i \geq 0} \Gr_{G,x_i}$ classifies pairs $(E, \gamma)$ where $E$ is a $G$-bundle on $X$ and $\gamma$ a trivialization on $\Sigma$.
The given loop group elements determine a $G$-bundle that we call $E_\alpha$. 
In these terms, the displayed map is the base-change the map $\pt \to \Bun_G$ determined by $E_\alpha$. It remains to notice that, as we already discussed, the latter map is t-excellent.
\end{proof}

\ssec{The case of $X=\bbP^1$} \label{ssec:case of P1}

In this section, we prove that tempered objects and $*$-generated objects coincide in the special case of $X = \PP^1$.

\begin{thm}
The full subcategories 
$$
\Dmod(\Bun_G(\PP^1))^\temp \subseteq \Dmod(\Bun_G(\PP^1))
$$
$$
\Dmod(\Bun_G(\PP^1))^\stargen \subseteq \Dmod(\Bun_G(\PP^1))$$ 
coincide.
\end{thm} 

\begin{proof}
Theorem \ref{mainthm:star} yields the inclusion 
$$
\Dmod(\Bun_G(\PP^1))^\stargen \subseteq \Dmod(\Bun_G(\PP^1))^\temp.
$$ 
To prove the opposite inclusion, recall (\cite{omega}) that tempered objects on $\Bun_G(\PP^1)$ are generated by one single object under the Hecke action and colimits. This object is the $*$-extension $j_*(\omega_{\pt/G})$, where $j: \pt/G \hto \Bun_G$ is the open embedding of the locus of trivial $G$-bundles. It remains to show that for any compact object $\G \in \Sph_G$, the result of the Hecke action $\G \star j_*(\omega_{\pt/G})$ is a $*$-extension, too. This is straightforward.
\end{proof}

\begin{rem}
In the above proof, we appealed to the general proof of Theorem \ref{mainthm:star}. This is an over-kill: for a direct proof, it suffices to follow the argument of \cite[Theorem 3.1.9]{omega}.
\end{rem}

\begin{rem}[This can be skipped by the reader]
The inclusion $\Dmod(\Bun_G(\PP^1))^\temp \hto \Dmod(\Bun_G(\PP^1))$ admits a continuous right adjoint, hence the same must hold true for the inclusion
$$
\Dmod(\Bun_G(\PP^1))^\stargen \hto \Dmod(\Bun_G(\PP^1)).
$$
Let us determine the right adjoint explicitly. We claim that it is given by
$$
\M \mapsto \lim (j_U)_* (j_U)^* (\M),
$$
where the limit is taken in $\Dmod(\Bun_G(\PP^1))^\stargen$.
Note that the inclusion $\Dmod(\Bun_G(\PP^1))^\stargen \hto \Dmod(\Bun_G(\PP^1))$ does not preserve limits. 
\end{rem}

\sec{Tempered objects for $G=SL_2$} \label{sec:SL2}

Our goal in this section is two-fold: we prove Theorem \ref{mainthm:rank1} and at the same time prepare the stage for the essential surjectivity part of the automorphic gluing theorem.
Specifically, we will we provide a convenient characterization of the anti-tempered objects of $\Sph_G$ and of $\Dmod(\Gr_G)$ in the special case $G = SL_2$.

\ssec{The Kostant slice}

By abuse of notation\footnote{This object should more properly denoted by $\omega_{\Hecke_{G}^\loc}$, but we find the notation $\omega_{\Sph_G}$ less heavy.}, let 
$$
\omega_{\Sph_G} \in \Sph_G \simeq \Dmod(\Hecke_{G}^\loc)
$$
denote the dualizing sheaf of $\Hecke_{G}^\loc$. 
We begin with the study of $\omega_{\Sph_G}$ on the spectral side of derived Satake: the statements of the present section are valid for any group $G$, we will specialize to $G=SL_2$ in the next section (Section \ref{ssec:SL2}).

\begin{rem}

We already know that $\omega_{\Sph_G}$ is anti-tempered: to see this, it suffices to check that $\oblv^\GO(\omega_{\Sph_G}) \simeq \omega_{\Gr_G}$ is infinitely connective, and this is one of the main results of \cite{omega}.

\end{rem}

\sssec{}

We wish to describe the object that corresponds to $\omega_{\Sph_G}$ under derived Satake and Koszul duality, i.e. under the equivalences%
\footnote{
We refer the reader to \cite[Section 4]{omega} for an extensive discussion of the right equivalence, and in particular of the shearing operation $(- )^\Rightarrow$.
}
\begin{equation} \label{eqn:Satake and KD}
\Sph_G 
\simeq
\ICoh_{\Nch}(\Omega \gch/\Gch)
\simeq
\ICoh
\bigt{
(\gch/G)^\wedge_{\Nch/\Gch}
}^\Rightarrow.
\end{equation}
As we now explain, the answer already appears in \cite{BF}.

\sssec{}

We need to recall the \emph{renormalized derived Satake} equivalence of \cite[Section 12]{AG1}.
Define $\Sph_G^{\on{loc-cpt}} \subseteq \Sph_G$ to be the non-cocomplete DG category consisting of those $\F \in \Sph_G$ whose underlying object $\oblv^\GO(\F) \in \Dmod(\Gr_G)$ is compact. Next, let $\Sph_G^{ren}$ be the ind-completion of $\Sph_G^{\on{loc-cpt}}$. It turns out that the natural  functor $\Psi_{\ren}: \Sph_G^{\ren} \to \Sph_G$ is essentially surjective and that it is equipped with a fully faithful left adjoint $\Xi_{\ren}$. We have:

\begin{thm}[Renormalized derived Satake, \cite{AG1, BF}] \label{thm:ren-Sat}
There is a monoidal equivalence of DG 
$$
\Sat_G^{\ren}:
\ICoh
\bigt{
\gch/\Gch
}^\Rightarrow
\xto{\; \; \simeq \; \; }
\Sph_G ^\ren
$$
compatible with \eqref{eqn:Satake and KD} in the sense that the following diagram is commutative:
\begin{equation} 
\nonumber
\begin{tikzpicture}[scale=1.5]
\node (00) at (0,0) {$ \ICoh
\Bigt{
(\gch/G)^\wedge_{\Nch/\Gch}
}^\Rightarrow$};
\node (10) at (3.2,0) {$ \Sph_G$.};
\node (01) at (0,1.2) {$ \ICoh
\bigt{
\gch/\Gch
}^\Rightarrow$};
\node (11) at (3.2,1.2) {$\Sph_G^\ren$};
\path[->,font=\scriptsize,>=angle 90]
(00.east) edge node[above] {$\Sat_G$}  (10.west); 
\path[->,font=\scriptsize,>=angle 90]
(01.east) edge node[above] {$\Sat_G^\ren$} (11.west); 
\path[->>,font=\scriptsize,>=angle 90]
(01.south) edge node[right] {$\mathit{restrict}$} (00.north);
\path[->>,font=\scriptsize,>=angle 90]
(11.south) edge node[right] {$\Psi_\ren$} (10.north);
\end{tikzpicture}
\end{equation}
\end{thm}

\begin{rem}
Since $\gch/\Gch$ is smooth, one could replace the two occurrences of $\ICoh$ in the above diagram with $\QCoh$. However, we prefer the ind-coherent formulation because of its better functoriality properties: see below for a manifestation of these in practice. Another (vaguely related) example is the calculation of the Serre functor of $\ICoh
\bigt{
(\gch/G)^\wedge_{\Nch/\Gch}
}^\Rightarrow$ in \cite[Section 4]{omega}.
\end{rem}

\sssec{}

Fix a principal $\mathfrak{sl}_2$-triple $(e,h,f)$ in $\gch$, with $e \in \check{\fn}$ and $f \in \check{\fn}^-$. This yields the \emph{Kostant slice} $\Kos := e + \fz(f) \subseteq \gch$.
Consider the correspondence
\begin{equation} \label{eqn:Kostant correspo}
\gch/\Gch 
\xleftarrow{\;\; \wt \kappa \;\;} 
\Kos
\xto{\;\; \pi_{\Kos} \;\;}
\pt.
\end{equation}
The map $\wt \kappa$ is $\Gm$-equivariant for a particular $\Gm$-action on $\Kos$, discussed for instance in \cite[Section 3]{VLaff}.
In particular, we can apply the shearing to $\QCoh(\Kos)$ and to the functors $\wt \kappa^*$ and $(\pi_\Kos)_*$.

\sssec{}

Denote by $(p_!)^\ren : \Sph_G^{\ren} \to \Vect$ the unique continuous functor that restricts to $p_!$ on $\Sph_G^{\on{loc-cpt}}$. Thanks to \cite[Theorem 4]{BF}, we know that $(p_!)^\ren $ corresponds under Theorem \ref{thm:ren-Sat} to the pull-push
$$
\QCoh(\gch/\Gch)^\Rightarrow
\xto{ \;\; \wt \kappa ^* \;\;}
\QCoh(\Kos)^\Rightarrow
\xto{\;\; (\pi_\Kos)_*}
\Vect
$$
along the correspondence \eqref{eqn:Kostant correspo}. As mentioned above, we prefer a formulation in terms of ind-coherent sheaves. So, we precompose the above chain with the canonical equivalence $\ICoh(\gch/\Gch)^\Rightarrow \simeq \QCoh(\gch/\Gch)^\Rightarrow$ induced by $\Upsilon_{\gch/\Gch}$.

\begin{lem}
Under renormalized Satake, the functor $(p_!)^\ren : \Sph_G^{\ren} \to \Vect$ corresponds to the functor
$$
\ICoh(\gch/\Gch)^\Rightarrow
\xto{ \;\; \wt \kappa ^! \;\;}
\ICoh(\Kos)^\Rightarrow
\xto{\;\; (\pi_\Kos)_*^\ICoh [- \dim(\Kos)]}
\Vect
$$
along the correspondence \eqref{eqn:Kostant correspo}.\end{lem}

\begin{proof}
It suffices to prove that $(\pi_\Kos)_*^\ICoh \circ\Upsilon_\Kos \simeq (\pi_\Kos)_* [\dim (\Kos)]$. Since $\Kos$ is isomorphic to an affine space, this is a simple calculation.
\end{proof}

\sssec{}

We are interested in the functor $p_!: \Sph_G \to \Vect$, which is left adjoint to the functor $\Vect \to \Sph_G$ determined by $\omega_{\Sph_G}$.
By construction, $p_!$ arises as the composition 
$$
\Sph_G \xto{\Xi_{\ren}} \Sph_G^{\ren} \xto{(p_!)^\ren} \Vect. 
$$
Hence, on the Langlands-Koszul dual side, this corresponds to the functor induced by the diagram
$$
(\gch/\Gch)^\wedge_{\Nch/\Gch} 
\to
\gch/\Gch 
\xleftarrow{\;\; \wt \kappa \;\;} 
\Kos
\xto{\;\; \pi_{\Kos} \;\;}
\pt.
$$
A simple base-change computation, together with the fact that the Kostant slice intersects $\Nch$ only in $e$, yields the following result.

\begin{cor}
Under derived Satake and Koszul duality, the functor $p_!: \Sph_G \to \Vect$ corresponds to the composition
\begin{equation} \label{eqn:p!-on-spectral-side}
\ICoh\Bigt{ (\gch/\Gch)^\wedge_{\Nch/\Gch}}^\Rightarrow
\xto{ \;\; \kappa ^! \;\;}
\ICoh(\Kos^\wedge_e)^\Rightarrow
\xto{\;\; \beta_*^\ICoh [- \dim(\Kos)]}
\Vect
\end{equation}
induced by the natural maps
$$
(\gch/\Gch)^\wedge_{\Nch/\Gch} 
\xleftarrow{\;\; \kappa \;\;} 
\Kos^\wedge_e
\xto{\;\; \beta \;\;}
\pt.
$$
\end{cor}

\begin{prop}
Under derived Satake and Koszul duality, $\omega_{\Sph_G}$ corresponds to the object
$$
\kappa_*^\ICoh (\omega_{\Kos^\wedge_e}) 
\in 
\ICoh( (\gch/\Gch)^\wedge_{\Nch/\Gch} )^\Rightarrow.
$$
\end{prop}

\begin{proof}
By taking the right adjoint of \eqref{eqn:p!-on-spectral-side}, we see that $\omega_{\Sph}$ must correspond to the object
$$
(\kappa^!)^R \bigt{
\omega_{\Kos^\wedge_e}
}[\dim(\Kos)].
$$
Here we used the fact that $\beta$ is a nil-isomorphism (see \cite[Volume II, Chapter 3.3]{Book}), and so $\beta^! \simeq (\beta_*^\ICoh)^R$. Thus, it remains to prove that
$$
(\kappa^!)^R \simeq \kappa_*^\ICoh [-\dim(\Kos)].
$$
To this end, we will provide a functorial isomorphism
\begin{equation} \label{eqn:Kos-funct}
\HHom_{\ICoh(\Kos^\wedge_e)^\Rightarrow} \Bigt{ \kappa^!(\F), \G }
\simeq
\HHom_{\ICoh( (\gch/\Gch)^\wedge_{\Nch/\Gch} )^\Rightarrow} \Bigt{ \F, \kappa_*^\ICoh( \G ) } [-\dim(\Kos)]
\end{equation}
for $\F \in \ICoh( (\gch/\Gch)^\wedge_{\Nch/\Gch} )^\Rightarrow$ and $\G \in \ICoh(\Kos^\wedge_e)^\Rightarrow$. As the shearings do not play a role, we omit them from now on.

Consider the tautological nil-isomorphism
$$
\phi:
\Nch/\Gch
\longto
(\gch/\Gch)^\wedge_{\Nch/\Gch}.
$$
The general theory of ind-coherent sheaves on formal completions and nil-isomorphisms (\cite{Book}) yields an adjunction
\begin{equation}
\nonumber
\begin{tikzpicture}[scale=1.5]
\node (a) at (0,1) {$\phi_*^\ICoh: \ICoh(\Nch/\Gch) $};
\node (b) at (3.5,1) {$\ICoh
\Bigt{
(\gch/\Gch)^\wedge_{\Nch/\Gch}
}:
\phi^!$};
\path[->,font=\scriptsize,>=angle 90]
([yshift= 1.5pt]a.east) edge node[above] {$ $} ([yshift= 1.5pt]b.west);
\path[->,font=\scriptsize,>=angle 90]
([yshift= -1.5pt]b.west) edge node[below] {$ $} ([yshift= -1.5pt]a.east);
\end{tikzpicture}
\end{equation}
with conservative right adjoint (equivalently, the essential image of $\phi_*^\ICoh$ generates the target under colimits). Thus, in \eqref{eqn:Kos-funct} we may assume that $\F \simeq \phi_*^\ICoh(\M)$ for some $\M \in \ICoh(\Nch/\Gch)$: 
\begin{equation} 
\nonumber
\HHom_{\ICoh(\Kos^\wedge_e)^\Rightarrow} \Bigt{ \kappa^!(\phi_*^\ICoh(\M)), \G }
\simeq
\HHom_{\ICoh( (\gch/\Gch)^\wedge_{\Nch/\Gch} )^\Rightarrow} \Bigt{ \phi_*^\ICoh(\M), \kappa_*^\ICoh( \G ) } [-\dim(\Kos)].
\end{equation}
The compositions $\phi_*^\ICoh \circ \kappa^!$ and $\phi^! \circ \kappa_*^\ICoh$ simplify by base-change: since $\Kos$ intersects the nilpotent cone transversally, we obtain that 
$$
\Kos^\wedge_e
\utimes_{(\gch/\Gch)^\wedge_{\Nch/\Gch}}
{\Nch/\Gch}
\simeq
e,
$$
where by abuse of notation $e$ denotes the point scheme $\{e\} \subset \gch$.
Denoting by $\alpha: e \to \Nch/\Gch$ and $\iota: e \to \Kos^\wedge_e$ the obvious maps, it remains to provide a functorial isomorphism
\begin{equation}
\nonumber
\HHom_{\Vect} 
\Bigt{ 
\alpha^!(\M), 
\iota^!(\G) }
\simeq
\HHom_
{\Vect} 
\Bigt{ 
\alpha^{*,\ICoh}(\M),
\iota^!(\G)
} [-\dim(\Kos)].
\end{equation}
We will provide an isomorphism
$$
\alpha^! 
\simeq
\alpha^{*,\ICoh}[\dim(\Kos)].
$$
This follows immediately from \cite[Lemma 2.3.6]{omega}, together with the identity $\dim(\Gch) = \dim(\Kos) + \dim(\Nch)$.
\end{proof}

\begin{rem}
Note that $\kappa$ obviously lands in $(\gch^\times /\Gch)^\wedge_{\Nch^\times/\Gch}$. This is yet another proof of the fact that $\omega_{\Sph_G}$ is anti-tempered.
\end{rem}

\ssec{Anti-tempered objects of $\Sph_{SL_2}$} \label{ssec:SL2}

In this section, we discuss some special features of the $G=SL_2$ case. We first prove Theorem \ref{thm:antitemp Sph SL2}, which shows that $\omega_{\Sph_G}$ generates $\Sph_G^{\antitemp}$ in a strong sense. We will then record some general consequences concerning the anti-tempered subcategory of $\Dmod(\Gr_{SL_2})$.

\sssec{}

For arbitrary $G$, denote by $\bbone_{\Sph}^{\antitemp} \in \Sph_G$ the \emph{anti-tempered unit} of $\Sph_G$. By definition, this is the image of the monoidal unit $\bbone_{\Sph_G}$ under the projection $\Sph_G \tto \Sph_G^\antitemp$.

Now assume $G = SL_2$ until further notice.

\begin{thm} \label{thm:antitemp Sph SL2}
For $G=SL_2$, the anti-tempered unit $\bbone_{\Sph}^{\antitemp} \in \Sph_G$ can be expressed as a cone of two shifted copies of $\omega_{\Sph_G}$. More precisely,
$$
\bbone_{\Sph}^{\antitemp} \simeq 
\ker
\Bigt{
\omega_{\Sph_G} \to \omega_{\Sph_G}[2]
}.
$$
\end{thm}

\begin{proof}
By derived Satake and Koszul duality, we have: 
$$
 \Sph_G ^\antitemp
\simeq 
\ICoh
\Bigt{ (\gch/\Gch)^\wedge_{\Nch^\times/\Gch}
}
^\Rightarrow.
$$
Our first goal is to express the right-hand-side as the DG category of modules over an algebra. We work without the shearing and turn it on only at the end.
Similarly to the previous proof, consider the formal completion $(\gch/\Gch)^\wedge_{\Nch^\times/\Gch}$ and the tautological nil-isomorphism
$$
\psi:
\Nch^\times/\Gch
\longto
(\gch/\Gch)^\wedge_{\Nch^\times/\Gch}.
$$
Again, we obtain an adjunction
\begin{equation}
\nonumber
\begin{tikzpicture}[scale=1.5]
\node (a) at (0,1) {$\psi_*^\ICoh: \ICoh(\Nch^\times/\Gch) $};
\node (b) at (3.5,1) {$\ICoh
\Bigt{
(\gch/\Gch)^\wedge_{\Nch^\times/\Gch}
}:
\psi^!$};
\path[->,font=\scriptsize,>=angle 90]
([yshift= 1.5pt]a.east) edge node[above] {$ $} ([yshift= 1.5pt]b.west);
\path[->,font=\scriptsize,>=angle 90]
([yshift= -1.5pt]b.west) edge node[below] {$ $} ([yshift= -1.5pt]a.east);
\end{tikzpicture}
\end{equation}
with conservative right adjoint.

The monad of the adjunction is the universal envelope of the relative tangent Lie algebroid attached to the map $\Nch^\times/\Gch \hto \gch^\times/\Gch$. As in \cite[Lemma 4.3.10]{omega}, this monad is the functor of tensoring with the exterior algebra $\Sym(\Tang_{\fc_{\Gch},0}[-1])$: this follows from the isomorphism
$$
\Nch^\times/\Gch 
\simeq
 \gch^\times/\Gch
\utimes_{\fc_\Gch} 0,
$$
where $\fc_{\Gch}$ is Chevalley's space.

Now let us use the fact that $\Gch \simeq PGL_2$. In this case, the above exterior algebra is $1$-dimensional: we denote it by $\kk[\eta']$, with $\eta'$ in cohomological degree $-1$. The reason for the primed notation will be clear later on, when we reinsert the shearing.

Next, note that the $\Gch$ acts on $\Nch^\times$ transitively and that the stabilizer of $e \in \Nch^\times$ equals $\GG_a$. This implies that $\Nch^\times/\Gch \simeq \pt/\GG_a$, from which it follows that
$$
\ICoh(\Nch^\times/\Gch)
\simeq
\QCoh(\Nch^\times/\Gch)
\simeq
\QCoh(\pt/\GG_a)
\simeq \kk[\epsilon']\mod,
$$
with $\epsilon'$ a generator of cohomological degree $1$. We normalize this equivalence by declaring that $\omega_{\Nch^\times/\Gch} \in \ICoh(\Nch^\times/\Gch)$ corresponds to the regular $\kk[\epsilon']$-module.

\medskip

All in all, we obtain an equivalence 
$$
L':
\ICoh
\Bigt{ (\gch/\Gch)^\wedge_{\Nch^\times/\Gch}
} 
\simeq 
\kk[\epsilon', \eta'] \mod.
$$
which sends $\psi_*^\ICoh (\omega_{\Nch^\times/\Gch})$ to the regular $\kk[\epsilon', \eta']$-module.
We can now turn on the shearing: following the grading described in the proof of \cite[Lemma 4.3.10]{omega}, we obtain an equivalence
$$
L:
\ICoh
\Bigt{ (\gch/\Gch)^\wedge_{\Nch^\times/\Gch}
} ^\Rightarrow
\simeq 
\kk[\epsilon, \eta] \mod,
$$
where $\epsilon$ and $\eta$ have now cohomological degree $-1$.

\medskip

Next, let us express $\bbone_{\Sph}^{\antitemp} \in \Sph_G$ and $\omega_{\Sph_G}$ in terms of this equivalence. 
The derived Satake equivalence is monoidal, hence $\bbone_{\Sph}^{\antitemp}$ corresponds to the dualizing sheaf in $\ICoh
\bigt{ (\gch/\Gch)^\wedge_{\Nch^\times/\Gch}
}^\Rightarrow$.
One easily checks that, under $L$, this objects goes over to the $\kk[\epsilon, \eta]$-module $\kk[\epsilon]$.
On the other hand, we saw that $\omega_{\Sph_G}$ corresponds to the push-forward of the dualizing sheaf on the Kostant slice. As in the previous proof, the slice intersects the nilpotent cone transversally and so
$$
\Kos^\wedge_e
\utimes_{(\gch/\Gch)^\wedge_{\Nch^\times/\Gch}}
{\Nch^\times/\Gch}
\simeq
e.
$$
This proves that, under $L$ and derived Satake, $\omega_{\Sph_G}$ goes over to the augmentation $\kk[\epsilon, \eta]$-module $\kk$.
Since $\kk[\epsilon] \simeq \ker(\kk \to \kk[2])$ as $\kk[\epsilon, \eta]$-modules, we conclude that $\bbone_{\Sph_G}^\antitemp$ can be expressed in terms of $\omega_{\Sph_G}$ as claimed.
\end{proof}

\sssec{}

Let us collect some consequences of the above statement and its proof. The first consequence is a general fact, valid for any $\C$ acted on by $\Sph_{SL_2}$; the second and third are specific to the DG category $\Dmod(\Gr_{SL_2})$.

\begin{cor}
Let $G=SL_2$ and consider a DG category $\C$ equipped with an action of $\Sph_G$. Then the full subcategory $\lantitemp \C$ coincides with the cocompletion of the essential image of the functor $\omega_{\Sph_G} \star - : \C \to \C$.
In particular, an object $c \in \C$ is tempered iff $\omega_{\Sph_G} \star c \simeq 0$. 
\end{cor}

\begin{cor} \label{cor:antitemp Gr}
Let $G=SL_2$. An object $ \F \in \Dmod(\Gr_G)$ is anti-tempered iff it is of the form 
$$
\coker (\omega_{\Gr_G} \otimes V \xto{\beta} \omega_{\Gr_G} \otimes V')
$$
for some $V, V' \in \Vect$ and some morphism $\beta$ in $\Dmod(\Gr_G)$. 
\end{cor}

\begin{proof}
Since $\omega_{\Gr_G}$ is anti-tempered, so is any object of the prescribed form. Now suppose that $\F$ is anti-tempered.
This implies that 
$$
\F 
\simeq
\bbone_{\Sph_G}^\antitemp
\star \F
\simeq
\ker
\bigt{
\omega_{\Sph_G} \star \F
\to
\omega_{\Sph_G}\star \F[2]
}. 
$$
By construction, $\omega_{\Sph_G} \star \F$ is obtained by $!$-pulling and $!$-pushing $\F$ along the following correspondence:
$$
\Gr_G 
\xleftarrow{proj}
\Gr_G \times^\GO \GK
\xto{\act}
\Gr_G
$$ 
$$
\GO g' \mapsfrom [ \GO g,g'] \mapsto \GO g g'.
$$ 
Now observe that this correspondence \virg{splits}, that is, it is isomorphic to the product correspondence $\Gr_G \times \Gr_G \rr \Gr_G$. Hence, 
$$
\omega_{\Sph_G} \star \F 
\simeq 
\omega_{\Gr_G} \otimes (p_{\Gr_G})_!(\F)
$$
and our assertion follows.
\end{proof}

\begin{cor} \label{cor:temp on Gr-SL2}
For $G=SL_2$, an object $\F \in \Dmod(\Gr_G)$ is tempered iff $(p_{\Gr_G})_!(\F) \simeq 0$.
\end{cor}

\begin{proof}
The only if direction (valid for any $G$) is simply the fact that $\omega_{\Gr_G}$ is anti-tempered. 
Now suppose that $(p_{\Gr_G})_!(\F) \simeq 0$. We showed above that any anti-tempered object $\A \in \Dmod(\Gr_G)$ can be written as 
$$
\A \simeq \ker (V \otimes \omega_{\Gr_G} \to V' \otimes \omega_{\Gr_G}).
$$
It follows immediately that 
$$
\HHom(\F, \A)
 \simeq 
\ker 
\bigt{
(\HHom((p_{\Gr_G})_!(\F) , V) \to (\HHom((p_{\Gr_G})_!(\F), V') 
} \simeq 0,
$$
a vanishing that proves the temperedness of $\F$.
\end{proof}

\ssec{Proof of Theorem \ref{mainthm:rank1}}

Here we use the results proven above to deduce that any $\F \in \Dmod(\Bun_{SL_2})$ with $(p_{\Bun_G})_!(\F) \simeq 0$ is tempered.

\sssec{}

Given a scheme $Y$, thought of as target, let us define following \citep{Barlev} the prestack $Y[X]^\rat$ of \emph{rational maps from $X$ to $Y$}. To an affine scheme $S$, it assigns the set of equivalence classes $(U, f: U \to Y)$, where $U \subseteq X_S$ is an open subscheme that is universally dense over $S$. The equivalence relation is defined by declaring $(U,f) \approx (V,g)$ if $f=g$ on $U \cap V$.

We do not require $X$ to be proper and in fact we will use this definition for $X^* = X -x$. As for the target, we will take the Borel subgroup $B \subset G$. In this case, $B[X^*]^\rat$ is a \emph{group} prestack. Moreover, it is equipped with a group morphism $B[X^*]^\rat \to \BK \to  \GK$ and thus it acts on $\Gr_{G,x}$.

\begin{lem}
When $G$ is semisimple, we have
$$
\Dmod(\Gr_{G,x})^{G[X^*]}
\simeq
\Dmod(\Bun_G),
$$
as well as
$$
\Dmod(\Gr_{G,x})^{B[X^*]^\rat}
\simeq
\Dmod(\Bun_G^{B\ggen}).
$$
\end{lem}

\begin{proof}
The maps $\Gr_{G,x} \to \Bun_G^{B\ggen}$ and $\Gr_{G,x} \to \Bun_G$ are \'etale-surjections by \cite{DS}. Hence, $\Dmod(\Bun_G^{B\ggen})$ is equivalent to the totalization of the cosimplicial DG category obtained from the Cech complex of $\Gr_{G,x} \to \Bun_G^{B\ggen}$ by applying the contravariant functor $\Dmod(-)$. 
It remains to observe that 
$$
\Gr_{G,x} \utimes_{\Bun_G^{B\ggen}} \Gr_{G,x}
\simeq
\Gr_{G,x} \times {B[X^*]^\rat}
$$
$$
\Gr_{G,x} \utimes_{\Bun_G } \Gr_{G,x}
\simeq
\Gr_{G,x} \times {G[X^*]}
$$
and similarly for the higher iterated fiber products forming the Cech complex.
\end{proof}

\sssec{}

Now recall that $\Dmod(\Gr_{G,x})$ is a bimodule for the left\footnote{
This is because we are seeing $\Gr_{G,x}$ as the quotient of the left action of $\GO$ on $\GK$. For clarity, in the present discussion, we will place the decoration \virg{$\antitemp$} on the left.} action of $\Sph_{G}$ and the right action of the loop group $\GK$ at $x$. (This holds true for any reductive group $G$.) 
Then the above lemma implies that, for $G$ semisimple,
$$
\lantitemp
\Dmod(\Bun_G)
\simeq
\lantitemp
\Bigt{ 
\Dmod(\Gr_{G,x})^{G[X^*]}
}
\simeq
\Bigt{ 
\lantitemp
\Dmod(\Gr_{G,x}) }^{G[X^*]}.
$$

\sssec{}

Next, consider the forgetful functor
$$
\oblv^{G[X^*]}:
\Dmod(\Gr_{G,x})^{G[X^*]}
\longto
\Dmod(\Gr_{G,x}).
$$
Under the equivalence $\Dmod(\Bun_G) \simeq \Dmod(\Gr_{G,x})^{G[X^*]}$, this functor goes over to the $!$-pullback along the map $\pi: \Gr_{G,x} \to \Bun_G$.
Clearly, $\oblv^{G[X^*]}$ is  $\Sph_G$-linear and thus it restricts to a forgetful functor
$$
\oblv^{G[X^*]}:
\lantitemp
\Dmod(\Bun_G)
\longto
\lantitemp
\Dmod(\Gr_{G,x}).
$$

\begin{lem}
In the case $G=SL_2$, the latter functor admits a left adjoint given by the restriction of $\pi_!$ to the anti-tempered subcategories.
\end{lem}

\begin{proof}
Recall that the projection $\Dmod(\Bun_G) \tto \lantitemp\Dmod(\Bun_G)$ left adjoint to the natural inclusion is always well-defined: this projection is the anti-temperization functor of \eqref{eqn:exact-seq Sat-C}. By formal nonsense, the left adjoint in question is given (if it exists) by the composition
$$
\lantitemp
\Dmod(\Gr_{G,x})
\hto
\Dmod(\Gr_{G,x})
\xto{\;\;\pi_! \;\;}
\Dmod(\Bun_G)
\tto
\lantitemp\Dmod(\Bun_G).
$$
It remains to show that $\pi_!$ is well-defined on the entire $\lantitemp
\Dmod(\Gr_{G,x})$ and that it lands in $\lantitemp
\Dmod(\Bun_G)$. In view of Corollary \ref{cor:antitemp Gr}, the DG category $\lantitemp
\Dmod(\Gr_{G,x})$ is generated by objects of the form $\omega_{\Gr_G} \otimes V$, with $V \in \Vect$. In particular $\pi_!$ is well-defined by ind-holonomicity. Moreover, 
$$
\pi_! (\omega_{\Gr_G}) 
\simeq 
\Av^{G[X^*]}_! \circ \oblv^{G[X^*]} (\omega_{\Bun_G})
\simeq
\omega_{\Bun_G}
\otimes
H_*(G[X^*]).
$$
The anti-temperedness of the latter object is known by \cite[Theorem A]{omega}.
\end{proof}

\sssec{}

Since 
$$
\oblv^{G[X^*]}:
\lantitemp
\Dmod(\Bun_G)
\longto
\lantitemp
\Dmod(\Gr_{G,x})
$$
is conservative by construction, the essential image of its left adjoint generates $\lantitemp
\Dmod(\Bun_G)$ under colimits. It follows that any anti-tempered object of $\Dmod(\Bun_G)$ arises as a colimit of objects of the form $(p_{\Bun_G})^!(V)$.

\sssec{}

Mutatis mutandis, the above arguments renders to the case of $\Bun_G^{B\ggen}$. In that case observe that, in view of the contractibility of $B[X^*]^\rat$, we have
$$
\Av^{B[X^*]^\rat}_! \circ \oblv^{B[X^*]^\rat} (\omega_{\Bun_G^{B\ggen}})
\simeq
\omega_{\Bun_G^{B\ggen}}.
$$
Thus, we have proved the first part of the following theorem.

\begin{thm} \label{thm:omegas}
For $G = SL_2$, let $\Y = \Bun_G$ or $\Y = \Bun_G^{B \ggen}$. Then
\begin{itemize}
\item
the DG category $\lantitemp\Dmod(\Y)$ is the cocompletion of the essential image of $(p_\Y)^!: \Vect \to \Dmod(\Y)$;
\item
the DG category $\ltemp\Dmod(\Y)$ coincides with the full subcategory $\ker \bigt{ (p_\Y)_!} \subseteq \Dmod(\Y)$.
\end{itemize}

\end{thm}

\begin{proof}

It remains to prove the second statement. Let us first assume that $\Y = \Bun_G$. We already know that tempered objects belong to the kernel of $(p_{\Bun_G})_!$. So, let $\F \in \Dmod(\Bun_G)$ be such that $(p_{\Bun_G})_!(\F) \simeq 0$.  We wish to show that $\F$ is tempered. 
By the conservativity and $\Sph_G$-linearity of $(\pi_{G,x})^!$, it suffices to show that $(\pi_{G,x})^!(\F)$ is tempered.
In turn, by Corollary \ref{cor:temp on Gr-SL2}, it suffices to show that 
$$
(p_{\Gr_G})_! (\pi_{G,x})^!(\F) = 0.
$$ 
We have:
\begin{eqnarray}
\nonumber
(p_{\Gr_G})_! (\pi_{G,x})^!(\F) 
& \simeq &
(p_{\Bun_G})_! 
\bigt{
(\pi_{G,x})_! 
(\pi_{G,x})^!(\F)
}
\\
\nonumber
& \simeq &
(p_{\Bun_G})_! 
\bigt{
\F \otimes H_*(G[X^*])
}
\\
\nonumber
& \simeq &
(p_{\Bun_G})_! (\F) \otimes  H_*(G[X^*])
\\
\nonumber
& \simeq &
0.
\end{eqnarray}
The case of $\Y = \Bun_G^{B \ggen}$ works in the same manner.
\end{proof}

\begin{rem}
We can now give another proof of Theorem \ref{mainthm:star}, valid only for $SL_2$. Namely, we start with the following result, proven in \cite{DL}: for $G$ arbitrary, $(p_{\Bun_G})_!(\F) \simeq 0$ whenever $\F \in \Dmod(\Bun_G)^\stargen$. Now, let us combine this with the result proven above: $\F \in \Dmod(\Bun_{SL_2})$ is tempered if and only if $(p_{\Bun_{SL_2}})_!(\F) \simeq 0$.  
\end{rem}

\sec{Automorphic gluing for $G=SL_2$} \label{sec:aut-gluing}

\ssec{The statement} \label{ssec:aut-gluing-statement}

Let $G$ be a connected reductive group and $P$ a parabolic subgroup.
Following \cite{Outline}, let us recall the definition of the DG category $I(G,P)$.

\sssec{}

First, we have the prestack $\Bun_G^{P \ggen}$, see \citep{Barlev}, and the corresponding DG category of D-modules $\Dmod(\Bun_G^{P \ggen})$.
We have maps
$$
\Bun_P 
\longto 
\Bun_G^{P \ggen} 
\xto{\;\; \p_P^\gen \;\;}
\Bun_G.
$$
The pullback functor along the second map is fully faithful: this is one of the main theorems of \cite{Barlev}, based on \cite{contract}. 
The pullback functor along the left map is conservative: this is because the two prestacks involved have the same field-valued points.

\sssec{}

Consider also the usual map $\q_P: \Bun_P \to \Bun_M$, where $M$ is the Levi quotient of $P$. It is clear that $\q_P^{*}: \Dmod(\Bun_M) \to \Dmod(\Bun_P)$ is well-defined and fully faithful. 

Define $I(G,P)$ to be the full subcategory of $\Dmod(\Bun_G^{P \ggen})$ spanned by objects whose restriction to $\Bun_P$ lies in the essential image of $\q_P^{*}$. In formulas:
$$
I(G,P):= 
\Dmod(\Bun_G^{P \ggen})
\utimes_{\Dmod(\Bun_P)}
\Dmod(\Bun_M).
$$

\begin{rem}
It is immediate to see that the dualizing sheaf $\omega_{\Bun_G^{P\ggen}} \in \Dmod(\Bun_G^{P\ggen})$ actually belongs to $I(G,P)$.
\end{rem}

\begin{lem}
The structure inclusion $I(G,P)\hto \Dmod(\Bun_G^{P \ggen})$ admits a continuous right adjoint, which we denote by $\Av^{U(\AA)}$. 
\end{lem}

\begin{lem}
The obvious $\Sph_G$-action on $\Dmod(\Bun_G^{P \ggen})$ preserves $I(G,P)$.
\end{lem}

\sssec{}

Define $\CT_P^\enh: \Dmod(\Bun_G) \to I(G,P)$ to be the composition $\CT_P^\enh:= \Av^{U(\AA)} \circ (\p_P^{\gen})^!$. 
The lemma above guarantees that $\CT_P^\enh$ is $\Sph_G$-linear.

Thus, it makes sense to consider $I(G,P)^\temp$. Since $\CT_P^\enh$ (like any other $\Sph_G$-linear functor) commutes with temperization, the square
\begin{equation}  \label{diag:CT-temp}
\begin{tikzpicture}[scale=1.5]
\node (00) at (0,0) {$ \Dmod(\Bun_G)^\temp$};
\node (10) at (2.5,0) {$ I(G,P)^\temp$.};
\node (01) at (0,1) {$ \Dmod(\Bun_G)$};
\node (11) at (2.5,1) {$I(G,P)$};
\path[->,font=\scriptsize,>=angle 90]
(00.east) edge node[above] {$\CT_P^\enh$}  (10.west); 
\path[->,font=\scriptsize,>=angle 90]
(01.east) edge node[above] {$\CT_P^\enh$} (11.west); 
\path[->>,font=\scriptsize,>=angle 90]
(01.south) edge node[right] {$\temp$} (00.north);
\path[->>,font=\scriptsize,>=angle 90]
(11.south) edge node[right] {$\temp$} (10.north);
\end{tikzpicture}
\end{equation}
is commutative.

\sssec{}

The above discussion yields the $\Sph_G$-linear functor
$$
\gamma_{G,P}:
\Dmod(\Bun_G)
\longto
\Dmod(\Bun_G)^\temp
\utimes_{I(G,P)^\temp}
I(G,P).
$$
By construction, it sends $\F \mapsto (\temp(\F), \CT_P^\enh(\F))$, the latter pair equipped with the obvious gluing datum.
Our goal is to prove that $\gamma_{G,B}$ is an equivalence for $G=SL_2$.

\sssec{}

For future use, let us note that $\CT_P^{\enh}$ admits a $\Sph_G$-linear left adjoint $\Eis_P^\enh$. This is the \emph{enhanced Eisenstein series functor}, arising as the composition
$$
I(G,P) 
\hto 
\Dmod(\Bun_G^{P\ggen})
\xto{(\p_P^\gen)_!}
\Dmod(\Bun_G).
$$
It follows that the four arrows of \eqref{diag:CT-temp} all admit left adjoints.

\ssec{Deligne-Lusztig duality}

We digress to record an application of Theorem \ref{mainthm:star} to the Deligne-Lusztig duality functor $\DL_G$. Here $G$ denotes an arbitrary connected reductive group.

\sssec{} \label{ssec:DL-definition}

First, we recall that $\DL_G$ is the endofunctor of $\Dmod(\Bun_G)$ given by the formula
$$
\coker
\Bigt{
\colim_{P \in \Par'}
\Eis_P^\enh \CT_P^\enh
\longto
\id_
{\Dmod(\Bun_G)}
}.
$$
By construction, $\DL_G$ is $\Sph_G$-linear.
We invite the reader to consult \cite{DW, Wang, LinChen, DL} for more information on the functor $\DL_G$.

\sssec{}

Let us now recall Chen's theorem, see \cite{LinChen}. This theorem, valid for any $G$, expresses $\DL_G$ in completely different terms via a composition of two dualities: for any $G$, we have
$$
\DL_G \simeq \TBunG [2 \dim(\Bun_G) + \dim(Z_G)],
$$
where $\TBunG$ is by definition the composition of 
$$
\Psid_*: \Dmod(\Bun_G)^\vee \to \Dmod(\Bun_G)
$$
with the inverse of 
$$
\Psid_!: \Dmod(\Bun_G)^\vee \to \Dmod(\Bun_G).
$$ 
The reader might recall that $\Psid_*$ appeared in Section \ref{sssec:Psid*}: it is the functor given by the kernel $\Delta_*(\omega_{\Bun_G})$. Similarly, $\Psid_!$ is the functor given by the Verdier-dual kernel. Contrarily to $\Psid_*$, it turns out that $\Psid_!$ is an equivalence, called \emph{miraculous duality}, see \cite{miraculous}.

\sssec{}

Let us know invoke \cite[Lemma 2.1.9]{DL}, which states that the essential image of $\TBunG$ equals $\Dmod(\Bun_G)^\stargen$. Combining this result with Chen's theorem and with Theorem \ref{mainthm:star}, we deduce that $\DL_G$ is a $\Sph_G$-linear endofunctor of $\Dmod(\Bun_G)$ whose essential image is contained in $\Dmod(\Bun_G)^\temp$. This yields:

\begin{cor} \label{cor:DL}
For any reductive group $G$, the functor $\DL_G: \Dmod(\Bun_G) \to \Dmod(\Bun_G)$ annihilates all anti-tempered objects.
\end{cor}

\begin{proof}
Let $\A \in \Dmod(\Bun_G)$ be anti-tempered. Since $\DL_G$ is $\Sph_G$-linear, $\DL_G(\A)$ is anti-tempered. But $\DL_G$ lands in $\Dmod(\Bun_G)^\temp$, so $\DL_G(\A)$ is also tempered.
\end{proof}

\ssec{Fully faithfulness}

In this section, we prove that $\gamma_{SL_2, B}$ is fully faithful. We exploit a few facts that hold true for any pair $(G,P)$.

\sssec{}

Define $\Glue_{G,P}$ to be the target of $\gamma_{G,P}$, that is,
$$
\Glue_{G,P}
:=
\Dmod(\Bun_G)^\temp
\utimes_{ I(G,P)^\temp}
I(G,P).
$$
Note that $\Glue_{G,P}$ can be realized as a colimit of DG categories obtained by replacing the arrows of \eqref{diag:CT-temp} with their left adjoints. In other words, the diagram
\begin{equation} 
\nonumber
\begin{tikzpicture}[scale=1.5]
\node (00) at (0,0) {$ \Dmod(\Bun_G)^\temp$};
\node (10) at (2.5,0) {$  I(G,P)^\temp$.};
\node (01) at (0,1) {$ \Glue_{G,P}$};
\node (11) at (2.5,1) {$I(G,P)$};
\path[<-,font=\scriptsize,>=angle 90]
(00.east) edge node[above] {$\Eis_P^\enh$}  (10.west); 
\path[<-,font=\scriptsize,>=angle 90]
(01.east) edge node[above] {$\Eis_P^\enh$} (11.west); 
\path[<- right hook,font=\scriptsize,>=angle 90]
(01.south) edge node[right] { } (00.north);
\path[<- right hook ,font=\scriptsize,>=angle 90]
(11.south) edge node[right] { } (10.north);
\end{tikzpicture}
\end{equation}
is a pushout of DG categories.

\sssec{}

It follows formally that $\gamma_{G,P}: \Dmod(\Bun_G) \to \Glue_{G,P}$ admits a left adjoint $(\gamma_{G,P})^L$.
By construction, the composition $ (\gamma_{G,P})^L \circ \gamma_{G,P}$ is the endofunctor of $\Dmod(\Bun_G)$ that sends $\F \in \Dmod(\Bun_G)$ to the colimit of the correspondence given by the solid arrows below: 
\begin{equation}  \label{diag:pushout of objects}
\begin{tikzpicture}[scale=1.5]
\node (00) at (0,0) {$ \temp(\F)$};
\node (10) at (2.5,0) {$ \F $.};
\node (01) at (0,.8) {$ \Eis_P^\enh \CT_P^\enh (\temp(\F))$};
\node (11) at (2.5,.8) {$\Eis_P^\enh \CT_P^\enh(\F)$};
\path[dashed,->,font=\scriptsize,>=angle 90]
(00.east) edge node[above] {$ $}  (10.west); 
\path[->,font=\scriptsize,>=angle 90]
(01.east) edge node[above] {$ $} (11.west); 
\path[->,font=\scriptsize,>=angle 90]
(01.south) edge node[right] {$\on{counit}$} (00.north);
\path[dashed,-> ,font=\scriptsize,>=angle 90]
(11.south) edge node[right] { } (10.north);
\end{tikzpicture}
\end{equation}
Of course, $(\gamma_{G,P})^L \circ \gamma_{G,P}$ admits a natural transformation to $\id_{\Dmod(\Bun_G)}$: this corresponds to the datum of an extension (functorial in $\F$) of solid diagram above to a commutative square as indicated.
If $\F \in \Dmod(\Bun_G)$ is tempered, then this diagram is obviously a pushout.

\sssec{}

Now let us focus on the case of $G = SL_2$. In this case, the proof of the fully faithfulness of $\gamma_{SL_2,B}$ reduces to showing that the square in \eqref{diag:pushout of objects} is a pushout diagram for any \emph{anti-tempered} $\F \in \Dmod(\Bun_{SL_2})$.
By definition, $\temp(\F)=0$ for such an $\F$; hence we just need to show that the natural arrow
$$
\Eis_B^\enh \CT_B^\enh(\F) \to \F
$$
is an isomorphism. Put another way, we need to show that $\CT_B^\enh$ is fully faithful on the anti-tempered subcategory $\Dmod(\Bun_{SL_2})^{\antitemp}$.
Yet equivalently, we will show that 
$$
\coker(\Eis_B^\enh \CT_B^\enh(\F) \to \F) \simeq 0
$$
for any anti-tempered $\F$.

\sssec{}

Now we recognize that the functor  
$$
\coker \bigt {
\Eis_B^\enh \CT_B^\enh \to \id_{\Dmod(\Bun_{SL_2})}
}
$$
is \emph{by definition} the Deligne-Lusztig duality functor
$\DL_{SL_2}$ for the group $SL_2$. Hence, the fully faithfulness of $\gamma_{SL_2, B}$ is a consequence of Corollary \ref{cor:DL}.

\ssec{Completion of the proof}

Our current task is to show that $\gamma:=\gamma_{SL_2, B}$ is an equivalence. In the previous section, we proved that $\gamma$ is fully faithful, so it remains to show that $\gamma^L$ is conservative.

\sssec{}

Let us begin with an explicit description of $\gamma^L$. We represent $\G \in \Glue_{G,B}$ as a triple $(\T, \M,  \eta)$ with 
\begin{eqnarray}
\nonumber
& & \T \in  \Dmod(\Bun_G)^\temp \\
\nonumber
& &  \M \in I(G,B) \\
\nonumber
& &  \eta: \temp(\M) \simeq \CT_B^\enh(\T).
\end{eqnarray}
Note that $\eta$ induces an arrow
$$
\eta': 
\Eis_B^\enh(\temp(\M)) 
\simeq
\Eis_B^\enh(\CT_B^\enh(\T)) 
\xto{\mathit{counit}}
\T.
$$
A simple calculation shows that $\gamma^L(\G)$ equals the pushout of
\begin{equation}  \label{diag:gamma left adjoint}
\begin{tikzpicture}[scale=1.5]
\node (00) at (0,0) {$ \Eis_B^\enh(\M).$};
\node (01) at (0,.8) {$ \Eis_B^\enh(\temp(\M))$};
\node (11) at (2.5,.8) {$\T$};
\path[->,font=\scriptsize,>=angle 90]
(01.east) edge node[above] {$\eta'$} (11.west); 
\path[->,font=\scriptsize,>=angle 90]
(01.south) edge node[right] {$ $} (00.north);
\end{tikzpicture}
\end{equation}

\sssec{}

We can now proceed with the proof of the conservativity of $\gamma^L$. We will prove that $\gamma \circ \gamma^L$ is conservative: so, given $\G \in \Glue_{G,B}$ such that $\gamma \circ \gamma^L (\G) \simeq 0$, we wish to show that $\G \simeq 0$. 
The assumption $\gamma \circ \gamma^L (\G) \simeq 0$ is equivalent to having $\temp(\gamma^L(\G)) \simeq 0$ and $\CT_B^\enh(\gamma^L(\G)) \simeq 0$.

\sssec{}

The explicit expression \eqref{diag:gamma left adjoint} yields $\temp(\gamma^L(\G)) \simeq \T$. It follows that 
$$
\G \simeq (0, \M, \eta: \temp(\M) \simeq 0)
$$ 
and that $\CT_B^\enh(\gamma^L(\G)) \simeq \CT_B^\enh \Eis_B^\enh(\M)$. 
In other words, $\G$ is determined by an anti-tempered object of $I(SL_2,B)$, and it remains to show that $\CT_B^\enh \Eis_B^\enh$ is conservative on $I(SL_2,B)^\antitemp$.
It suffices to show that $\Eis_B^\enh$ is conservative on $I(SL_2,B)^\antitemp$: indeed, we already know that $\CT_B^\enh$ is fully faithful when restricted to $\Dmod(\Bun_G)^\antitemp$.

\sssec{}

To this end, we use Theorem \ref{mainthm:rank1} which describes $\Dmod(\Bun_G)^\antitemp$ and $ I(G,B)^\antitemp$. The former is generated by $\omega_{\Bun_G}$ under colimits, the latter by $\omega_{\Bun_G^{B\ggen}}$.

It remains to notice that 
$$
\CT_B^\enh(\omega_{\Bun_G}) 
\simeq 
\omega_{\Bun_G^{B\ggen}}
$$
$$
\Eis_B^\enh
\bigt{ 
\omega_{\Bun_G^{B\ggen}}
}
\simeq
\omega_{\Bun_G}.
$$
The first assertion is obvious from the definitions, the second one follows from the contractibility of the fibers of the projection $\Bun_G^{B\ggen} \to \Bun_G$.
In particular, this proves:
\begin{cor}
The adjoint functors 
\begin{equation}
\nonumber
\begin{tikzpicture}[scale=1.5]
\node (a) at (0,1) {$\Eis_B^{\enh}: I(SL_2, B)^\antitemp $};
\node (b) at (3.5,1) {$\Dmod(\Bun_{SL_2})^\antitemp:
\CT_B^{\enh}$};
\path[->,font=\scriptsize,>=angle 90]
([yshift= 1.5pt]a.east) edge node[above] {$ $} ([yshift= 1.5pt]b.west);
\path[->,font=\scriptsize,>=angle 90]
([yshift= -1.5pt]b.west) edge node[below] {$ $} ([yshift= -1.5pt]a.east);
\end{tikzpicture}
\end{equation}
are mutually inverse equivalences of DG categories.
\end{cor}

\sec{Geometric Langlands for $G=SL_2$: a preview} \label{sec:preview}

In \cite{Outline}, Gaitsgory outlined a proof of the geometric Langlands conjecture for $G=GL_2$. Applying the same strategy to the case of $G=SL_2$ is problematic: in view of the disconnectedness of $Z(SL_2)$, the extended Whittaker category is nontrivial to define (see \cite[Section 9]{Outline} and \cite{ext-whit}), and consequently it is difficult to match it with the spectral side. 

In the present informal\footnote{A rigorous treatment would require much more space and will appear elsewhere.} section, we propose a few modifications to the outline \cite{Outline} that would allow to solve the case of $G=SL_2$. We count on the reader's familiarity with the general methods and ideas explained in \cite[Introduction]{Outline}. The main new ingredient is our automorphic gluing theorem, Theorem \ref{mainthm:aut-gluing}: this is designed to match the spectral gluing theorem exactly, thereby circumventing the issues with the extended Whittaker category. We combine this with two more ideas, explained below:
\begin{itemize}
\item
the fully faithfulness of the spectral Deligne-Lusztig duality when restricted to compact objects;
\item
the relation between the Steinberg D-module $\St_G \in \Dmod(\LS_G)$ and opers.
\end{itemize}
The first of these was proven in \cite{DL}, the second one is almost obvious for $G=SL_2$ (but altogether nontrivial for general $G$). Thus, summarizing the contents of this paper, we may say that the geometric Ramanujan conjecture, together with the Deligne-Lusztig dualities of \cite{DL} and \cite{LinChen}, suffices to clear the case of $SL_2$.

\ssec{Constructing the functor}

Let $G=SL_2$, so that $\Gch = PGL_2$. Let us recall the statements of the automorphic and spectral gluing theorems:
$$
\Dmod(\Bun_{G})
\simeq
\Dmod(\Bun_{G})^\temp
\utimes_{ I(G, B) ^\temp}
I(G, B),
$$
$$
\ICoh_{\Nch}(\LS_{\Gch})
\simeq
\QCoh(\LS_{\Gch})
\utimes_{
\QCoh \bigt{ (\LS_{\Gch})^\wedge_ {\LS_{\check{B}}} }
}
\ICoh_0 \bigt{ (\LS_{\Gch})^\wedge_{\LS_{\check{B}}} }.
$$
The main idea is of course to define the functor
$$
\bbL_G: 
\ICoh_{\Nch}(\LS_{\Gch})
\longto
\Dmod(\Bun_G)
$$
component-wise.

\sssec{}

The functor
$$
\LL_{G,B}:
\ICoh_0 \bigt{ (\LS_{\Gch})^\wedge_{\LS_{\check{B}}} }
\longto
I(G,B)
$$
has been defined by Raskin in his thesis and in subsequent works (see \cite{Sam-thesis} and \cite{Sam2}).
The functor
$$
\LL_{G,B}^\temp:
\QCoh \bigt{ (\LS_{\Gch})^\wedge_{\LS_{\check{B}}} }
\longto
I(G,B)^\temp
$$
is obtained by the above one (which is $\Sph_G$-linear) by restricting to the tempered subcategories. These functors ought to be equivalences.
Finally, let us come to the functor
$$
\LL_{G}^\temp:
\QCoh(\LS_{\Gch})
\longto
\Dmod(\Bun_G)^\temp.
$$
This is defined exactly as in \cite{Outline}, using the  so-called \emph{vanishing theorem}, that is, the action of $\QCoh(\LS_{\Gch})$ on $\Dmod(\Bun_G)^\temp$. Namely, $\LL_{G}^\temp$ is defined as the $\QCoh(\LS_{\Gch})$-action on a particular tempered D-module $\Poinc_! \in \Dmod(\Bun_G)$. 

\sssec{}

Next, one should prove that these three functors are compatible with gluing, and therefore they assemble to yield a functor $\LL_G: \ICoh_{\Nch}(\LS_{\Gch})
\longto
\Dmod(\Bun_G)$.

\sssec{}

Once $\LL_G$ is constructed, we need to show it is an equivalence. The method of \cite{Outline} for essential surjectivity should go through, so it remains to prove that $\LL_G$ is fully faithful. The fully faithfulness of $\bbL_{G,B}$ (and consequently of $\bbL_{G,B}^\temp$) are being taken care of by the above mentioned works of Raskin. Hence, it remains to prove:

\begin{conj} \label{conj:fully faithfulness}
The functor
$$
\LL_{G}^\temp:
\QCoh(\LS_{\Gch})
\longto
\Dmod(\Bun_G)^\temp
$$
is fully faithful.
\end{conj}

\sssec{}

We move towards the proof of this conjecture by explaining how the Hom spaces of $\QCoh(\LS_{\Gch})$ can be calculated in automorphic terms. Since $\QCoh(\LS_{\Gch})$ is compactly generated by perfect objects and since perfect objects are dualizable, it suffices to focus on
$$
\CHom_{\QCoh(\LS_{\Gch})}(\F, \O_{\LS_{\Gch}})
$$
for $\F \in \Perf(\LS_{\Gch})$.

\ssec{A digression on the Steinberg object}

\sssec{}

Now recall the spectral Deligne-Lusztig duality functor $\DL_\Gch$, introduced in \cite{DL}. This is the functor
\begin{equation} \label{eqn:spec-DL}
\DL_\Gch:
\ICoh_{\Nch}(\LSGch)
\xto{\Psi_{0 \to \Nch}}
\QCoh(\LSGch)
\xto{ \ul\St_{\Gch} \otimes -}
\QCoh(\LSGch)
\xto{\Xi_{0 \to \Nch}}
\ICoh_{\Nch}(\LSGch),
\end{equation}
where $\St_\Gch \in \Dmod(\LSGch)$ is the \emph{Steinberg D-module} (defined immediately below) and $\ul\St_\Gch \in \QCoh(\LSGch)$ its underlying quasi-coherent sheaf.
In \cite{DL}, we argued that, under Geometric Langlands, $\DL_{\Gch}$ ought to correspond to the Deligne-Lusztig functor $\DL_G$ of Section \ref{ssec:DL-definition}. Accordingly, the definition of $\St_{\Gch}$ imitates the definition of $\DL_G$: namely,
$$
\St_\Gch
:=
\coker
\Bigt{
\colim_{P \in \Par'}
(\fp_{\Pch})_!(\omega_{\LSPch})
\longto
\omega_{\LSGch}
}
\in \Dmod(\LSGch).
$$
For instance, when $\Gch = PGL_2$, this simplifies as $\St_{\Gch} \simeq \coker ((\fp_{\check{B}})_!(\omega_{\LS_{\check{B}}})
\longto
\omega_{\LSGch})$.

\sssec{}

In spite of the long definition, $\St_\Gch$ is remarkable in at least three respects.
\begin{itemize}
\item
The $\kk$-fibers of $\St_\Gch$ are either $1$-dimensional or zero, depending on whether the $\Gch$-local system in question is semi-simple or not. The full formulas appeared in \cite[Theorem D']{DL}; here we just give the three main examples in case $G=SL_2$:
\begin{equation} \label{eqn:cases-St}
\restr{\St_{\Gch}}{\sigma}
\simeq
\begin{cases}
\kk & \mbox{ if $\sigma$ is irreducible}; \\
\kk[3] & \mbox{ if $\sigma$ is trivial}; \\
0 & \mbox{ if $\sigma$ is not semi-simple}.
\end{cases}
\end{equation}

\item

Even though the first two arrows of the composition \eqref{eqn:spec-DL} are not at all fully faithful, the functor $\DL_{\Gch}$ is fully faithful when restricted to $\Coh_{\Nch}(\LSGch)$. This is \cite[Theorem E]{DL}. In particular, the functor 
\begin{equation} \label{eqn:tensor-by-St}
\Perf(\LSGch) \xto{\ul\St_{\Gch} \otimes -} \QCoh(\LSGch)
\end{equation}
is fully faithful.

\item

Finally, $\St_\Gch$ is intimately connected to the space of \emph{$\Gch$-opers}. This connection (discussed below in Section \ref{ssec:Steinberg}) is work-in-progress for general $G$, but it is almost obvious for $G$ of rank $1$. 
\end{itemize}

\ssec{Steinberg and opers} \label{ssec:Steinberg}

\sssec{}

Let us elaborate on the third item above. Consider the prestack $\LS_{\Gch}^{\Op\ggen}$ of \emph{$\Gch$-local systems on $X$ equipped with a generic oper structure}. We do not define this prestack in general, but point out that, for $\Gch = PGL_2$, an $S$-point of $\OpUnrGlob$ is:
\begin{itemize}
\item
an $S$-point $(E,\nabla)$ of $\LSGch$;
\item
a generic $\check{B}$-reduction of $E$, in the sense of \cite{Barlev}, which is not preserved by $\nabla$.
\end{itemize}
Denote by $\pi: \OpUnrGlob \to \LSGch$ the map that forgets the generic oper structure. Similarly to the case of the prestack $Y[X]^\gen$ of generic maps from $X$ to a scheme $Y$, the DG category $\Dmod(\OpUnrGlob)$ is well-defined and equipped with a renormalized push-forward functor 
$$
\pi_{*,\ren}: 
\Dmod(\OpUnrGlob)
\longto
\Dmod(\LSGch).
$$
Now we claim:

\begin{thm}
For $\Gch = PGL_2$, there is a natural isomorphism
$$
\pi_{*,\ren}
\bigt{
\omega_{\OpUnrGlob}
}
\simeq
\St_\Gch.
$$
\end{thm}

\begin{rem}
We expect this result to be true for all groups $\Gch$ (it is work in progress at the moment): see \cite{Arinkin-opers}, and \cite{opers} for a related result for classical groups.
\end{rem}

\sssec{}

Rather than giving a construction and a proof of this statement (these will appear elsewhere), we find it more instructive to look at three examples that match the three cases of \eqref{eqn:cases-St}.
Let $\sigma$ be a $\kk$-point of $\LSGch$ and set 
$$
\Op^{\gen}_{G,\sigma}
:=
\OpUnrGlob \utimes_{\LSGch}\sigma.
$$
One of the features of the renormalized push-forward is base-change: the pullback of $\pi_{*,\ren}
\bigt{
\omega_{\OpUnrGlob}
}$ along $\sigma: \Spec(\kk) \to \LSGch$ computes the Borel-Moore homology of $\Op^{\gen}_{G,\sigma}$. So, let us check that such Borel-Moore homology matches the results of \eqref{eqn:cases-St} in the three cases of $\sigma$ irreducible, trivial and not semi-simple.

\sssec{}

It follows from the definition that (at the reduced level, which is enough for the computations of Borel-Moore homology)
\begin{equation} \label{eqn:cases-St-opers}
\Op^{\gen}_{G,\sigma}
\simeq
\begin{cases}
\PP^1[X]^\gen & \mbox{ if $\sigma$ is irreducible}; \\
\PP^1[X]^\gen - \bbP^1 & \mbox{ if $\sigma$ is trivial}; \\
\AA^1[X]^\gen & \mbox{ if $\sigma$ is not semi-simple}.
\end{cases}
\end{equation}
Indeed, if $\sigma$ is irreducible, no $\check{B}$-reduction of the underlying $\Gch$-bundle can be preserved by the connection. Similarly, if $\sigma$ is not semi-simple and $\check{B}$-reduced, no other $\check{B}$-reduction of the underlying $\Gch$-bundle can be preserved by the connection: otherwise these two $\check{B}$-reductions would yield a $\Tch$-reduction of $\sigma$, contradicting semi-simplicity.

\sssec{}
Now, $\Hren(\PP^1[X]^\gen) \simeq H_*(\PP^1[X]^\gen) \simeq \kk$: this is the  homological contractibility of the space of rational maps into a flag variety, proven in \cite{Barlev} building up on \cite{contract}.
We deduce that
$$
\Hren(\PP^1[X]^\gen - \bbP^1)
\simeq
\coker
\Bigt{
H_*(\PP^1)
\longto
H_*(\PP^1[X]^\gen)
}
\simeq
\coker(\kk \oplus \kk[2] \to \kk)
\simeq
\kk[3]
$$
and that
$$
\Hren(\AA^1[X]^\gen )
\simeq
\Hren(\PP^1[X]^\gen - \pt)
\simeq
\coker
\Bigt{
H_*(\pt)
\longto
H_*(\PP^1[X]^\gen)
}
\simeq
\coker(\kk \to \kk)
\simeq
0,
$$
as desired.

\ssec{Fully faithfulness}

Let us now return to Conjecture \ref{conj:fully faithfulness} and in particular to the study of
$$
\CHom_{\QCoh(\LS_{\Gch})}(\F, \O_{\LS_{\Gch}})
$$
for $\F \in \Perf(\LS_{\Gch})$.

\sssec{}

By the fully faithfulness of \eqref{eqn:tensor-by-St}, we see that
$$
\CHom_{\QCoh(\LS_{\Gch})}(\F, \O_{\LS_{\Gch}})
\simeq
\CHom_{\QCoh(\LS_{\Gch})}
(
\F \otimes \ul\St_\Gch
,
 \ul\St_\Gch
 )
$$
and we wish to compute the latter in automorphic terms. The object $\F \otimes \ul\St_\Gch$ is, like any other object of $\QCoh(\LS_{\Gch})$, generated under colimits by perfect objects, so it suffices to study the Hom spaces
$$
\CHom_{\QCoh(\LS_{\Gch})}
(
\P
,
 \ul\St_\Gch
 )
 \hspace{.4cm}
 \mbox{for all $\P \in \Perf(\LSGch)$}.
$$

\sssec{}

Now, it is known by the so-called \emph{localization principle}, see \cite{vanishing}, that a collection of perfect generators of $\QCoh(\LSGch)$ is given by objects coming from $\Rep(\Gch)_{\Ran}$. Taking one of these for $\P$, and expressing $\ul\St_{\Gch}$ via generic opers as above, brings the computation of
$$
\CHom_{\QCoh(\LS_{\Gch})}
(
\P
,
 \ul\St_\Gch
 )
 $$
to (a variant of) the main commutative diagram of \cite[Corollary 10.4.5]{Outline}, called \emph{fundamental commutative diagram} there.


\begin{thebibliography}{99}

\bibitem{Arinkin-opers} D.~Arinkin, Irreducible connections admit generic oper structures. Arxiv:1602.08989.


\bibitem{AG1} D.~Arinkin, D.~Gaitsgory, {Singular support of coherent sheaves and the geometric Langlands conjecture}. {\newblock Selecta Math. (N.S.) 21 (2015), no. 1, 1-199.}

\bibitem{AG2} D.~Arinkin, D.~Gaitsgory, {The category of singularities as a crystal and global Springer fibers}.
{\newblock J. Amer. Math. Soc. 31 (2018), no. 1, 135-214.}

\bibitem{BL}
A. Beauville et Y. Laszlo. Un lemme de descente. 
Comptes Rendus Acad. Sci. Paris, vol 320, 335-340, 1995.

\bibitem{Barlev} J. Barlev. {D-modules on spaces of rational maps}, 
\newblock Compositio Math. 150 (2014), 835-876.


\bibitem{BD-quantization} A. Beilinson, V. Drinfeld, 
{Quantization of Hitchin's integrable system and Hecke eigensheaves}




\bibitem{thesis} D. Beraldo. {Loop group actions on categories and Whittaker invariants.} Advances in Mathematics, 322 (2017) 565-636.

\bibitem{ext-whit} D. Beraldo. {On the extended Whittaker category}, Sel. Math. New Ser. (2019) 25: 28.



\bibitem{shvcatHH} D. Beraldo, {Sheaves of categories with local actions of Hochschild cochains},
\newblock Compositio Mathematica, 155(8), 1521-1567.

\bibitem{omega} D. Beraldo, {Tempered D-modules and Borel-Moore homology vanishing},
\newblock Inventiones mathematicae 225, pp. 453-528. 




\bibitem{epiga} D. Beraldo. {The spectral gluing theorem revisited},
\newblock EPIGA, Volume 4 (2020), Article Nr. 9.

\bibitem{DL} D. Beraldo. {Deligne-Lusztig duality on the stack of local systems}, \newblock J. reine angew. Math. 778 (2021), 31-63.


\bibitem{BF} R. Bezrukavnikov, M. Finkelberg.
\newblock Equivariant Satake category and Kostant-Whittaker reduction.
\newblock Mosc. Math. J., 2008, Volume 8, Number 1,	Pages 39-72


\bibitem{LinChen} L. Chen. Deligne-Lusztig duality on the moduli stack of bundles. \newblock ArXiv: 2008.09348.

\bibitem{DG-cptgen} V. Drinfeld, D. Gaitsgory, 
Compact generation of the category of D-modules on the stack of G-bundles on a curve.
Cambridge Journal of Mathematics 2015 (3), 19-125.

\bibitem{CT} V. Drinfeld and D. Gaitsgory, Goemetric constant term functors.

\bibitem{finiteness} V. Drinfeld and D. Gaitsgory, On some finiteness questions for algebraic stacks, GAFA 23 (2013),149-294.

\bibitem{DS} V. Drinfeld and C. Simpson, $B$-structures on $G$-bundles and local triviality,
Mathematical Research Letters 2, 823-829 (1995)

\bibitem{DW}  V. Drinfeld and J. Wang, On a strange invariant bilinear form on the space of automorphic forms, arXiv:1503.04705.

\bibitem{indep}
J. Faergeman, S. Raskin. The Arinkin-Gaitsgory temperedness conjecture, arxiv:2108.02719.


\bibitem{FGV} E. Frenkel, D. Gaitsgory and K. Vilonen, Whittaker patterns in the geometry of moduli spaces of bundles on curves, Annals of Math. 153 (2001), no. 3, 699-748.



\bibitem{vanishing} D. Gaitsgory, A generalized vanishing theorem, available at
https://people.math.harvard.edu/~gaitsgde/GL/GenVan.pdf


\bibitem{contract} D.~Gaitsgory, {Contractibility of the space of rational maps.} Invent. math. 191, 91-196 (2013).


\bibitem{ker-adj} D. Gaitsgory, Functors given by kernels, adjunctions and duality, Journal of Algebraic Geometry 25 (2016), 461-548.


\bibitem{ICoh} D.~Gaitsgory, {Ind-coherent sheaves.}
Mosc. Math. J. 13 (2013), no. 3, 399-528, 553. 

\bibitem{miraculous} D. Gaitsgory, A strange functional equation for Eisenstein series and Verdier duality on the moduli stack of bundles, arXiv:1404.6780.

\bibitem{Outline} D. Gaitsgory, { Outline of the proof of the geometric Langlands conjecture for $GL_2$}. Asterisque.



\bibitem{Book} D.~Gaitsgory, N. Rozenblyum, {\it Studies in derived algebraic geometry}. Mathematical Surveys and Monographs, 221. American Mathematical Society, Providence, RI, 2017.


\bibitem{Crystals} D.~Gaitsgory, N. Rozenblyum, {\it Crystals and D-modules}. 





\bibitem{opers}
D. Kazhdan, T. Schlank. Contractibility of the Space of Opers for Classical Groups. Arxiv:1801.00655.



\bibitem{VLaff} V. Lafforgue, {\it Quelques calculs reli\'{e}s \`{a} la correspondance de Langlands g\'{e}om\'{e}trique pour $\PP^1$.} \\
\newblock Available at http://vlafforg.perso.math.cnrs.fr/files/geom.pdf.





\bibitem{HA} J.~Lurie, {\it Higher algebra}.
\newblock Available at \url{http://www.math.harvard.edu/~lurie}.

\bibitem{MV} I. Mirkovic, K. Vilonen,
\newblock {Geometric Langlands duality and
representations of algebraic groups
over commutative rings.}
\newblock Annals of Mathematics, 166 (2007), 95-143.







\bibitem{Sam}
S. Raskin. W-algebras and Whittaker categories. Arxiv:1611.04937.

\bibitem{Sam-thesis}
S. Raskin. Chiral Principal Series Categories. Harvard thesis, 2014.

\bibitem{Sam2}
S. Raskin. Chiral principal series categories I: finite-dimensional calculations.


\bibitem{Wang} J. Wang. 
On an invariant bilinear form on the space of automorphic forms via asymptotics. 
Duke Mathematical Journal, 167(16):2965-3057 (2018).

\bibitem{Winter} 
Winter school on local geometric Langlands theory, Jan 2018.
\url{https://sites.google.com/site/winterlanglands2018/notes-of-talks}

\bibitem{Xinwen}
X. Zhu, An introduction to affine Grassmannians and the geometric
Satake equivalence. IAS/Park City Mathematics Series.




\end{thebibliography}
\end{document}